\documentclass{article}%
\usepackage{amsmath}
\usepackage{graphicx}
\usepackage{amsfonts}
\usepackage{amssymb}%
\providecommand{\U}[1]{\protect\rule{.1in}{.1in}}
\newtheorem{theorem}{Theorem}

\newtheorem{corollary}{Corollary}

\newtheorem{lemma}{Lemma}

\newtheorem{proposition}{Proposition}
\newtheorem{remark}{Remark}

\newenvironment{proof}[1][Proof]{\textbf{#1.} }{\ \rule{1em}{1em}}
\begin{document}

\title{Small BGK waves and nonlinear Landau damping}
\author{Zhiwu Lin and Chongchun Zeng\\School of Mathematics\\Georgia Institute of Technology\\Atlanta, GA 30332, USA}
\date{}
\maketitle

\begin{abstract}
Consider $1$D Vlasov-poisson system with a fixed ion background and periodic
condition on the space variable. First, we show that for general homogeneous
equilibria, within any small neighborhood in the Sobolev space $W^{s,p}%
\ \left(  p>1,s<1+\frac{1}{p}\right)  \ $of the steady distribution function,
there exist nontrivial travelling wave solutions (BGK waves) with arbitrary
minimal period and traveling speed. This implies that nonlinear Landau damping
is not true in $W^{s,p}\left(  s<1+\frac{1}{p}\right)  \ $space for any
homogeneous equilibria and any spatial period. Indeed, in $W^{s,p}$ $\left(
s<1+\frac{1}{p}\right)  \ $neighborhood of any homogeneous state, the long
time dynamics is very rich, including travelling BGK waves, unstable
homogeneous states and their possible invariant manifolds. Second, it is shown
that for homogeneous equilibria satisfying Penrose's linear stability
condition, there exist no nontrivial travelling BGK waves and unstable
homogeneous states in some $W^{s,p}$ $\left(  p>1,s>1+\frac{1}{p}\right)
\ $neighborhood. Furthermore, when $p=2,$we prove that there exist no
nontrivial invariant structures in the $H^{s}$ $\left(  s>\frac{3}{2}\right)
$ neighborhood$~$of stable homogeneous states. These results suggest the long
time dynamics in the $W^{s,p}$ $\left(  s>1+\frac{1}{p}\right)  $ and
particularly, in the $H^{s}$ $\left(  s>\frac{3}{2}\right)  $ neighborhoods of
a stable homogeneous state might be relatively simple. We also demonstrate
that linear damping holds for initial perturbations in very rough spaces, for
linearly stable homogeneous state. This suggests that the contrasting dynamics
in $W^{s,p}$ spaces with the critical power $s=1+\frac{1}{p}\ $is a trully
nonlinear phenomena which can not be traced back to the linear level.

\end{abstract}

\section{Introduction}

Consider a one-dimensional collisionless electron plasma with a fixed
homogeneous neutralizing ion background. The fixed ion background is a good
physical approximation since the motion of ions is much slower than electrons.
But we consider fixed ion mainly to simplify notations and the main results in
this paper are also true for electrostatic plasmas with two or more species.
The time evolution of such electron plasmas can be modeled by the
Vlasov-Poisson system
\begin{subequations}
\label{vpe}
\begin{equation}
\frac{\partial f}{\partial t}+v\frac{\partial f}{\partial x}-E\frac{\partial
f}{\partial v}=0,\label{vlasov}%
\end{equation}
\begin{equation}
\,\,\,\,\ \frac{\partial E}{\partial x}=-\int_{-\infty}^{+\infty
}fdv+1,\label{poisson}%
\end{equation}
\end{subequations}
where $f(x,v,t)$ is the electron distribution function, $E\left(  x,t\right)
\,\ $the electric field, and $1$ is the ion density. The one-dimensional
assumption is proper for a high temperature and dilute plasma immersed in a
constant magnetic field oriented in the $x$-direction. For example, recent
discovery by satellites of electrostatic structures near geomagnetic fields
can be justified by using such Vlasov-Poisson models (\cite{krasovsky-et-04},
\cite{muschietti-et-99}). We assume: 1) $f\left(  x,v,t\right)  \geq0$ and
$E\left(  x,t\right)  $ are $T-$periodic in $x$. 2) Neutral condition:
$\int_{0}^{T}\int_{\mathbf{R}}f\left(  x,v,0\right)  dxdv=T$. 3) $\int_{0}%
^{T}E\left(  x,t\right)  dx=0,\ $so $E\left(  x,t\right)  =-\partial_{x}%
\phi\left(  x,t\right)  $, where the electric potential $\phi\left(
x,t\right)  $ is $T-$periodic in $x$. Since $\int\int f\left(  x,v,t\right)
\ dxdv$ is an invariant, the neutral condition 2) is preserved for all time.
The condition 3) ensures that $E\left(  t\right)  $ is determined uniquely by
$f\left(  t\right)  $ from (1b) and the system (\ref{vpe}) can be considered
to be an evolution equation of $f\ $only. It is shown in
\cite{klimas-cooper83} that with condition 3), the system (\ref{vpe}) is
equivalent to the following one-dimensional Vlasov-Maxwell system%

\[
\frac{\partial f}{\partial t}+v\frac{\partial f}{\partial x}-E\frac{\partial
f}{\partial v}=0,
\]%
\[
\frac{\partial E}{\partial x}=-\int_{-\infty}^{+\infty}fdv+1,
\]%
\[
\frac{\partial E}{\partial t}=\int_{\mathbf{R}}vf\left(  x,v,t\right)  dv-U,
\]
where $U$ is the bulk velocity of the ion background. The system (\ref{vpe})
is non-dissipative and time-reversible. It has infinitely many equilibria,
including the homogeneous states $\left(  f_{0}\left(  v\right)  ,0\right)  $
where $f_{0}\left(  v\right)  $ is any nonnegative function satisfying
$\int_{\mathbf{R}}f_{0}\left(  v\right)  dv=1$.

In 1946, Landau \cite{landau}, looking for analytical solutions of the
linearized Vlasov-Poisson system around Maxwellian $\left(  e^{-\frac{1}%
{2}v^{2}},0\right)  $, pointed out that the electric field is subject to time
decay even in the absence of collisions. The effect of this Landau damping, as
it is subsequently called, plays a fundamental role in the study of plasma
physics. However, Landau's treatment is in the linear regime; that is, only
for infinitesimally small initial perturbations. Despite many numerical,
theoretical and experimental efforts, no rigorous justification of the Landau
damping has been given in a nonlinear dynamical sense. In the past decade,
there has been renewed interest \cite{ishenko} \cite{manfredi97}
\cite{lancellotti-dorning} \cite{glassey-schaeffer} \cite{glassey-schaeffer-2}
\cite{zhou-et-98} \cite{caglioti} \cite{hwang} \cite{villani09}
\cite{valentini-et-05} as well as controversy about the Landau damping. In
\cite{caglioti} \cite{hwang}, it was shown that there exist certain analytical
perturbations for which electric fields decay exponentially in the nonlinear
level. More recently, in \cite{villani09} nonlinear Landau damping was shown
for general analytical perturbations of stable equilibria with linear
exponential decay. For non-analytic perturbations, the linear decay rate of
electric fields is known to be only algebraic (i.e. \cite{weizner63}) and the
nonlinear damping is more difficult to justify if it is true. Moreover, in the
nonlinear regime, it has been know (\cite{oneil65}) that the damping can be
prevented by particles trapped in the potential well of the wave. Such
particle trapping effect is ignored in Landau's linearized analysis as well as
other physically equivalent linear theories (\cite{case60}, \cite{van-kampen}%
), which assume that the small amplitude of waves have a negligible effect on
the evolution of distribution functions. As early as in 1949, Bohm and Gross
(\cite{bohn-gross}) already recognized the importance of particle trapping
effects and the possibility of nonlinear travelling waves of small but
constant amplitude. In 1957, Bernstein, Greene and Kruskal (\cite{bgk})
formalized the ideas of Bohm and Gross and found a general class of exact
nonlinear steady imhomogeneous solutions of the Vlasov-Poisson system. Since
then, such steady solutions have been known as BGK modes, BGK waves or BGK
equilibria. The nontrivial steady waves of this type are made possible by the
existence of particles trapped forever within the electrostatic potential
wells of the wave. The existence of such undamped waves in any small
neighborhood of an equilibrium will certainly imply that nonlinear damping is
not true.

Furthermore, numerical simulations \cite{manfredi97} \cite{deimo-zweifel}
\cite{rosenbluth-et-98} \cite{driscoll-et-04} \cite{brunetti-et-00} indicate
that for certain small initial data near a stable homogeneous state including
Maxwellian, there is no decay of electric fields and the asymptotic state is a
BGK wave or superposition of BGK waves which were formally constructed in
\cite{buchanan-dorning-mBGK}. Moreover, BGK waves also appear as the
asymptotic states for the saturation of an unstable homogeneous state
(\cite{amstrong-monto-67}). These suggest that small BGK waves play important
role in understanding the long time behaviors of Vlasov-Poisson system, near
homogeneous equilibria. In this paper, we provide a sharp characterization of
the Sobolev spaces in which small BGK waves exist in any small neighborhood of
a homogeneous equilibrium. Denote the fractional order Sobolev spaces by
$W^{s,p}\left(  \mathbf{R}\right)  ~$or $W_{x,v}^{s,p}\left(  \left(
0,T\right)  \times\mathbf{R}\right)  $ with $p\geq1,s\geq0$. These spaces are
the interpolation spaces (see \cite{adams}, \cite{triebel}) of $L^{p}$ space
and Sobolev space $W^{m,p}$ $\left(  m\ \text{positive integer}\right)  $.

\begin{theorem}
\label{thm-existence}Assume the homogeneous distribution function
$f_{0}\left(  v\right)  \in W^{s,p}\left(  \mathbf{R}\right)  $ $\left(  p>1,s
\in[0, 1+\frac1p)\right)  $ satisfies%
\[
f_{0}\left(  v\right)  \geq0,\ \int f_{0}\left(  v\right)  dv=1,\ \int
v^{2}f_{0}\left(  v\right)  <+\infty.
\]
Fix $
T>0$ and $c\in\mathbf{R}$. Then for any $\varepsilon>0$, there exist
travelling BGK wave solutions of the form $\left(  f_{\varepsilon}\left(
x-ct,v\right)  ,E_{\varepsilon}\left(  x-ct\right)  \right)  $ to (\ref{vpe}),
such that$\ \left(  f_{\varepsilon}\left(  x,v\right)  ,E_{\varepsilon}\left(
x\right)  \right)  $ has minimal period $T$ in $x$,$\ f_{\varepsilon}\left(
x,v\right)  \geq0,$ $E_{\varepsilon}\left(  x\right)  $ is not identically
zero, and
\begin{equation}
\ \left\Vert f_{\varepsilon}-f_{0}\right\Vert _{L_{x,v}^{1}}+\ \int_{0}%
^{T}\int_{\mathbf{R}}v^{2}\left\vert f_{\varepsilon}\left(  x,v\right)
-\ f_{0}\left(  v\right)  \right\vert dxdv+\left\Vert f_{\varepsilon}%
-f_{0}\right\Vert _{W_{x,v}^{s,p}}<\varepsilon.\ \label{norm-thm-stability}%
\end{equation}

\end{theorem}

The first two terms in (\ref{norm-thm-stability}) imply that the BGK wave is
close to the homogeneous state $\left(  f_{0},0\right)  \ $in the norms of
total mass and energy. When $p>1,s=1,$ the fractional Sobolev space is
equivalent to the usual Sobolev space $W^{1,p}$. The conclusions in Theorem
\ref{thm-existence} are also true for the Sobolev space $W_{x,v}^{1,1}$ by the
same proof. Above theorem immediately implies that nonlinear Landau damping is
\textit{not true} for perturbations in any $W^{s,p}$ $\left(  s<1+\frac{1}%
{p}\right)  \ $space, for any homogeneous equilibrium in
$W^{s,p}\ $ and any spatial period.

As a corollary of the proof, we show that there exist unstable homogeneous
states in $W^{s,p}\left(  \mathbf{R}\right)  $ $\left(  s<1+\frac{1}%
{p}\right)  \ $neighborhood of any homogeneous equilibrium.

\begin{corollary}
\label{cor-unsatble}Under the assumption of Theorem \ref{thm-existence}, for
any fixed $T>0$, $\ \exists$ $\varepsilon_{0}>0$, such that for any
$0<\varepsilon<\varepsilon_{0}$, there exists a homogeneous state $\left(
f_{\varepsilon}\left(  v\right)  ,0\right)  \ $which is linearly unstable
under perturbations of $x-$period $T$,
\[
f_{\varepsilon}\left(  v\right)  \geq0,\ \int_{\mathbf{R}}f_{\varepsilon
}\left(  v\right)  \ dv=1,
\]
and
\[
\left\Vert f_{\varepsilon}\left(  v\right)  -f_{0}\left(  v\right)
\right\Vert _{L^{1}\left(  \mathbf{R}\right)  }+\int_{\mathbf{R}}%
v^{2}\left\vert f_{\varepsilon}\left(  v\right)  -\ f_{0}\left(  v\right)
\right\vert dv+\left\Vert f_{\varepsilon}\left(  v\right)  -f_{0}\left(
v\right)  \right\Vert _{W^{s,p}\left(  \mathbf{R}\right)  }<\varepsilon.\
\]

\end{corollary}

By above Corollary and Remark \ref{rmk-double-period} following the proof of
Theorem \ref{thm-existence}, in $W_{x,v}^{s,p}\ \left(  s<1+\frac{1}%
{p}\right)  \ $neighborhood of any homogeneous state there exist lots of
unstable homogeneous states and unstable nontrivial BGK waves. In a work in
progress, we are constructing stable and unstable manifolds near an unstable
equilibrium of Vlasov-Poisson system by extending our work
(\cite{lin-zeng-euler-invar}) on invariant manifolds of Euler equations. Such
(possible) invariant manifolds might reveal more complicated global invariant
structures such as heteroclinic or homoclinic orbits. Moreover, in some
physical reference (\cite{buchanan-dorning-mBGK}), small BGK waves are
formally shown to follow a nonlinear superposition principle to form
time-periodic or quasi-periodic orbits. We note that Maxwellian or any
homogeneous equilibria $f_{0}\left(  v\right)  =\mu\left(  \frac{1}{2}%
v^{2}\right)  $ with $\mu$ monotonically decreasing, were shown by Newcomb in
1950s (see Appendix I, pp. 20-21 of \cite{bernstein58-newcomb}) to be
nonlinearly stable in the norm $\left\Vert f\right\Vert _{L^{2}}$. So our
result suggests, in particular, that in any invariant small $L^{2}$
neighborhood of Maxwellian, the long time dynamical behaviors are very rich.

The following Theorem shows that there exist no nontrivial BGK waves near a
stable homogeneous state in $W_{x,v}^{s,p}$ space when $p>1,\ s>1+\frac{1}%
{p}.$

\begin{theorem}
\label{thm-non-existence} Assume $f_{0}\left(  v\right)  \in W^{s,p}\left(
\mathbf{R}\right)  $ $\left(  p>1,s>1+\frac{1}{p}\right)  .\ $Let $S=\left\{
v_{i}\right\}  _{i=1}^{l}$ be the set of all extrema points of $f_{0}.\ $Let
$0<T_{0}\leq+\infty$ be defined by
\begin{equation}
\left(  \frac{2\pi}{T_{0}}\right)  ^{2}=\max\left\{  0,\max_{v_{i}\in S}%
\int\frac{f_{0}^{\prime}\left(  v\right)  }{v-v_{i}}dv\right\}  .
\label{defn-T-0}%
\end{equation}
Then for any $T<T_{0}$, $\exists$ $\varepsilon_{0}\left(  T\right)  >0$, such
that there exist no nontrivial travelling wave solutions $\left(  f\left(
x-ct,v\right)  ,E\left(  x-ct\right)  \right)  \ $to (\ref{vpe}) for any
$c\in\mathbf{R}$, satisfying that$\ \left(  f\left(  x,v\right)  ,E\left(
x\right)  \right)  $ has period $T$ in $x$,$\ E\left(  x\right)  $ not
identically $0$,
\[
\int_{0}^{T}\int_{\mathbf{R}}v^{2}f\left(  x,v\right)  dvdx<\infty,\text{
(assumption of finite energy)}%
\]
$\ $and$\ \left\Vert f-f_{0}\right\Vert _{W_{x,v}^{s,p}}<\varepsilon_{0}.$
\end{theorem}

By Penrose's stability criterion (\cite{penrose} or Lemma \ref{lemma-penrose})
the homogeneous equilibrium $\left(  f_{0}\left(  v\right)  ,0\right)  $ is
linearly stable to perturbations of $x-$period $T<T_{0}$.\ Moreover, in
Proposition \ref{prop-linear}, the linear damping of electrical field is shown
for such stable states in a rough function space. Theorems \ref{thm-existence}
and \ref{thm-non-existence} imply that for any $p>1,\ s=1+\frac{1}{p}$ is the
critical index for existence or non-existence of small BGK waves in
$W^{s,p}\ $neighborhood of a stable homogeneous state. In Lemma
\ref{lemma-existence-critical}, we show that the stability condition
$0<T<T_{0}$ is in some sense also necessary for the above non-existence result
in $W^{s,p}\ \left(  s>1+\frac{1}{p}\right)  $.

The following corollary shows that all homogeneous equilibria in a
sufficiently small $W^{s,p}\left(  \mathbf{R}\right)  \left(  s>1+\frac{1}%
{p}\right)  $ neighborhood of a stable homogeneous state remain linearly
stable. With Corollary \ref{cor-unsatble}, it implies that $s=1+\frac{1}{p}$
is also the critical index for persistence of linear stability of homogeneous
states under perturbations in $W^{s,p}\left(  \mathbf{R}\right)  $ space.

\begin{corollary}
\label{cor-stable}Assume $f_{0}\left(  v\right)  \in W^{s,p}\left(
\mathbf{R}\right)  $ $\left(  p>1,s>1+\frac{1}{p}\right)  .\ $Let $S=\left\{
v_{i}\right\}  _{i=1}^{l}$ be the set of all extrema points of $f_{0}$ and
$T_{0}$ be defined in (\ref{defn-T-0}). Then for any $T<T_{0}$, $\exists$
$\varepsilon_{0}\left(  T\right)  >0$ such that any homogeneous state $\left(
f\left(  v\right)  ,0\right)  $ satisfying$\ $%
\[
\left\Vert f\left(  v\right)  -f_{0}\left(  v\right)  \right\Vert
_{W^{s,p}\left(  \mathbf{R}\right)  }<\varepsilon_{0}%
\]
is linearly stable under perturbations of $x-$period $T$.
\end{corollary}

Theorem \ref{thm-non-existence} and the above Corollary suggest that the
dynamical structures in small $W^{s,p}$ $\left(  s>1+\frac{1}{p}\right)
\ $neighborhood of a stable homogeneous equilibrium might be relatively
simple, since the only nearby steady structures, including travelling waves,
are stable homogeneous states. The physical implication of Theorem
\ref{thm-non-existence} is that when the initial perturbation is small in
$W^{s,p}$ $\left(  s>1+\frac{1}{p}\right)  $, the potential well of the wave
is unable to trap particles forever to form BGK waves. So the particles will
get out of the potential well sooner or later and perform free flights, then
the linear damping effect might manifest itself at the nonlinear level.

Furthermore, when $p=2,$ we get a much stronger result that any invariant
structure near a stable homogeneous state in $H^{s}$ space $\left(  s>\frac
{3}{2}\right)  \ $must be trivial, that is, the electric field is identically zero.

\begin{theorem}
\label{thm-invariant}Assume the homogeneous profile $f_{0}\left(  v\right)
\in H^{s}\left(  \mathbf{R}\right)  $ $\left(  s>\frac{3}{2}\right)  .\ $For
any $T<T_{0}$ (defined by (\ref{defn-T-0})), there exists $\varepsilon_{0}>0$,
such that if $\left(  f\left(  t\right)  ,E\left(  t\right)  \right)  $ is a
solution to the nonlinear VP equation (\ref{vlasov})-(\ref{poisson}) and
\begin{equation}
\left\Vert f\left(  t\right)  -f_{0}\right\Vert _{L_{x}^{2}H_{v}^{s}%
}<\varepsilon_{0},\ \text{for\ all\ }t\in\mathbf{R},
\label{assumption-thm-invariant}%
\end{equation}
then $E\left(  t\right)  \equiv0$ for all $t\in\mathbf{R}.$
\end{theorem}

The space $L_{x}^{2}H_{v}^{s}$ is contained in the Sobolev space $H_{x,v}^{s}%
$. The above theorem excludes any nontrivial invariant structure, such as
almost periodic solutions and heteroclinic (homoclinic) orbits, in the $H^{s}$
$\left(  s>\frac{3}{2}\right)  \ $neighborhood of a stable homogeneous state.
In Theorem \ref{thm-semi-invariant}, we also show that nonlinear decay of
electric field is true for any positive or negative invariant structure (see
Section 5 for definition) in the $H^{s}$ $\left(  s>\frac{3}{2}\right)
\ $neighborhood of stable homogeneous states. These results reveal that in
contrary to the $H^{s}$ $\left(  s<\frac{3}{2}\right)  \,$\ case, there are no
obstacles in the $H^{s}$ $\left(  s>\frac{3}{2}\right)  \ $neighborhood to
prevent nonlinear Landau damping$.$

We note that Theorems \ref{thm-existence}, \ref{thm-non-existence} and
\ref{thm-invariant} about the contrasting nonlinear dynamics in $W^{s,p}%
\ $spaces with $s<1+\frac{1}{p}$ or $s>1+\frac{1}{p}\ $(particularly when
$p=2$),\ have no any analogue at the linear level. Indeed, under Penrose's
stability condition, it is shown in Section 4 that the linear decay of
electrical fields holds true for very rough initial data, particularly, no
derivatives of $f\left(  t=0\right)  $ is required for linear damping. We
refer to Propositions \ref{prop-linear-integral-estimate} and
\ref{prop-linear}, as well as Remark \ref{rmk-linear} in Section 4 for more
details. This shows once again the importance of particle trapping effects on
nonlinear dynamics, which are completely ignored at the linear level.

Finally, we briefly describe main ideas in the proof of Theorems
\ref{thm-existence}, \ref{thm-non-existence} and \ref{thm-invariant}. For
simplicity, we look at steady BGK waves. The first attempt would be to
construct BGK waves near $\left(  f_{0}\left(  v\right)  ,0\right)  $ directly
by the bifurcation theory. However, this requires a bifurcation condition: for
bifurcation period $T>0,$%
\begin{equation}
\left(  \frac{2\pi}{T}\right)  ^{2}=\int_{\mathbf{R}}\frac{f_{0}^{\prime
}\left(  v\right)  }{v}dv. \label{condition-bifurcation}%
\end{equation}
For general homogeneous equilibria and period $T$, the bifurcation condition
(\ref{condition-bifurcation}) is not satisfied. For example, for Maxwellian,
this condition fails for any $T>0$. Our strategy is to modify $f_{0}\left(
v\right)  \ $to get a nearby homogeneous state satisfying
(\ref{condition-bifurcation}) and then do bifurcation near this modified
state. In the modification step, we introduce two parameters, one is to to
obtain (\ref{condition-bifurcation}) and the other one is to ensure that the
modification results in a small $W^{s,p}$ $\left(  s<1+\frac{1}{p}\right)
\ $norm change. For the proof of non-existence of travelling waves in
$W^{s,p}$ $\left(  s>1+\frac{1}{p}\right)  $, our idea is to get an second
order equation for the electrical field $E\left(  x\right)  $ from steady
Vlasov-Poisson equations and show the integral form of this equation is not
compatible when $T<T_{0}\ $and the perturbation is small in $W^{s,p}$ $\left(
s>1+\frac{1}{p}\right)  $. Interestingly, $T_{0}$ (defined by (\ref{defn-T-0}%
)) is exactly the critical period for linear stability by Penrose's criterion,
which is also used in the proof of Corollaries \ref{cor-unsatble} and
\ref{cor-stable}. To prove Theorem \ref{thm-invariant}, we use the integral
form of the linear decay estimate (Proposition
\ref{prop-linear-integral-estimate}) and the $H^{s}\ \left(  s>\frac{3}%
{2}\right)  $ invariant assumption to obtain similar nonlinear decay estimates
in the integral form. From such integral estimates, we can show the
homogeneous nature of the invariant structures and the decay of electric field
for semi-invariant solutions.

Here we are in a position to offer a conceptual explanation why $s=1+\frac
{1}{p}$ appears as the critical Sobolev exponent for the existence of small
BKG waves and possibly also in the nonlinear Landau damping. By Penrose'
stability criterion, the critical spatial period $T_{0}$ for linear stability
of $\left(  f_{0}\left(  v\right)  ,0\right)  \ $is determined in
\eqref{defn-T-0} by integrals $\int\frac{f_{0}^{\prime}}{v-c_{i}}dv$, where
$c_{i}$ are critical points of $f_{0}$. These integrals are controlled by
$||f_{0}||_{W^{s,p}}$ if $s>1+\frac{1}{p}$, but not if $s<1+\frac{1}{p}$. In
the latter case, a small homogeneous perturbation to $f_{0}$ in $W^{s,p}$
space may dramatically change its stability for any fixed spatial period $T$.
Due to this change of stability, bifurcations occur and produce small BKG
waves and possibly other complicated structures. In the opposite case when
$s>1+\frac{1}{p}$, small homogeneous perturbations do not change the stability
of $\left(  f_{0}\left(  v\right)  ,0\right)  $, therefore the bifurcation of
nontrivial structures cannot occur and the nonlinear Landau damping might be
expected.

The result of this paper has also been extended to a related problem of
inviscid decay of Couette flow $\vec{v}_{0}=\left(  y,0\right)  \ $of 2D Euler
equations. The linear decay of vertical velocity near Couette flow was already
known by Orr (\cite{orr-1907}) in 1907. This inviscid decay problem is
important to understand the formation of coherent structures in 2D turbulence.
In \cite{lin-zeng-euler-decay}, we are able to obtain similar results near the
Couette flow.

This paper is organized as follows. In Section 2, we prove the existence
result in $W^{s,p}\ \left(  s<1+\frac{1}{p}\right)  $. In Section 3,
non-existence of BGK waves in $W^{s,p}\left(  s>1+\frac{1}{p}\right)  $ is
shown. In Section 4, we study the linear damping problem in Sobolev spaces. In
Section 5, we use the linear decay estimate in Section 4 to show that all
invariant structures in $H^{s}\ \left(  s>\frac{3}{2}\right)  $ are trivial.
The appendix is to reformulate Penrose's linear stability criterion used in
this paper. Throughout this paper, we use $C$ to denote a generic constant in
the estimates and the dependence of $C$ is indicated only when it matters in
the proof.

\section{Existence of BGK waves in $W^{s,p}\ \left(  s<1+\frac{1}{p}\right)
$}

In this Section, we construct small BGK waves near any homogeneous state in
the space $W^{s,p}$ $\left(  s<1+\frac{1}{p}\right)  $. Our strategy is to
first construct BGK waves near proper smooth homogeneous states. Then we show
that any homogeneous state can be approximated by such smooth states in
$W^{s,p}$.

\begin{lemma}
\label{lemma-even}Assume $u\left(  x\right)  \in C^{\infty}\left(
\mathbf{R}\right)  $, supp $u\subset\left[  -b,b\right]  $, and $u\left(
x\right)  $ is even, then there exists $g\in C^{\infty}\left(  \mathbf{R}%
\right)  $, supp $g\subset\left[  -\sqrt{b},\sqrt{b}\right]  ,$ such that
$u\left(  x\right)  =g\left(  x^{2}\right)  $.
\end{lemma}

\begin{proof}
The proof is essentially given in \cite[P. 394]{hormander-1}. We repeat it
here for completeness. When $k$ is odd, since $u^{\left(  k\right)  }\left(
x\right)  $ is odd we have $u^{\left(  k\right)  }\left(  0\right)  =0$. By
Theorem 1.2.6 in \cite{hormander-1}, we can choose $g_{0}\in C_{0}^{\infty
}\left(  -\sqrt{b},\sqrt{b}\right)  $ with the Taylor expansion $\sum
u^{\left(  2k\right)  }\left(  0\right)  x^{k}/\left(  2k\right)  !.$ Then all
derivatives of $u_{1}\left(  x\right)  =u\left(  x\right)  -g_{0}\left(
x^{2}\right)  \ $vanish at $0$. Define
\[
g\left(  x\right)  =\left\{
\begin{array}
[c]{cc}%
g_{0}\left(  x\right)  +u_{1}\left(  \sqrt{x}\right)  & \text{if }x>0\\
g_{0}\left(  x\right)  & \text{if }x\leq0
\end{array}
\right.  .
\]
Then $g\left(  x\right)  $ satisfies all the required properties. In
particular, $g\left(  x\right)  $ is $C^{\infty}$ at $x=0$ because all
derivatives of $u_{1}\left(  x\right)  $ vanish there.
\end{proof}

\begin{proposition}
\label{propo-existence}Assume
\[
f_{0}\left(  v\right)  \in C^{\infty}\left(  \mathbf{R}\right)  \cap
W^{2,p}\left(  \mathbf{R}\right)  \ \left(  p>1\right)  ,\
\]
$f_{0}$ is even near $v=0$, and
\[
f_{0}>0,\ \int_{\mathbf{R}}f_{0}\left(  v\right)  dv=1,\ \int_{\mathbf{R}%
}v^{2}f_{0}\left(  v\right)  dv<\infty.
\]
Then for any fixed $s<1+\frac{1}{p},\ T>0,\ $and any $\varepsilon>0$, there
exist steady BGK solutions of the form $\left(  f_{\varepsilon}\left(
x,v\right)  ,E_{\varepsilon}\left(  x\right)  \right)  $ to (\ref{vpe}), such
that$\ \left(  f_{\varepsilon}\left(  x,v\right)  ,E_{\varepsilon}\left(
x\right)  \right)  $ has period $T$ in $x$,$\ f_{\varepsilon}\left(
x,v\right)  >0,$ $E_{\varepsilon}\left(  x\right)  $ is not identically zero,
and
\begin{equation}
\ \left\Vert f_{\varepsilon}-f_{0}\right\Vert _{L_{x,v}^{1}}+\ \int_{0}%
^{T}\int_{\mathbf{R}}v^{2}\left\vert f_{0}-f_{\varepsilon}\right\vert
\ dxdv+\left\Vert f_{\varepsilon}-f_{0}\right\Vert _{W_{x,v}^{s,p}%
}<\varepsilon.\ \label{norm-proposition}%
\end{equation}

\end{proposition}

\begin{proof}
Assume $f_{0}\left(  v\right)  $ is even in $\left[  -2a,2a\right]  $ $\left(
a>0\right)  $. Let $\sigma\left(  x\right)  =\sigma\left(  \left\vert
x\right\vert \right)  $ to be the cut-off function such that$\ \sigma\left(
x\right)  \in C_{0}^{\infty}\left(  \mathbf{R}\right)  ,$
\begin{equation}
\ 0\leq\sigma\left(  x\right)  \leq1;\ \sigma\left(  x\right)  =1\text{ when
}\left\vert x\right\vert \leq1\text{; }\sigma\left(  x\right)  =0\text{ when
}\left\vert x\right\vert \geq2\text{.} \label{cut-off}%
\end{equation}
By Lemma \ref{lemma-even}, there exists $g_{0}\left(  x\right)  \in C^{\infty
}\left(  \mathbf{R}\right)  $, supp $g_{0}\subset\left[  -\sqrt{2a},\sqrt
{2a}\right]  ,$ such that
\[
f_{0}\left(  v\right)  \sigma\left(  \frac{v}{a}\right)  =g_{0}\left(
v^{2}\right)  .
\]
Define $g_{+}\left(  x\right)  ,\ g_{-}\left(  x\right)  \in C^{\infty}\left(
\mathbf{R}\right)  $ by
\[
g_{\pm}\left(  x\right)  =\left\{
\begin{array}
[c]{cc}%
f_{0}\left(  \pm\sqrt{x}\right)  \left(  1-\sigma\left(  \frac{\sqrt{x}}%
{a}\right)  \right)  +g_{0}\left(  x\right)  & \text{if }x>\sqrt{a}\\
g_{0}\left(  x\right)  & \text{if }-\sqrt{2a}<x\leq\sqrt{a}\\
0 & \text{if }x\leq-\sqrt{2a}%
\end{array}
.\right.
\]
Then%
\[
f_{0}\left(  v\right)  =\left\{
\begin{array}
[c]{cc}%
g_{+}\left(  v^{2}\right)  & \text{if }v>0\\
g_{-}\left(  v^{2}\right)  & \text{if }v\leq0
\end{array}
\right.  .
\]
Since $f_{0}^{\prime}\left(  0\right)  =0$, $f_{0}\in W^{2,p}\left(
\mathbf{R}\right)  \cap C^{\infty}\left(  \mathbf{R}\right)  $, we have
$\left\vert \int_{\mathbf{R}}\frac{f_{0}^{\prime}\left(  v\right)  }%
{v}dv\right\vert <\infty$. Indeed,%
\begin{align*}
\left\vert \int_{\mathbf{R}}\frac{f_{0}^{\prime}\left(  v\right)  }%
{v}dv\right\vert  &  \leq\int_{\left\vert v\right\vert \leq1}\left\vert
\frac{f_{0}^{\prime}\left(  v\right)  }{v}\right\vert dv+\int_{\left\vert
v\right\vert \geq1}\left\vert \frac{f_{0}^{\prime}\left(  v\right)  }%
{v}\right\vert dv\\
&  \leq2\max_{\left\vert v\right\vert \leq1}\left\vert f_{0}^{\prime\prime
}\left(  v\right)  \right\vert +\left(  \int_{\left\vert v\right\vert \geq
1}\frac{1}{\left\vert v\right\vert ^{p^{\prime}}}dv\right)  ^{\frac
{1}{p^{\prime}}}\left\Vert f_{0}^{\prime}\right\Vert _{L^{p}}\\
&  =2\max_{\left\vert v\right\vert \leq1}\left\vert f_{0}^{\prime\prime
}\left(  v\right)  \right\vert +\left(  \frac{1}{p^{\prime}-1}\right)
^{\frac{1}{p^{\prime}}}\left\Vert f_{0}^{\prime}\right\Vert _{L^{p}}<\infty.
\end{align*}
We consider three cases.

Case 1: $\int_{\mathbf{R}}\frac{f_{0}^{\prime}\left(  v\right)  }{v}dv<\left(
\frac{2\pi}{T}\right)  ^{2}$. Choose a function $F\left(  v\right)  \in
C^{\infty}\left(  \mathbf{R}\right)  ,$ such that $F\in W^{2,p}\left(
\mathbf{R}\right)  ,\ F\left(  v\right)  $ is even,
\begin{equation}
F\left(  v\right)  >0,\int_{\mathbf{R}}F\left(  v\right)  dv<\infty
,~\int_{\mathbf{R}}v^{2}F\left(  v\right)  dv<\infty,\ \int_{\mathbf{R}}%
\frac{F^{\prime}\left(  v\right)  }{v}dv>0. \label{property-F}%
\end{equation}
An example of such functions is given by
\[
F\left(  v\right)  =\exp\left(  -\frac{\left(  v-v_{0}\right)  ^{2}}%
{2}\right)  +\exp\left(  -\frac{\left(  v+v_{0}\right)  ^{2}}{2}\right)  ,
\]
where $v_{0}$ is a large positive constant. Indeed,
\[
\int_{\mathbf{R}}\frac{F^{\prime}\left(  v\right)  }{v}dv=\int\frac{F\left(
v\right)  -F\left(  0\right)  }{v^{2}}dv>0,\ \text{when }v_{0}\ \text{is large
enough,}%
\]
and other properties in (\ref{property-F}) are easy to check. Since $F\left(
v\right)  $ is even, by Lemma \ref{lemma-even}, there exists $G\left(
x\right)  \in C^{\infty}\left(  \mathbf{R}\right)  \ $such that $F\left(
v\right)  =G\left(  v^{2}\right)  $. Let $\gamma,\delta>0$ be two small
parameters to be fixed, define
\begin{equation}
f_{\gamma,\delta}\left(  v\right)  =\frac{1}{1+C_{0}\gamma^{2}}\left[
f_{0}\left(  v\right)  +\frac{\gamma}{\delta}F\left(  \frac{v}{\gamma\delta
}\right)  \right]  , \label{defn-f-gamma-delta}%
\end{equation}
where $C_{0}=\int F\left(  v\right)  dv>0$. Note that $f_{\gamma,\delta}\in
C^{\infty}\left(  \mathbf{R}\right)  \cap W^{2,p}\left(  \mathbf{R}\right)
,\ \int_{\mathbf{R}}f_{\gamma,\delta}\left(  v\right)  dv=1,$ and
\[
\int_{\mathbf{R}}\frac{f_{\gamma,\delta}^{\prime}\left(  v\right)  }%
{v}dv=\frac{1}{1+C_{0}\gamma^{2}}\left[  \int_{\mathbf{R}}\frac{f_{0}^{\prime
}\left(  v\right)  }{v}dv+\frac{1}{\delta^{2}}\int_{\mathbf{R}}\frac
{F^{\prime}\left(  v\right)  }{v}dv\right]  .
\]
Since $\int_{\mathbf{R}}\frac{f_{0}^{\prime}\left(  v\right)  }{v}dv<\left(
\frac{2\pi}{T}\right)  ^{2}$, there exists $0<\delta_{1}<\delta_{2}$ such
that
\[
0<\int_{\mathbf{R}}\frac{f_{0}^{\prime}\left(  v\right)  }{v}dv+\frac
{1}{\delta_{2}^{2}}\int_{\mathbf{R}}\frac{F^{\prime}\left(  v\right)  }%
{v}dv<\left(  \frac{2\pi}{T}\right)  ^{2}<\int_{\mathbf{R}}\frac{f_{0}%
^{\prime}\left(  v\right)  }{v}dv+\frac{1}{\delta_{1}^{2}}\int_{\mathbf{R}%
}\frac{F^{\prime}\left(  v\right)  }{v}dv.
\]
Thus there exists $\gamma_{0}>0$ small enough, such that
\begin{equation}
0<\int_{\mathbf{R}}\frac{f_{\gamma,\delta_{2}}^{\prime}\left(  v\right)  }%
{v}dv<\left(  \frac{2\pi}{T}\right)  ^{2}<\int_{\mathbf{R}}\frac
{f_{\gamma,\delta_{1}}^{\prime}\left(  v\right)  }{v}dv,\text{ when }%
0<\gamma<\gamma_{0}\text{.} \label{ineq-period-0}%
\end{equation}
We look for steady BGK waves near the homogeneous states$\ \left(
f_{\gamma,\delta}\left(  v\right)  ,0\right)  $. Consider a steady BGK
solution $\left(  f^{0}\left(  x,v\right)  ,E^{0}\left(  x\right)  =-\beta
_{x}\left(  x\right)  \right)  $ to ($\ref{vpe})$. Denote $e=\frac{1}{2}%
v^{2}-\beta\left(  x\right)  $ to be the particle energy. From the steady
Vlasov equation, $f^{0}\left(  x,v\right)  $ is constant along each particle
trajectory. So for trapped particles with $-\max\beta<e<-\min\beta$, $f^{0}$
depends only on $e$, and for free particles with $e>-\min\beta$, $f^{0}$
depends on $e~$and the sign of the initial velocity $v$. We look for BGK waves
near $\left(  f_{\gamma,\delta},0\right)  $ of the form
\begin{equation}
f_{\gamma,\delta}^{\beta}\left(  x,v\right)  =\left\{
\begin{array}
[c]{cc}%
\frac{1}{1+C_{0}\gamma^{2}}\left[  g_{+}\left(  2e\right)  +\frac{\gamma
}{\delta}G\left(  \frac{2e}{\left(  \gamma\delta\right)  ^{2}}\right)  \right]
& \text{if }v>0\\
\frac{1}{1+C_{0}\gamma^{2}}\left[  g_{-}\left(  2e\right)  +\frac{\gamma
}{\delta}G\left(  \frac{2e}{\left(  \gamma\delta\right)  ^{2}}\right)  \right]
& \text{if }v\leq0
\end{array}
\right.  . \label{defn-f-steady}%
\end{equation}
For $\left\Vert \beta\right\Vert _{L^{\infty}}$ sufficiently small,
$f_{\gamma,\delta}^{\beta}\left(  x,v\right)  >0$ and it satisfies the steady
Vlasov equation, since in particular for trapped particles $f_{\gamma,\delta
}^{\beta}\left(  x,v\right)  =f_{\gamma,\delta}^{\beta}\left(  x,-v\right)  $.
To satisfy Poisson's equation, we solve the ODE
\begin{align*}
\beta_{xx}  &  =\int_{\mathbf{R}}f_{\gamma,\delta}^{\beta}\left(  x,v\right)
\ dv-1\\
&  =\frac{1}{1+C_{0}\gamma^{2}}\left[  \int_{v>0}g_{+}\left(  2e\right)
dv+\int_{v\leq0}g_{-}\left(  2e\right)  dv+\int_{\mathbf{R}}\frac{\gamma
}{\delta}G\left(  \frac{2e}{\left(  \gamma\delta\right)  ^{2}}\right)
dv\right]  -1\\
&  :=h_{\gamma,\delta}\left(  \beta\right)  .
\end{align*}
Then $h_{\gamma,\delta}\in C^{\infty}\left(  \mathbf{R}\right)  $. Since
$f_{\gamma,\delta}^{\beta=0}\left(  x,v\right)  =f_{\gamma,\delta}\left(
v\right)  $, so
\[
h_{\gamma,\delta}\left(  0\right)  =\int_{\mathbf{R}}f_{\gamma,\delta}\left(
v\right)  \ dv-1=0
\]
and
\begin{align*}
h_{\gamma,\delta}^{\prime}\left(  0\right)   &  =\frac{-2}{1+C_{0}\gamma^{2}%
}\left\{  \int_{v>0}g_{+}^{\prime}\left(  v^{2}\right)  dv+\int_{v\leq0}%
g_{-}^{\prime}\left(  v^{2}\right)  dv+\int_{\mathbf{R}}\frac{\gamma}{\delta
}\frac{1}{\left(  \gamma\delta\right)  ^{2}}G^{\prime}\left(  \frac{v^{2}%
}{\left(  \gamma\delta\right)  ^{2}}\right)  dv\right\} \\
&  =-\int_{\mathbf{R}}\frac{f_{\gamma,\delta}^{\prime}\left(  v\right)  }%
{v}dv.
\end{align*}
Thus when $0<\gamma<\gamma_{0},\ \delta_{1}<\delta<\delta_{2},\ $we have
$h_{\gamma,\delta}^{\prime}\left(  0\right)  <0$, which implies that $\beta=0$
is a center of the second order ODE
\begin{equation}
\beta_{xx}=h_{\gamma,\delta}\left(  \beta\right)  . \label{ode-beta}%
\end{equation}
So by the standard bifurcation theory of periodic solutions near a center, for
any fixed $\gamma\in\left(  0,\gamma_{0}\right)  ,\ $there exists $r_{0}>0$
(independent of $\delta\in\left(  \delta_{1},\delta_{2}\right)  $)$\,$, such
that for each $0<r<r_{0}\,$, there exists a $T\left(  \gamma,\delta;r\right)
-$periodic solution $\beta_{\gamma,\delta;r}$ to the ODE (\ref{ode-beta}) with
$\left\Vert \beta_{\gamma,\delta;r}\right\Vert _{H^{2}\left(  0,T\left(
\gamma,\delta;r\right)  \right)  }=r$. Moreover,
\[
\left(  \frac{2\pi}{T\left(  \gamma,\delta;r\right)  }\right)  ^{2}%
\rightarrow\int_{\mathbf{R}}\frac{f_{\gamma,\delta}^{\prime}\left(  v\right)
}{v}dv\text{, when }r\rightarrow0.
\]
By (\ref{ineq-period-0}), when $r$ is small enough,
\[
T\left(  \gamma,\delta_{1};r\right)  <T<T\left(  \gamma,\delta_{2};r\right)
.
\]
Since $T\left(  \gamma,\delta;r\right)  $ is continuous
in
$\delta,$ for each $\gamma,r>0$ small enough, there exists $\delta_{T}\left(
\gamma,r\right)  \in\left(  \delta_{1},\delta_{2}\right)  \,$, such that
$T\left(  \gamma,\delta_{T};r\right)  =T$. Define $f_{\gamma,r}^{T}\left(
x,v\right)  =f_{\gamma,\delta_{T}}^{\beta}\left(  x,v\right)  $ from
(\ref{defn-f-steady}) by setting $\beta=\beta_{\gamma,\delta_{T};r}$ and let
$E_{\gamma,r}\left(  x\right)  =-\beta_{\gamma,\delta_{T};r}^{\prime}\left(
x\right)  $. Then $\left(  f_{\gamma,r}^{T}\left(  x,v\right)  ,E_{\gamma
,r}\left(  x\right)  \right)  $ is a nontrivial BGK solution to ($\ref{vpe})$
with $x-$period $T$. For any fixed $\gamma>0$, let
\[
\delta\left(  \gamma\right)  =\lim_{r\rightarrow0}\delta_{T}\left(
\gamma,r\right)  \in\left[  \delta_{1},\delta_{2}\right]  .
\]
By the dominant convergence theorem, it is easy to show that
\[
\left\Vert f_{\gamma,r}^{T}\left(  x,v\right)  -f_{\gamma,\delta\left(
\gamma\right)  }\left(  v\right)  \right\Vert _{L_{x,v}^{1}}+\ \int_{0}%
^{T}\int_{\mathbf{R}}v^{2}\left\vert f_{\gamma,r}^{T}\left(  x,v\right)
-f_{\gamma,\delta\left(  \gamma\right)  }\left(  v\right)  \right\vert
\ dxdv\ \ \ \
\]%
\[
\ +\left\Vert f_{\gamma,r}^{T}\left(  x,v\right)  -f_{\gamma,\delta\left(
\gamma\right)  }\left(  v\right)  \right\Vert _{W_{x,v}^{2,p}}\rightarrow
0,\ \ \ \
\]
when $r=\left\vert \beta_{\gamma,\delta_{T};r}\right\vert _{H^{2}\left(
0,T\right)  }\rightarrow0.\ $Since $s<1+\frac{1}{p}<2,\ $for any $\gamma>0$
small, there exists $r=r\left(  \gamma,\varepsilon\right)  >0$ such that
\[
\left\Vert f_{\gamma,r}^{T}\left(  x,v\right)  -f_{\gamma,\delta\left(
\gamma\right)  }\left(  v\right)  \right\Vert _{L_{x,v}^{1}}+\ \int_{0}%
^{T}\int_{\mathbf{R}}v^{2}\left\vert f_{\gamma,r}^{T}\left(  x,v\right)
-f_{\gamma,\delta\left(  \gamma\right)  }\left(  v\right)  \right\vert
\ dxdv\ \
\]%
\[
\ +\left\Vert f_{\gamma,r}^{T}\left(  x,v\right)  -f_{\gamma,\delta\left(
\gamma\right)  }\left(  v\right)  \right\Vert _{W_{x,v}^{s,p}}<\frac
{\varepsilon}{2},
\]
Next, we show that the modified homogeneous state$\ f_{\gamma,\delta\left(
\gamma\right)  }\left(  v\right)  $ is arbitrarily close to $f_{0}\left(
v\right)  $ in the sense that
\[
\left\Vert f_{0}\left(  v\right)  -f_{\gamma,\delta\left(  \gamma\right)
}\left(  v\right)  \right\Vert _{L^{1}}+\ T\int_{\mathbf{R}}v^{2}\left\vert
f_{0}\left(  v\right)  -f_{\gamma,\delta\left(  \gamma\right)  }\left(
v\right)  \right\vert \ dv+\left\Vert f_{0}\left(  v\right)  -f_{\gamma
,\delta\left(  \gamma\right)  }\left(  v\right)  \right\Vert _{W_{x,v}^{s,p}%
}\rightarrow0,
\]
when $\gamma\rightarrow0.\ $Note that the deviation is%
\[
f_{0}\left(  v\right)  -f_{\gamma,\delta\left(  \gamma\right)  }\left(
v\right)  =\frac{1}{1+C_{0}\gamma^{2}}\left[  -C_{0}\gamma^{2}f_{0}\left(
v\right)  -\frac{\gamma}{\delta}F\left(  \frac{v}{\gamma\delta}\right)
\right]  .
\]
Since $\delta\left(  \gamma\right)  \in\left[  \delta_{1},\delta_{2}\right]
$,\ when $\gamma\rightarrow0$,
\[
\int_{\mathbf{R}}\frac{\gamma}{\delta\left(  \gamma\right)  }F\left(  \frac
{v}{\gamma\delta\left(  \gamma\right)  }\right)  dv=\gamma^{2}\int
_{\mathbf{R}}F\left(  v\right)  dv\rightarrow0,
\]%
\[
\int_{\mathbf{R}}v^{2}\frac{\gamma}{\delta\left(  \gamma\right)  }F\left(
\frac{v}{\gamma\delta\left(  \gamma\right)  }\right)  dv\ =\gamma^{4}%
\delta\left(  \gamma\right)  ^{2}\int_{\mathbf{R}}v^{2}F\left(  v\right)
dv\rightarrow0,\
\]%
\[
\left\Vert \frac{\gamma}{\delta\left(  \gamma\right)  }F\left(  \frac
{v}{\gamma\delta\left(  \gamma\right)  }\right)  \right\Vert _{L^{p}}%
=\gamma^{1+\frac{1}{p}}\delta\left(  \gamma\right)  ^{\frac{1}{p}-1}\left\Vert
F\left(  v\right)  \right\Vert _{L^{p}}\rightarrow0,
\]%
\[
\left\Vert \frac{d}{dv}\left(  \frac{\gamma}{\delta\left(  \gamma\right)
}F\left(  \frac{v}{\gamma\delta\left(  \gamma\right)  }\right)  \right)
\right\Vert _{L^{p}}=\gamma^{\frac{1}{p}}\delta\left(  \gamma\right)
^{\frac{1}{p}-2}\left\Vert F^{\prime}\left(  v\right)  \right\Vert _{L^{p}%
}\rightarrow0,
\]
and thus%
\[
\left\Vert f_{0}\left(  v\right)  -f_{\gamma,\delta\left(  \gamma\right)
}\left(  v\right)  \right\Vert _{L^{1}}+\ T\int_{\mathbf{R}}v^{2}\left\vert
f_{0}\left(  v\right)  -f_{\gamma,\delta\left(  \gamma\right)  }\left(
v\right)  \right\vert \ dv+\left\Vert f_{0}\left(  v\right)  -f_{\gamma
,\delta\left(  \gamma\right)  }\left(  v\right)  \right\Vert _{W_{x,v}^{1,p}%
}\rightarrow0.
\]
It remains to check
\[
\left\Vert \left\vert D\right\vert ^{s-1}\frac{d}{dv}\left(  f_{0}\left(
v\right)  -f_{\gamma,\delta\left(  \gamma\right)  }\left(  v\right)  \right)
\right\Vert _{L^{p}}\rightarrow0,\text{ when }\gamma\rightarrow0,
\]
where $\left\vert D\right\vert ^{\delta}$ $\left(  \delta>0\right)  \ $is the
fractional differentiation operator with the Fourier symbol $\left\vert
\xi\right\vert ^{\delta}.$ By using the scaling equality%
\[
\left(  \left\vert D\right\vert ^{\delta}\chi_{d}\right)  \left(  v\right)
=\frac{1}{d^{\delta}}\left(  \left\vert D\right\vert ^{\delta}\chi\right)
\left(  \frac{v}{d}\right)  ,
\]
where $\chi_{d}\left(  v\right)  =\chi\left(  v/d\right)  $, we have \
\[
\left\Vert \left\vert D\right\vert ^{s}\frac{d}{dv}\left(  \frac{\gamma
}{\delta\left(  \gamma\right)  }F\left(  \frac{v}{\gamma\delta\left(
\gamma\right)  }\right)  \right)  \right\Vert _{L^{p}}=\gamma^{1-s+\frac{1}%
{p}}\delta\left(  \gamma\right)  ^{-1-s+\frac{1}{p}}\left\Vert \left(
\left\vert D\right\vert ^{s}F^{\prime}\right)  \left(  v\right)  \right\Vert
_{L^{p}}\rightarrow0,
\]
when $\gamma\rightarrow0$, since $s<1+\frac{1}{p}$. So we can choose
$\gamma=\gamma\left(  \varepsilon\right)  >0$ small such that
\[
\left\Vert f_{0}-f_{\gamma,\delta\left(  \gamma\right)  }\right\Vert
_{L^{1}\left(  \mathbf{R}\right)  }+\ T\int_{\mathbf{R}}v^{2}\left\vert
f_{0}\left(  v\right)  -f_{\gamma,\delta\left(  \gamma\right)  }\left(
v\right)  \right\vert \ dv+\left\Vert f_{0}-f_{\gamma,\delta\left(
\gamma\right)  }\right\Vert _{W^{s,p}\left(  \mathbf{R}\right)  }%
<\frac{\varepsilon}{2}.
\]
Then
\[
\left(  f_{\varepsilon},E_{\varepsilon}\right)  =\left(  f_{\gamma\left(
\varepsilon\right)  ,r\left(  \gamma\left(  \varepsilon\right)  ,\varepsilon
\right)  }^{T}\left(  x,v\right)  ,E_{\gamma\left(  \varepsilon\right)
,r\left(  \gamma\left(  \varepsilon\right)  ,\varepsilon\right)  }\left(
x\right)  \right)
\]
is a steady BGK wave solution satisfying (\ref{norm-proposition}).

Case 2: $\int_{\mathbf{R}}\frac{f_{0}^{\prime}\left(  v\right)  }{v}dv>\left(
\frac{2\pi}{T}\right)  ^{2}$. Choose $F\left(  v\right)  =\exp\left(
-\frac{v^{2}}{2}\right)  ,$ then $\int_{\mathbf{R}}\frac{F^{\prime}\left(
v\right)  }{v}dv<0$. Define $f_{\gamma,\delta}\left(  v\right)  $ as in Case 1
(see (\ref{defn-f-gamma-delta})). Then there exists $0<\delta_{1}<\delta_{2}$
such that
\[
0<\int_{\mathbf{R}}\frac{f_{0}^{\prime}\left(  v\right)  }{v}dv+\frac
{1}{\delta_{1}^{2}}\int_{\mathbf{R}}\frac{F^{\prime}\left(  v\right)  }%
{v}dv<\left(  \frac{2\pi}{T}\right)  ^{2}<\int_{\mathbf{R}}\frac{f_{0}%
^{\prime}\left(  v\right)  }{v}dv+\frac{1}{\delta_{2}^{2}}\int_{\mathbf{R}%
}\frac{F^{\prime}\left(  v\right)  }{v}dv.
\]
The rest of the proof is the same as in Case 1.

Case 3: $\int_{\mathbf{R}}\frac{f_{0}^{\prime}\left(  v\right)  }{v}dv=\left(
\frac{2\pi}{T}\right)  ^{2}$. For $\delta>0,\ $define
\[
f_{\delta}\left(  v\right)  =\frac{1}{\delta}f_{0}\left(  \frac{v}{\delta
}\right)  .
\]
Then $f_{\delta}\in C^{\infty}\left(  \mathbf{R}\right)  \cap W^{2,p}\left(
\mathbf{R}\right)  ,\ f_{\delta}\left(  v\right)  >0,\ \int_{\mathbf{R}%
}f_{\delta}\left(  v\right)  dv=1,$ and
\[
\int_{\mathbf{R}}\frac{f_{\delta}^{\prime}\left(  v\right)  }{v}dv=\frac
{1}{\delta^{2}}\int_{\mathbf{R}}\frac{f_{0}^{\prime}\left(  v\right)  }{v}dv.
\]
For any $\varepsilon>0$ small, there exist $0<\delta_{1}\left(  \varepsilon
\right)  <1<\delta_{2}\left(  \varepsilon\right)  $ such that
\[
0<\frac{1}{\delta_{2}^{2}}\int_{\mathbf{R}}\frac{f_{0}^{\prime}\left(
v\right)  }{v}dv<\left(  \frac{2\pi}{T}\right)  ^{2}<\frac{1}{\delta_{1}^{2}%
}\int_{\mathbf{R}}\frac{f_{0}^{\prime}\left(  v\right)  }{v}dv
\]
and when $\delta\in\left(  \delta_{1}\left(  \varepsilon\right)  ,\delta
_{2}\left(  \varepsilon\right)  \right)  ,$
\[
\left\Vert f_{0}\left(  v\right)  -f_{\delta}\left(  v\right)  \right\Vert
_{L^{1}\left(  \mathbf{R}\right)  }+\ T\int_{\mathbf{R}}v^{2}\left\vert
f_{0}\left(  v\right)  -f_{\delta}\left(  v\right)  \right\vert
\ dv+\left\Vert f_{0}\left(  v\right)  -f_{\delta}\left(  v\right)
\right\Vert _{W^{2,p}\left(  \mathbf{R}\right)  }<\frac{\varepsilon}{2}.
\]
For $\delta\in\left(  \delta_{1}\left(  \varepsilon\right)  ,\delta_{2}\left(
\varepsilon\right)  \right)  ,\ $we consider bifurcation of steady BGK waves
near $\left(  f_{\delta}\left(  v\right)  ,0\right)  $, which are of the form
\begin{equation}
f_{\delta}^{\beta}\left(  x,v\right)  =\left\{
\begin{array}
[c]{cc}%
\frac{1}{\delta}g_{+}\left(  \frac{2e}{\delta^{2}}\right)  & \text{if }v>0\\
\frac{1}{\delta}g_{-}\left(  \frac{2e}{\delta^{2}}\right)  & \text{if }v\leq0
\end{array}
\right.  ,\text{ }e=\frac{1}{2}v^{2}-\beta\left(  x\right)  ,\ E=-\beta
_{x}\text{. } \label{defn-f-steady-case3}%
\end{equation}
The existence of BGK waves is then reduced to solve the ODE
\begin{equation}
\beta_{xx}=\int_{\mathbf{R}}f_{\delta}^{\beta}\left(  x,v\right)
\ dv-1:=h_{\delta}\left(  \beta\right)  \label{ode-beta-case3}%
\end{equation}
As in Case 1, for any $\delta\in\left(  \delta_{1}\left(  \varepsilon\right)
,\delta_{2}\left(  \varepsilon\right)  \right)  ,\ \exists\ r_{0}\left(
\varepsilon\right)  >0$ (independent of $\delta$) such that for each
$0<r<r_{0}\,$, there exists a $T\left(  \delta;r\right)  -$periodic solution
$\beta_{\delta;r}$ to the ODE $(\ref{ode-beta-case3})$ with $\left\Vert
\beta_{\delta;r}\right\Vert _{H^{2}\left(  0,T\left(  \delta;r\right)
\right)  }=r$. Moreover,
\[
\left(  \frac{2\pi}{T\left(  \delta;r\right)  }\right)  ^{2}\rightarrow
\int_{\mathbf{R}}\frac{f_{\delta}^{\prime}\left(  v\right)  }{v}dv\text{, when
}r\rightarrow0.
\]
So when $r$ is small enough, $T\left(  \delta_{1};r\right)  <T<T\left(
\delta_{2};r\right)  \ $and there exists $\delta_{T}\left(  r,\varepsilon
\right)  \in\left(  \delta_{1}\left(  \varepsilon\right)  ,\delta_{2}\left(
\varepsilon\right)  \right)  $ such that $T\left(  \delta_{T};r\right)  =T$.
Define $f_{r,\varepsilon}^{T}\left(  x,v\right)  =f_{\delta_{T}\left(
r,\varepsilon\right)  }^{\beta}\left(  x,v\right)  $ with $\beta=\beta
_{\delta_{T}\left(  r,\varepsilon\right)  ;r}$ in (\ref{defn-f-steady-case3})
and $E_{r,\varepsilon}\left(  x\right)  =-\beta_{\delta_{T}\left(
r,\varepsilon\right)  ;r}^{\prime}\left(  x\right)  $. Then $\left(
f_{r,\varepsilon}^{T}\left(  x,v\right)  ,E_{r,\varepsilon}\left(  x\right)
\right)  $ is a nontrivial BGK solution to ($\ref{vpe})$ with $x-$period $T$.
Let%
\[
\delta\left(  \varepsilon\right)  =\lim_{r\rightarrow0}\delta_{T}\left(
r,\varepsilon\right)  \in\left[  \delta_{1}\left(  \varepsilon\right)
,\delta_{2}\left(  \varepsilon\right)  \right]
\]
As in Case 1, by dominance convergence theorem, we can choose $r=r\left(
\varepsilon\right)  >0$ small enough, such that
\[
\left\Vert f_{r\left(  \varepsilon\right)  ,\varepsilon}^{T}\left(
x,v\right)  -f_{\delta\left(  \varepsilon\right)  }\left(  v\right)
\right\Vert _{L_{x,v}^{1}}+\ \int_{0}^{T}\int_{\mathbf{R}}v^{2}\left\vert
f_{r\left(  \varepsilon\right)  ,\varepsilon}^{T}\left(  x,v\right)
-f_{\delta\left(  \varepsilon\right)  }\left(  v\right)  \right\vert \ dxdv\
\]%
\[
+\left\Vert f_{r\left(  \varepsilon\right)  ,\varepsilon}^{T}\left(
x,v\right)  -f_{\delta\left(  \varepsilon\right)  }\left(  v\right)
\right\Vert _{W_{x,v}^{2,p}}<\frac{\varepsilon}{2}.
\]
Then
\[
\left(  f_{\varepsilon},E_{\varepsilon}\right)  =\left(  f_{r\left(
\varepsilon\right)  ,\varepsilon}^{T}\left(  x,v\right)  ,E_{r\left(
\varepsilon\right)  ,\varepsilon}\left(  x\right)  \right)
\]
is a steady BGK wave solution satisfying
\[
\left\Vert f_{\varepsilon}-f_{0}\right\Vert _{L_{x,v}^{1}}+\ \int_{0}^{T}%
\int_{\mathbf{R}}v^{2}\left\vert f_{0}-f_{\varepsilon}\right\vert
\ dxdv+\left\Vert f_{\varepsilon}-f_{0}\right\Vert _{W_{x,v}^{2,p}%
}<\varepsilon,
\]
which certainly implies (\ref{norm-proposition}). This finishes the proof of
the Proposition.
\end{proof}

To finish the proof of Theorem \ref{thm-existence}, we need the following
approximation result.

\begin{lemma}
\label{lemma-approxi}Fixed $p>1$, $0\leq s<1+\frac{1}{p}$ and $c\in\mathbf{R}%
$. Assume $f_{0}\in W^{s,p}\left(  \mathbf{R}\right)  ,\ f_{0}>0$,
$\int_{\mathbf{R}}f_{0}\left(  v\right)  dv=1$, and $\int_{\mathbf{R}}%
v^{2}f_{0}\left(  v\right)  dv<\infty$. Then for any $\varepsilon>0$, there
exists $f_{\varepsilon}\left(  v\right)  \in C^{\infty}\left(  \mathbf{R}%
\right)  \cap W^{2,p}\left(  \mathbf{R}\right)  $, such that $f_{\varepsilon}$
is even near$\ v=c,$ and
\[
\left\Vert f_{\varepsilon}-f_{0}\right\Vert _{L^{1}\left(  \mathbf{R}\right)
}+\int_{\mathbf{R}}v^{2}\left\vert f_{\varepsilon}-f_{0}\right\vert
\ dxdv+\left\Vert f_{\varepsilon}-f_{0}\right\Vert _{W^{s,p}\left(
\mathbf{R}\right)  }\leq\varepsilon.
\]

\end{lemma}

\begin{proof}
Let $\eta\left(  x\right)  $ be the standard mollifier function, that is,
\[
\eta\left(  x\right)  =\left\{
\begin{array}
[c]{cc}%
C\exp\left(  \frac{1}{x^{2}-1}\right)  & \ \ \ \text{if }\left\vert
x\right\vert <1\\
0 & \ \ \ \text{if }\left\vert x\right\vert \geq1
\end{array}
\right.  ,
\]
and $\eta_{\delta_{1}}\left(  x\right)  =\frac{1}{\delta_{1}}\eta\left(
\frac{x}{\delta_{1}}\right)  $. Define $f_{\delta_{1}}\left(  v\right)
:=\eta_{\delta_{1}}\left(  v\right)  \ast f_{0}\left(  v\right)  .$ Then by
the properties of mollifiers, we have%
\[
f_{\delta_{1}}\in C^{\infty}\left(  \mathbf{R}\right)  ,\ f_{\delta_{1}%
}\left(  v\right)  >0,\ \int_{\mathbf{R}}f_{\delta_{1}}\left(  v\right)
dv=1,
\]
and when $\delta_{1}$ is small enough
\begin{equation}
\left\Vert f_{\delta_{1}}-f_{0}\right\Vert _{L^{1}\left(  \mathbf{R}\right)
}+\int_{\mathbf{R}}v^{2}\left\vert f_{\delta_{1}}-f_{0}\right\vert
\ dxdv+\left\Vert f_{\delta_{1}}-f_{0}\right\Vert _{W^{s,p}\left(
\mathbf{R}\right)  }\leq\frac{\varepsilon}{2}. \label{estimate-delta-1}%
\end{equation}
We can assume $f_{\delta_{1}}\left(  v\right)  \in W^{2,p}$. Since otherwise,
we can modify $f_{\delta_{1}}\left(  v\right)  $ near infinity by cut-off to
get $\tilde{f}_{\delta_{1}}\left(  v\right)  \ $such that $\tilde{f}%
_{\delta_{1}}\left(  v\right)  \in W^{2,p}$ and
\[
\left\Vert f_{\delta_{1}}-\tilde{f}_{\delta_{1}}\right\Vert _{L^{1}\left(
\mathbf{R}\right)  }+\int_{\mathbf{R}}v^{2}\left\vert f_{\delta_{1}}-\tilde
{f}_{\delta_{1}}\right\vert \ dxdv+\left\Vert f_{\delta_{1}}-\tilde{f}%
_{\delta_{1}}\right\Vert _{W^{s,p}\left(  \mathbf{R}\right)  }\leq
\frac{\varepsilon}{2}.
\]
\ Solely to simplify notations, we set $c=0$ below$.$ Let $\sigma\left(
x\right)  =\sigma\left(  \left\vert x\right\vert \right)  $ to be the cut-off
function defined by (\ref{cut-off}). Let $\delta_{2}>0$ be a small number, and
define
\begin{align*}
f_{\delta_{1},\delta_{2}}\left(  v\right)   &  =f_{\delta_{1}}\left(
v\right)  \left(  1-\sigma\left(  \frac{v}{\delta_{2}}\right)  \right)
+\left(  \frac{f_{\delta_{1}}\left(  v\right)  +f_{\delta_{1}}\left(
-v\right)  }{2}\right)  \sigma\left(  \frac{v}{\delta_{2}}\right) \\
&  =f_{\delta_{1}}\left(  v\right)  -\left(  \frac{f_{\delta_{1}}\left(
v\right)  -f_{\delta_{1}}\left(  -v\right)  }{2}\right)  \sigma\left(
\frac{v}{\delta_{2}}\right)
\end{align*}
Then obviously,
\[
f_{\delta_{1},\delta_{2}}\in C^{\infty}\left(  \mathbf{R}\right)
,\ f_{\delta_{1},\delta_{2}}\left(  v\right)  >0,\ \int_{\mathbf{R}}%
f_{\delta_{1},\delta_{2}}\left(  v\right)  dv=\int_{\mathbf{R}}f_{\delta_{1}%
}\left(  v\right)  dv=1,
\]
and $f_{\delta_{1},\delta_{2}}\left(  v\right)  $ is even on the interval
$\left[  -\delta_{2},\delta_{2}\right]  $. Below, we prove that: when
$\delta_{2}$ is small enough
\begin{equation}
\left\Vert f_{\delta_{1}}-f_{0}\right\Vert _{L^{1}\left(  \mathbf{R}\right)
}+\int_{\mathbf{R}}v^{2}\left\vert f_{\delta_{1}}-f_{0}\right\vert
\ dv+\left\Vert f_{\delta_{1}}-f_{0}\right\Vert _{W^{s,p}\left(
\mathbf{R}\right)  }\leq\frac{\varepsilon}{2}. \label{estimate-delta-2}%
\end{equation}
Since%
\[
\left\Vert f_{\delta_{1}}-f_{\delta_{1},\delta_{2}}\right\Vert _{L^{1}}%
\leq\int_{\left\vert v\right\vert \leq2\delta_{2}}f_{\delta_{1}}\left(
v\right)  dv
\]%
\[
\int_{\mathbf{R}}v^{2}\left\vert f_{\delta_{1}}-f_{\delta_{1},\delta_{2}%
}\right\vert \ dv\leq\left(  2\delta_{2}\right)  ^{2}\int_{\left\vert
v\right\vert \leq2\delta_{2}}f_{\delta_{1}}\left(  v\right)  dv,
\]%
\[
\left\Vert f_{\delta_{1}}-f_{\delta_{1},\delta_{2}}\right\Vert _{L^{p}}%
\leq\left\Vert f_{\delta_{1}}\right\Vert _{L^{p}\left[  -2\delta_{2}%
,2\delta_{2}\right]  },
\]
and%
\[
\partial_{v}\left(  f_{\delta_{1}}-f_{\delta_{1},\delta_{2}}\right)  =\left(
\frac{f_{\delta_{1}}^{\prime}\left(  v\right)  +f_{\delta_{1}}^{\prime}\left(
-v\right)  }{2}\right)  \sigma\left(  \frac{v}{\delta_{2}}\right)
+\sigma^{\prime}\left(  \frac{v}{\delta_{2}}\right)  \frac{f_{\delta_{1}%
}\left(  v\right)  -f_{\delta_{1}}\left(  -v\right)  }{2\delta_{2}},
\]%
\begin{align*}
\left\Vert \partial_{v}\left(  f_{\delta_{1}}-f_{\delta_{1},\delta_{2}%
}\right)  \right\Vert _{L^{p}}  &  \leq\left\Vert f_{\delta_{1}}^{\prime
}\right\Vert _{L^{p}\left[  -2\delta_{2},2\delta_{2}\right]  }+\max\left\vert
\sigma^{\prime}\right\vert \left\Vert \frac{f_{\delta_{1}}\left(  v\right)
-f_{\delta_{1}}\left(  -v\right)  }{2\delta_{2}}\right\Vert _{L^{p}\left\{
\delta_{2}\leq\left\vert v\right\vert \leq2\delta_{2}\right\}  }\\
&  \leq\left\Vert f_{\delta_{1}}^{\prime}\right\Vert _{L^{p}\left[
-2\delta_{2},2\delta_{2}\right]  }+\frac{1}{2\delta_{2}}\max\left\vert
\sigma^{\prime}\right\vert \left\Vert \int_{-2\delta_{2}}^{2\delta_{2}%
}\left\vert f_{\delta_{1}}^{\prime}\right\vert dv\right\Vert _{L^{p}\left\{
\delta_{2}\leq\left\vert v\right\vert \leq2\delta_{2}\right\}  }\\
&  \leq\left\Vert f_{\delta_{1}}^{\prime}\right\Vert _{L^{p}\left[
-2\delta_{2},2\delta_{2}\right]  }+\frac{1}{2\delta_{2}}\max\left\vert
\sigma^{\prime}\right\vert \left(  4\delta_{2}\right)  ^{\frac{1}{p^{\prime}}%
}\left\Vert f_{\delta_{1}}^{\prime}\right\Vert _{L^{p}\left[  -2\delta
_{2},2\delta_{2}\right]  }\left(  2\delta_{2}\right)  ^{\frac{1}{p}}\\
&  \leq\left(  1+2\max\left\vert \sigma^{\prime}\right\vert \right)
\left\Vert f_{\delta_{1}}^{\prime}\right\Vert _{L^{p}\left[  -2\delta
_{2},2\delta_{2}\right]  },
\end{align*}
so when $\delta_{2}\rightarrow0$,
\[
\left\Vert f_{\delta_{1}}-f_{\delta_{1},\delta_{2}}\right\Vert _{L^{1}}%
+\int_{\mathbf{R}}v^{2}\left\vert f_{\delta_{1}}-f_{\delta_{1},\delta_{2}%
}\right\vert \ dv+\left\Vert f_{\delta_{1}}-f_{\delta_{1},\delta_{2}%
}\right\Vert _{W^{1,p}}\rightarrow0.
\]
Next, we show
\[
\left\Vert \partial_{v}\left(  f_{\delta_{1}}-f_{\delta_{1},\delta_{2}%
}\right)  \right\Vert _{W^{s-1,p}\left(  \mathbf{R}\right)  }\rightarrow
0\text{, when }\delta_{2}\rightarrow0\text{.}%
\]
This follows from Lemma \ref{lemma-strichartz} below, since $s-1<\frac{1}{p}$
and
\begin{align*}
&  \left\Vert \frac{f_{\delta_{1}}\left(  v\right)  -f_{\delta_{1}}\left(
-v\right)  }{2v}\right\Vert _{W^{s-1,p}\left(  \mathbf{R}\right)  }\\
&  =\left\Vert \int_{0}^{1}f_{\delta_{1}}^{\prime}\left(  \left(
2\tau-1\right)  v\right)  d\tau\right\Vert _{W^{s-1,p}\left(  \mathbf{R}%
\right)  }\leq\int_{0}^{1}\left\Vert f_{\delta_{1}}^{\prime}\left(  \left(
2\tau-1\right)  v\right)  \right\Vert _{W^{s-1,p}\left(  \mathbf{R}\right)
}d\tau\\
&  \leq C\int_{0}^{1}\left(  \left\vert 2\tau-1\right\vert ^{-\frac{1}{p}%
}\left\Vert f_{\delta_{1}}^{\prime}\right\Vert _{L^{p}}+\left\vert
2\tau-1\right\vert ^{-\left(  \frac{1}{p}-s+1\right)  }\left\Vert \left\vert
D\right\vert ^{s-1}f_{\delta_{1}}^{\prime}\right\Vert _{L^{p}}\right)
d\tau\leq C\left\Vert f_{\delta_{1}}\right\Vert _{W^{s,p}\left(
\mathbf{R}\right)  }.
\end{align*}
So when $\delta_{2}\rightarrow0$,
\[
\left\Vert f_{\delta_{1}}-f_{\delta_{1},\delta_{2}}\right\Vert _{L_{v}^{1}%
}+\int_{\mathbf{R}}v^{2}\left\vert f_{\delta_{1}}-f_{\delta_{1},\delta_{2}%
}\right\vert \ dv+\left\Vert f_{\delta_{1}}-f_{\delta_{1},\delta_{2}%
}\right\Vert _{W^{s,p}}\rightarrow0.
\]
Thus by choosing $\delta_{2}$ small enough, (\ref{estimate-delta-2}) is
satisfied. By setting $f_{\varepsilon}=f_{\delta_{1},\delta_{2}}$, the
conclusion of the lemma follows from (\ref{estimate-delta-1}) and
(\ref{estimate-delta-2}).
\end{proof}

\begin{lemma}
\label{lemma-strichartz}Given $f\in W^{\frac{1}{p},p}\left(  \mathbf{R}%
\right)  \cap L^{\infty},\ g\in W^{s,p}\left(  \mathbf{R}\right)  \left(
p>1,0\leq s<\frac{1}{p}\right)  $, then for any $\delta>0,\ f\left(  \frac
{x}{\delta}\right)  g\left(  x\right)  \in W^{s,p}\left(  \mathbf{R}\right)  $
and
\begin{equation}
\left\Vert f\left(  \frac{x}{\delta}\right)  g\right\Vert _{W^{s,p}%
}\rightarrow0\text{, when }\delta\rightarrow0\text{. } \label{zero-lemma}%
\end{equation}

\end{lemma}

\begin{proof}
First, we cite a result of Strichartz (\cite{strichartz}): Given $h_{1}\in
W^{\frac{1}{p},p}\left(  \mathbf{R}\right)  \cap L^{\infty},\ h_{2}\in
W^{s,p}\left(  \mathbf{R}\right)  ,\ $then $h_{1}h_{2}\in W^{s,p}\left(
\mathbf{R}\right)  $ and
\[
\left\Vert h_{1}h_{2}\right\Vert _{W^{s,p}}\leq C\left(  \left\Vert
h_{1}\right\Vert _{W^{\frac{1}{p},p}},\left\Vert h_{1}\right\Vert _{L^{\infty
}}\right)  \left\Vert h_{2}\right\Vert _{W^{s,p}}\text{.}%
\]
Above result immediately implies that $f\left(  \frac{x}{\delta}\right)
g\left(  x\right)  \in W^{s,p}\left(  \mathbf{R}\right)  $. To show
(\ref{zero-lemma}), for any $\varepsilon>0,\ $we pick $g_{1}\in C_{0}^{\infty
}\left(  \mathbf{R}\right)  $ such that $\left\Vert g-g_{1}\right\Vert
_{W^{s,p}}<\varepsilon$. Since
\[
\left\Vert f\left(  \frac{x}{\delta}\right)  \right\Vert _{L^{p}}+\left\Vert
\left\vert D\right\vert ^{\frac{1}{p}}\left(  f\left(  \frac{x}{\delta
}\right)  \right)  \right\Vert _{L^{p}}=\delta^{\frac{1}{p}}\left\Vert
f\left(  x\right)  \right\Vert _{L^{p}}+\left\Vert \left\vert D\right\vert
^{\frac{1}{p}}f\right\Vert _{L^{p}},
\]
so when $\delta\leq1$,
\[
\left\Vert f\left(  \frac{x}{\delta}\right)  \right\Vert _{W^{\frac{1}{p},p}%
}\leq C\left\Vert f\right\Vert _{W^{\frac{1}{p},p}},\ \text{for some }C\text{
independent of }\delta\text{.}%
\]
Thus
\begin{align*}
&  \ \ \ \ \ \left\Vert f\left(  \frac{x}{\delta}\right)  g\right\Vert
_{W^{s,p}}\\
&  \leq\left\Vert f\left(  \frac{x}{\delta}\right)  \left(  g-g_{1}\right)
\right\Vert _{W^{s,p}}+\left\Vert f\left(  \frac{x}{\delta}\right)
g_{1}\right\Vert _{W^{s,p}}\\
&  \leq C\left(  \left\Vert f\right\Vert _{W^{\frac{1}{p},p}},\left\Vert
f\right\Vert _{L^{\infty}}\right)  \left\Vert g-g_{1}\right\Vert _{W^{s,p}%
}+\left\Vert f\left(  \frac{x}{\delta}\right)  \right\Vert _{W^{s,p}}C\left(
\left\Vert g_{1}\right\Vert _{W^{\frac{1}{p},p}},\left\Vert g_{1}\right\Vert
_{L^{\infty}}\right) \\
&  \leq C\left(  \left\Vert f\right\Vert _{W^{\frac{1}{p},p}},\left\Vert
f\right\Vert _{L^{\infty}}\right)  \varepsilon+\left(  \delta^{\frac{1}{p}%
}\left\Vert f\left(  x\right)  \right\Vert _{L^{p}}+\delta^{\frac{1}{p}%
-s}\left\Vert \left\vert D\right\vert ^{\frac{1}{p}}f\right\Vert _{L^{p}%
}\right)  C\left(  \left\Vert g_{1}\right\Vert _{W^{\frac{1}{p},p}},\left\Vert
g_{1}\right\Vert _{L^{\infty}}\right)  .
\end{align*}
Letting $\delta\rightarrow0$, we get
\[
\lim_{\delta\rightarrow0}\left\Vert f\left(  \frac{x}{\delta}\right)
g\right\Vert _{W^{s,p}}\leq C\left(  \left\Vert f\right\Vert _{W^{\frac{1}%
{p},p}},\left\Vert f\right\Vert _{L^{\infty}}\right)  \varepsilon.
\]
Since $\varepsilon$ is arbitrarily small, (\ref{zero-lemma}) is proved$.$
\end{proof}

\begin{proof}
[Proof of Theorem \ref{thm-existence}]Fixed the period $T>0$ and the travel
speed $c\in\mathbf{R}.$Then by Lemma \ref{lemma-approxi}, for any
$\varepsilon$ small enough, there exists $f_{1}\left(  v\right)  \in
C^{\infty}\left(  \mathbf{R}\right)  \cap W^{2,p}\left(  \mathbf{R}\right)  $,
such that $f_{1}\left(  v\right)  $ is even near$\ v=c$ and
\[
\left\Vert f_{1}-f_{0}\right\Vert _{L^{1}\left(  \mathbf{R}\right)  }%
+T\int_{\mathbf{R}}v^{2}\left\vert f_{1}-f_{0}\right\vert \ dv+\left\Vert
f_{1}-f_{0}\right\Vert _{W^{s,p}\left(  \mathbf{R}\right)  }\leq
\varepsilon/2.
\]
Our goal is to construct travelling BGK wave solutions of the form
\[
\left(  f_{\varepsilon}\left(  x-ct,v\right)  ,E_{\varepsilon}\left(
x-ct\right)  \right)
\]
$\ $near $\left(  f_{1}\left(  v\right)  ,0\right)  $, such that
\[
\left\Vert f_{\varepsilon}\left(  x,v\right)  -f_{1}\left(  v\right)
\right\Vert _{L_{x,v}^{1}}+\int_{\mathbf{R}}v^{2}\left\vert f_{\varepsilon
}\left(  x,v\right)  -f_{1}\left(  v\right)  \right\vert \ dxdv+\left\Vert
f_{\varepsilon}\left(  x,v\right)  -f_{1}\left(  v\right)  \right\Vert
_{W_{x,v}^{s,p}}<\frac{\varepsilon}{2}.
\]
It is equivalent to find steady BGK solutions $\left(  f_{\varepsilon}\left(
x,v+c\right)  ,E_{\varepsilon}\left(  x\right)  \right)  $ near $\left(
f_{1}\left(  v+c\right)  ,0\right)  $. $\ $By Proposition
\ref{propo-existence}, there exists steady BGK solution $\left(  f_{2}\left(
x,v\right)  ,E_{2}\left(  x\right)  \right)  $ near $\left(  f_{1}\left(
v+c\right)  ,0\right)  $ such that$\ E_{2}\left(  x\right)  $ not identically
$0,$
\[
\left\Vert f_{2}\left(  x,v\right)  -f_{1}\left(  v+c\right)  \right\Vert
_{L_{x,v}^{1}}+\int_{\mathbf{R}}v^{2}\left\vert f_{2}\left(  x,v\right)
-f_{1}\left(  v+c\right)  \right\vert \ dxdv
\]%
\[
+\left\Vert f_{2}\left(  x,v\right)  -f_{1}\left(  v+c\right)  \right\Vert
_{W_{x,v}^{s,p}}<\frac{\varepsilon}{2\left(  5+4c^{2}\right)  }.
\]
Setting
\[
f_{\varepsilon}\left(  x,v\right)  =f_{2}\left(  x,v-c\right)
,\ E_{\varepsilon}\left(  x\right)  =E_{2}\left(  x\right)  ,
\]
then $\left(  f_{\varepsilon}\left(  x-ct,v\right)  ,E_{\varepsilon}\left(
x-ct\right)  \right)  $ is a travelling BGK solution and
\[
\left\Vert f_{\varepsilon}-f_{1}\left(  v\right)  \right\Vert _{L_{x,v}^{1}%
}+\int_{\mathbf{R}}\left(  v-c\right)  ^{2}\left\vert f_{\varepsilon}%
-f_{1}\left(  v\right)  \right\vert \ dxdv
\]%
\[
+\left\Vert f_{\varepsilon}-f_{1}\left(  v\right)  \right\Vert _{W_{x,v}%
^{s,p}}<\frac{\varepsilon}{2\left(  5+4c^{2}\right)  }.
\]
Since$\ \left\vert v-c\right\vert \geq\left\vert v\right\vert /2\ $when
$\left\vert v\right\vert \geq2\left\vert c\right\vert ,$so$\ $%
\begin{align*}
&  \ \ \ \ \int_{\mathbf{R}}v^{2}\left\vert f_{\varepsilon}\left(  x,v\right)
-f_{1}\left(  v\right)  \right\vert \ dxdv\\
&  \leq\int_{\left\vert v\right\vert \geq2\left\vert c\right\vert }%
v^{2}\left\vert f_{\varepsilon}\left(  x,v\right)  -f_{1}\left(  v\right)
\right\vert \ dxdv+\int_{\left\vert v\right\vert \leq2\left\vert c\right\vert
}v^{2}\left\vert f_{\varepsilon}\left(  x,v\right)  -f_{1}\left(  v\right)
\right\vert \ dxdv\\
&  \leq4\int\left(  v-c\right)  ^{2}\left\vert f_{\varepsilon}\left(
x,v\right)  -f_{1}\left(  v\right)  \right\vert \ dxdv+4c^{2}\left\Vert
f_{\varepsilon}-f_{1}\right\Vert _{L_{x,v}^{1}}\\
&  <\frac{\left(  4+4c^{2}\right)  \varepsilon}{2\left(  5+4c^{2}\right)  },
\end{align*}
and thus
\begin{align*}
\ \ \  &  \left\Vert f_{\varepsilon}-f_{1}\left(  v\right)  \right\Vert
_{L_{x,v}^{1}}+\int_{\mathbf{R}}v^{2}\left\vert f_{\varepsilon}\left(
x,v\right)  -f_{1}\left(  v\right)  \right\vert \ dxdv+\left\Vert
f_{\varepsilon}\left(  x,v\right)  -f_{1}\left(  v\right)  \right\Vert
_{W_{x,v}^{s,p}}\\
&  <\frac{\varepsilon}{2\left(  5+4c^{2}\right)  }+\frac{\left(
4+4c^{2}\right)  \varepsilon}{2\left(  5+4c^{2}\right)  }=\frac{\varepsilon
}{2},
\end{align*}
So
\[
\left\Vert f_{\varepsilon}-f_{0}\left(  v\right)  \right\Vert _{L_{x,v}^{1}%
}+\int_{\mathbf{R}}v^{2}\left\vert f_{\varepsilon}\left(  x,v\right)
-f_{0}\left(  v\right)  \right\vert \ dxdv+\left\Vert f_{\varepsilon}\left(
x,v\right)  -f_{0}\left(  v\right)  \right\Vert _{W_{x,v}^{s,p}}<\varepsilon.
\]
and the proof of Theorem \ref{thm-existence} is finished.
\end{proof}

\begin{remark}
\label{rmk-double-period}For steady BGK waves $\left(  f\left(  x,v\right)
,E\left(  x\right)  \right)  $ of the form $E\left(  x\right)  =-\beta_{x}$
and
\begin{equation}
f\left(  x,v\right)  =\left\{
\begin{array}
[c]{cc}%
\mu^{+}\left(  e\right)  & \text{if }v\geq0\\
\mu^{-}\left(  e\right)  & \text{if\ \thinspace}v<0
\end{array}
\right\}  ,\ \text{ }e=\frac{1}{2}v^{2}-\beta\left(  x\right)  ,
\label{form-f}%
\end{equation}
with $\mu^{+},\mu^{-}\in C^{1}\left(  \mathbf{R}\right)  ,$ such as
constructed in the proof of Theorem \ref{thm-existence}, $E\left(  x\right)  $
has only two zeros in one minimal period. This is because the electric
potential $\beta$ satisfying the 2nd order autonomous ODE
\begin{equation}
\beta_{xx}=\int_{v\geq0}\mu^{+}\left(  \frac{1}{2}v^{2}-\beta\right)
\ dv+\int_{v<0}\mu^{-}\left(  \frac{1}{2}v^{2}-\beta\right)  \ dv-1=h\left(
\beta\right)  \label{ode-beta-remark}%
\end{equation}
with $h\in C^{1}\left(  \mathbf{R}\right)  .$ Any periodic solution of minimal
period to the ODE (\ref{ode-beta-remark}) has only one minimum and maximum,
and therefore $E=-\beta_{x}$ vanishes at only two points. By Theorem
\ref{thm-existence}, for $T>0$, near any homogeneous equilibria we can
construct small BGK waves such that multiple of its minimal period equal $T$.
By \cite{lin01} and \cite{lin-cpam}, any of such multi-BGK waves are linearly
and nonlinearly unstable under perturbations of period $T$. So far, the
existence of stable BGK wave of minimal period remains open, although some
numerical evidences suggest the existence of such stable BGK wave. For
example, in \cite{bertrand-et-1988} starting near a unstable multi-BGK wave,
numerical simulations shows that the long time asymptotics is to tend to a
seemingly stable BGK wave of minimal period.
\end{remark}

\begin{remark}
In (\cite{dorning-holloway} \cite{dorning-holloway2}), Dorning and Holloway
(see also \cite{buchanan-dorning95}, \cite{demeio-holloway}) studied the
bifurcation of small travelling BGK waves with speed $v_{p}\ $near homogeneous
equilibria $\left(  f_{0}\left(  v\right)  ,0\right)  \ $under the bifurcation
condition
\begin{equation}
\nu\left(  v_{p}\right)  =P\int\frac{f_{0}^{\prime}\left(  v\right)  }%
{v-v_{p}}dv>0, \label{bifurcation-holloway}%
\end{equation}
where $P$ denotes the principal value integral. It is equivalent to find
steady BGK waves near $\left(  f_{0}\left(  v+v_{p}\right)  ,0\right)  $. The
approach in (\cite{dorning-holloway} \cite{dorning-holloway2}) is as follows.
Define
\begin{align*}
f^{e,v_{p}}\left(  v\right)   &  =\frac{1}{2}\left(  f_{0}\left(
v+v_{p}\right)  +f_{0}\left(  -v+v_{p}\right)  \right)  ,\\
\ f^{o,v_{p}}\left(  v\right)   &  =\frac{1}{2}\left(  f_{0}\left(
v+v_{p}\right)  -f_{0}\left(  -v+v_{p}\right)  \right)  \ .
\end{align*}
Then
\[
\int\frac{\frac{d}{dv}f^{e,v_{p}}\left(  v\right)  }{v}dv=P\int\frac
{f_{0}^{\prime}\left(  v\right)  }{v-v_{p}}dv=\nu\left(  v_{p}\right)  >0.
\]
So by the bifurcation theory, there exist small BGK waves $\left(
f^{e}\left(  x,v\right)  ,-\beta_{x}\right)  \ $near $\left(  f^{e,v_{p}%
}\left(  v\right)  ,0\right)  $ with periods close to$\frac{2\pi}{\sqrt
{\nu\left(  v_{p}\right)  }}$, and $f^{e}\left(  x,v\right)  $ is even in $v$.
Next, the odd part $f^{o,v_{p}}\left(  x,v\right)  $ is defined by
\[
f^{o,v_{p}}\left(  x,v\right)  =\left(  1-\sigma\left(  \frac{e}{-2\min\beta
}\right)  \right)  \left\{
\begin{array}
[c]{cc}%
G^{o}\left(  e\right)  & \text{if }v\geq0\\
-G^{o}\left(  e\right)  & if\ v<0
\end{array}
\right.  ,
\]
where $\sigma\left(  x\right)  $ is the cut-off function as defined in
(\ref{cut-off}) and $G^{o}\left(  e\right)  =\ f^{o,v_{p}}\left(  \sqrt
{2e}\right)  $ when $e>0$. Define
\[
f\left(  x,v\right)  =f^{e,v_{p}}\left(  x,v\right)  +f^{o,v_{p}}\left(
x,v\right)  ,
\]
then $\left(  f\left(  x,v\right)  ,-\beta_{x}\right)  $ is a steady BGK wave,
since for trapped particles with $e<-\min\beta$, $f^{o,v_{p}}\left(
x,v\right)  =0$ and $f\left(  x,v\right)  $ only depends on $e$. It can be
shown that $\left(  f\left(  x,v\right)  ,-\beta_{x}\right)  \ $is close to
$\left(  f_{0}\left(  v+v_{p}\right)  ,0\right)  $ in $L_{x,v}^{p}$ norm. The
periods of the BGK waves constructed above are only near $\frac{2\pi}%
{\sqrt{\nu\left(  v_{p}\right)  }}$. In \cite{dorning-holloway}
\cite{dorning-holloway2}, \cite{demeio-holloway}, it was suggested that BGK
waves with exact period$\frac{2\pi}{\sqrt{\nu\left(  v_{p}\right)  }}\ $and
$\varepsilon-$close to $\left(  f_{0}\left(  v+v_{p}\right)  ,0\right)  $ in
$L_{x,v}^{p}$ norm can be constructed by performing above bifurcation from
$\left(  \left(  1+\mu\left(  \varepsilon\right)  \right)  f_{0}\left(
v+v_{p}\right)  ,0\right)  $ for proper small parameter $\mu$. It should be
pointed out that this strategy actually does not work to get exact period
$\frac{2\pi}{\sqrt{\nu\left(  v_{p}\right)  }}$. Since to ensure that $\left(
\left(  1+\mu\left(  \varepsilon\right)  \right)  f^{e,v_{p}}\left(  v\right)
,0\right)  $ is a bifurcation point, it is required that
\[
\int_{\mathbf{R}}\left(  1+\mu\left(  \varepsilon\right)  \right)  f^{e,v_{p}%
}\left(  v\right)  dv=1,
\]
and thus $\mu\left(  \varepsilon\right)  =0$, i,e, $\mu$ is not adjustable at all.

Second, by Lemma \ref{lemma-hardy} below,
\[
\left\vert \nu\left(  v_{p}\right)  \right\vert =\left\vert \int\frac{\frac
{d}{dv}f^{e,v_{p}}\left(  v\right)  }{v}dv\right\vert \leq\left\Vert
f^{e,v_{p}}\left(  v\right)  \right\Vert _{W^{2,p}}=\left\Vert f_{0}\left(
v\right)  \right\Vert _{W^{2,p}}\text{.}%
\]
So by the method in \cite{dorning-holloway} \cite{dorning-holloway2}, one can
not get small BGK waves with spatial periods less than $2\pi/\sqrt{\left\Vert
f_{0}\left(  v\right)  \right\Vert _{W^{2,p}}}$. By comparison, we construct
\ BGK waves with any minimal period near any homogeneous equilibrium $\left(
f_{0}\left(  v\right)  ,0\right)  $ in any $W^{s,p}$ $\left(  s<1+\frac{1}%
{p}\right)  \ $neighborhood.

It is also claimed in \cite{dorning-holloway} \cite{dorning-holloway2} that
for $v_{p}$ such that $\nu\left(  v_{p}\right)  <0$, there exist no travelling
BGK waves with travel speed $v_{p}$, arbitrarily near $\left(  f_{0}\left(
v\right)  ,0\right)  $. For Maxwellian $f_{0}\left(  v\right)  =e^{-\frac
{1}{2}v^{2}}$, the critical speed is about $v_{p}=1.35$ since $\nu\left(
v_{p}\right)  <0$ when $v_{p}<1.35.$ However, by our Theorem
\ref{thm-existence}, BGK waves with arbitrary travel speed exist near (in
$W^{s,p}$ space,$\ s<1+\frac{1}{p}$) any homogeneous equilibrium including
Maxwellian. So the claim of the critical travel speed based on
(\ref{bifurcation-holloway}) is not true.
\end{remark}

\begin{proof}
[Proof of Corollary \ref{cor-unsatble}]From the proof of Theorem
\ref{thm-existence} and Proposition \ref{propo-existence}, it follows that:
Fixed $T>0,\ $for any $\varepsilon>0$, there exists a homogeneous profile
$f_{\varepsilon}\left(  v\right)  \in C^{\infty}\left(  \mathbf{R}\right)
\cap W^{2,p}\left(  \mathbf{R}\right)  ,\ $such that$\ f_{\varepsilon}\left(
v\right)  \geq0,\int_{\mathbf{R}}f_{\varepsilon}\left(  v\right)  \ dv=1,$
\[
\ \left(  \frac{2\pi}{T}\right)  ^{2}=k_{0}^{2}=\int_{\mathbf{R}}%
\frac{f_{\varepsilon}^{\prime}\left(  v\right)  }{v-v_{\varepsilon}}dv\text{
\ with \ }f_{\varepsilon}^{\prime}\left(  v_{\varepsilon}\right)  =0,
\]
and
\begin{equation}
\left\Vert f_{\varepsilon}\left(  v\right)  -f_{0}\left(  v\right)
\right\Vert _{L^{1}\left(  \mathbf{R}\right)  }+\ T\int_{\mathbf{R}}%
v^{2}\left\vert f_{\varepsilon}\left(  v\right)  -\ f_{0}\left(  v\right)
\right\vert dv+\left\Vert f_{\varepsilon}\left(  v\right)  -f_{0}\left(
v\right)  \right\Vert _{W^{s,p}\left(  \mathbf{R}\right)  }<\frac{\varepsilon
}{2}. \label{ineqn1}%
\end{equation}
Define $f_{\delta}\left(  v\right)  \in C^{\infty}\left(  \mathbf{R}\right)
\cap W^{2,p}\left(  \mathbf{R}\right)  \ $by%
\[
f_{\delta}\left(  v\right)  =\frac{1}{\delta}f_{\varepsilon}\left(
v_{\varepsilon}+\frac{v-v_{\varepsilon}}{\delta}\right)  ,
\]
Then$\ f_{\delta}\left(  v\right)  \geq0,\ \int_{\mathbf{R}}f_{\delta}\left(
v\right)  \ dv=1$ and
\[
k_{0}\left(  \delta\right)  ^{2}=\int\frac{f_{\delta}^{\prime}\left(
v\right)  }{v-v_{\varepsilon}}dv=\frac{1}{\delta^{2}}\left(  \frac{2\pi}%
{T}\right)  ^{2}.
\]
We consider two cases below.

Case 1: $f_{\varepsilon}^{\prime\prime}\left(  v_{\varepsilon}\right)  >0$. By
Lemma \ref{lemma-penrose} and Remark \ref{rmk-penrose} thereafter,$\ $there
exist unstable modes of the linearized VP equation around $\left(
f_{\varepsilon}\left(  v\right)  ,0\right)  ,\ $for wave numbers $k$ in the
internal $\left(  k_{1},k_{0}\right)  $. Here $k_{1}$ is defined by
\[
\ k_{1}^{2}=\int_{\mathbf{R}}\frac{f_{\varepsilon}^{\prime}\left(  v\right)
}{v-c_{1}}dv,
\]
and $c_{1}\ $is a maximum point $f_{\varepsilon}\left(  v\right)  $. If there
is no maximum point $c_{1}\ $of $f_{\varepsilon}$ such that
\[
\int_{\mathbf{R}}\frac{f_{\varepsilon}^{\prime}\left(  v\right)  }{v-c_{1}%
}dv<k_{0}^{2},
\]
$\ $then $k_{1}=0$. Choose $\delta<1$ such that
\begin{equation}
\left\Vert f_{\varepsilon}\left(  v\right)  -f_{\delta}\left(  v\right)
\right\Vert _{L^{1}\left(  \mathbf{R}\right)  }+\ T\int_{\mathbf{R}}%
v^{2}\left\vert f_{\varepsilon}-\ f_{\delta}\right\vert dv+\left\Vert
f_{\varepsilon}-f_{\delta}\right\Vert _{W^{s,p}\left(  \mathbf{R}\right)
}<\frac{\varepsilon}{2}. \label{ineqn2}%
\end{equation}
Then again by Lemma \ref{lemma-penrose} and Remark \ref{rmk-penrose}, there
exist unstable modes of the linearized VP equation around $\left(  f_{\delta
}\left(  v\right)  ,0\right)  ,\ $for wave numbers $k$ in the internal
$\left(  k_{1}\left(  \delta\right)  ,k_{0}\left(  \delta\right)  \right)  $.
Since $k_{0}\left(  \delta\right)  >k_{0}$ and $k_{0}\left(  \delta\right)
-k_{1}\left(  \delta\right)  \rightarrow k_{0}-k_{1}>0$ when $\delta
\rightarrow1-$, we have $k_{0}\in\left(  k_{1}\left(  \delta\right)
,k_{0}\left(  \delta\right)  \right)  \ $when $\delta$ is close enough to $1$.
This implies that $\left(  f_{\delta}\left(  v\right)  ,0\right)  \ $is
linearly unstable under perturbations of period $T$. Moreover, the
inequalities (\ref{ineqn1}) and (\ref{ineqn2}) imply that
\[
\left\Vert f_{\delta}\left(  v\right)  -f_{0}\left(  v\right)  \right\Vert
_{L^{1}\left(  \mathbf{R}\right)  }+\ T\int_{\mathbf{R}}v^{2}\left\vert
f_{\delta}\left(  v\right)  -\ f_{0}\left(  v\right)  \right\vert
dv+\left\Vert f_{\delta}\left(  v\right)  -f_{0}\left(  v\right)  \right\Vert
_{W^{s,p}\left(  \mathbf{R}\right)  }<\varepsilon.
\]

Case 2: $f_{\varepsilon}^{\prime\prime}\left(  v_{\varepsilon}\right)  <0$.
Choose $\delta>1$ sufficiently close to $1$, then by the same argument as in
Case 1, $\left(  f_{\delta}\left(  v\right)  ,0\right)  \ $is linearly
unstable under perturbations of period $T$ and
\[
\left\Vert f_{\delta}\left(  v\right)  -f_{0}\left(  v\right)  \right\Vert
_{L^{1}\left(  \mathbf{R}\right)  }+\ T\int_{\mathbf{R}}v^{2}\left\vert
f_{\delta}\left(  v\right)  -\ f_{0}\left(  v\right)  \right\vert
dv+\left\Vert f_{\delta}\left(  v\right)  -f_{0}\left(  v\right)  \right\Vert
_{W^{s,p}\left(  \mathbf{R}\right)  }<\varepsilon.
\]
This finishes the proof of Corollary \ref{cor-unsatble}.
\end{proof}

\section{Nonexistence of BGK waves in $W^{s,p}\left(  s>1+\frac{1}{p}\right)
$}

In this Section, we prove Theorem \ref{thm-non-existence}. The next lemma is a
Hardy type inequality.

\begin{lemma}
\label{lemma-hardy} If $u\left(  v\right)  \in W^{s,p}\left(  \mathbf{R}%
\right)  $ $\left(  p>1,s>\frac{1}{p}\right)  ,\ $and $u\left(  0\right)
=0,\ $then%
\[
\int_{\mathbf{R}}\left\vert \frac{u\left(  v\right)  }{v}\right\vert dv\leq
C\left\Vert u\right\Vert _{W^{s,p}\left(  \mathbf{R}\right)  },
\]
for some constant $C$.
\end{lemma}

\begin{proof}
Since$\ s>\frac{1}{p}$, the space $W^{s,p}\left(  \mathbf{R}\right)  $ is
embedded to the H\"{o}lder space $C^{0,\alpha}$ with $\alpha\in\left(
0,s-\frac{1}{p}\right)  $. So
\[
\left\vert u\left(  v\right)  \right\vert =\left\vert u\left(  v\right)
-u\left(  0\right)  \right\vert \leq\left\vert v\right\vert ^{\alpha
}\left\Vert u\right\Vert _{C^{0,\alpha}}\leq C\left\Vert u\right\Vert
_{W^{s,p}}\left\vert v\right\vert ^{\alpha}%
\]
and thus%
\begin{align*}
\int_{\mathbf{R}}\left\vert \frac{u\left(  v\right)  }{v}\right\vert dv  &
\leq\int_{-1}^{1}\left\vert \frac{u\left(  v\right)  }{v}\right\vert
dv+\int_{\left\vert v\right\vert \geq1}\left\vert \frac{u\left(  v\right)
}{v}\right\vert dv\\
&  \leq C\left\Vert u\right\Vert _{W^{s,p}}\int_{-1}^{1}\left\vert
v\right\vert ^{-1+\alpha}dv+\left(  \int_{\left\vert v\right\vert \geq1}%
\frac{1}{\left\vert v\right\vert ^{p^{\prime}}}dv\right)  ^{\frac{1}%
{p^{\prime}}}\left\Vert u\right\Vert _{L^{p}}\\
&  \leq C\left\Vert u\right\Vert _{W^{s,p}\left(  \mathbf{R}\right)  }\text{.}%
\end{align*}

\end{proof}

\begin{proof}
[Proof of Theorem \ref{thm-non-existence}]Suppose otherwise, then there exist
a sequence $\varepsilon_{n}\rightarrow0$, and nontrivial travelling wave
solutions
\[
\left(  f_{n}\left(  x-c_{n}t,v\right)  ,E_{n}\left(  x-c_{n}t\right)
\right)
\]
$\ $to (\ref{vpe}) such that $E_{n}\left(  x\right)  $ is not identically
zero$,\ \int_{0}^{T}E_{n}\left(  x\right)  dx=0,\ f_{n}\left(  x,v\right)  $
and $\beta_{n}\left(  x\right)  $ are $T-$periodic in $x$,
\[
\int_{0}^{T}\int_{\mathbf{R}}v^{2}f_{n}\left(  x,v\right)  dvdx<\infty
\ \text{and }\left\Vert f_{n}\left(  x,v\right)  -f_{0}\left(  v\right)
\right\Vert _{W_{x,v}^{s,p}}<\varepsilon_{n}.
\]
The travelling BGK waves satisfy%
\begin{equation}
\left(  v-c_{n}\right)  \partial_{x}f_{n}-E_{n}\partial_{v}f_{n}=0,
\label{eqn-vlasov-n}%
\end{equation}
and
\begin{equation}
\frac{\partial E_{n}}{\partial x}=-\int_{-\infty}^{+\infty}f_{n}dv+1.
\label{eqn-poisson-n}%
\end{equation}
Because $f_{n}\in W_{x,v}^{s,p}$ with $s>1+\frac{1}{p}>\frac{2}{p},$ so by
Sobolev embedding
\[
\left\Vert f_{n}\right\Vert _{L_{x,v}^{\infty}}\leq C\left\Vert f_{n}%
\right\Vert _{W_{x,v}^{s,p}}<\infty.
\]
By a standard estimate in kinetic theory,
\begin{equation}
\rho_{n}=\int f_{n}dv\leq\left\Vert f_{n}\right\Vert _{L_{x,v}^{\infty}%
}^{\frac{2}{3}}\left(  \int v^{2}f_{n}dv\right)  ^{\frac{1}{3}}
\label{estimate-rho-n}%
\end{equation}
and thus $\rho_{n}\in L^{3}\left(  0,T\right)  $. So $E_{n}\left(  x\right)
\in W^{1,3}\left(  0,T\right)  $ which implies that $E_{n}\left(  x\right)
\in H^{1}\left(  0,T\right)  $ and $E_{n}\left(  x\right)  $ is absolutely
continuous. Define two sets $\mathbf{P}_{n}=\left\{  E_{n}\neq0\right\}  $ and
$\mathbf{Q}_{n}=\left\{  E_{n}=0\right\}  $. Then $\mathbf{P}_{n}$ is of
non-zero measure and $E_{n}^{\prime}=0\ $a.e. on $\mathbf{Q}_{n}.$ Thus we
have
\begin{equation}
\int_{0}^{T}\left\vert E_{n}^{\prime}\left(  x\right)  \right\vert
^{2}dx=-\int_{0}^{T}\rho_{n}\left(  x\right)  E_{n}^{\prime}\left(  x\right)
dx=-\int_{\mathbf{P}_{n}}\rho_{n}\left(  x\right)  E_{n}^{\prime}\left(
x\right)  dx. \label{eqn-e_n'-square}%
\end{equation}
Since $s-1>\frac{1}{p}$, by the trace theorem for fractional Sobolev Space,
\[
\partial_{x}f_{n}|_{v=c_{n}},\ \ \ \partial_{v}f_{n}|_{v=c_{n}}\in
L^{p}\left(  0,T\right)  .
\]
So from equation (\ref{eqn-vlasov-n}), $\partial_{v}f_{n}|_{v=c_{n}}=0$ for
a.e. $x\in\mathbf{P}_{n}$. By Lemma \ref{lemma-hardy}, for a.e. $x\in
\mathbf{P}_{n},$%
\[
\left\vert \int\frac{\partial_{v}f_{n}}{v-c_{n}}dv\right\vert \left(
x\right)  \leq C\left\Vert f_{n}\left(  x,v\right)  \right\Vert _{W_{v}^{s,p}%
}\in L^{p}\left(  \mathbf{P}_{n}\right)  .
\]
From (\ref{eqn-vlasov-n}), when $x\in\mathbf{P}_{n},$%
\[
\rho_{n}^{\prime}\left(  x\right)  =\int\frac{\partial_{v}f_{n}}{v-c_{n}%
}dvE_{n}\left(  x\right)  \in L^{p}\left(  \mathbf{P}_{n}\right)
\]
and it follows from (\ref{eqn-e_n'-square}) that
\begin{equation}
\int_{0}^{T}\left\vert E_{n}^{\prime}\left(  x\right)  \right\vert ^{2}%
dx-\int_{\mathbf{P}_{n}}\int\frac{\partial_{v}f_{n}}{v-c_{n}}dvE_{n}\left(
x\right)  ^{2}dx=0. \label{eqn-e_n'-square-2}%
\end{equation}
Denote $\left\vert \mathbf{P}_{n}\right\vert $ to be the measure of the set
$\mathbf{P}_{n}$. We consider two cases.

Case 1: $\left\vert \mathbf{P}_{n}\right\vert \rightarrow0$ when
$n\rightarrow\infty$. Since
\[
\left\Vert E_{n}\right\Vert _{L^{\infty}\left(  0,T\right)  }\leq\left\Vert
E_{n}^{\prime}\right\Vert _{L^{1}\left(  0,T\right)  }\leq\sqrt{T}\left\Vert
E_{n}^{\prime}\right\Vert _{L^{2}\left(  0,T\right)  },
\]
so from (\ref{eqn-e_n'-square-2}),
\begin{align*}
\left\Vert E_{n}^{\prime}\right\Vert _{L^{2}\left(  0,T\right)  }^{2}  &  \leq
T\left\Vert E_{n}^{\prime}\right\Vert _{L^{2}\left(  0,T\right)  }^{2}%
\int_{\mathbf{P}_{n}}\int\left\vert \frac{\partial_{v}f_{n}}{v-c_{n}%
}\right\vert dvdx\\
&  \leq T\left\Vert E_{n}^{\prime}\right\Vert _{L^{2}\left(  0,T\right)  }%
^{2}\int_{\mathbf{P}_{n}}\left\Vert f_{n}\left(  x,v\right)  \right\Vert
_{W_{v}^{s,p}}dx\\
&  \leq T\left\Vert E_{n}^{\prime}\right\Vert _{L^{2}\left(  0,T\right)  }%
^{2}\left(  \int_{\mathbf{P}_{n}}\left\Vert f_{n}\left(  x,v\right)
-f_{0}\right\Vert _{W_{v}^{s,p}}dx+\left\vert \mathbf{P}_{n}\right\vert
\left\Vert f_{0}\right\Vert _{W^{s,p}}\right) \\
&  \leq T\left\Vert E_{n}^{\prime}\right\Vert _{L^{2}\left(  0,T\right)  }%
^{2}\left(  C\left\Vert f_{n}\left(  x,v\right)  -f_{0}\right\Vert
_{W_{x,v}^{s,p}}+\left\vert \mathbf{P}_{n}\right\vert \left\Vert
f_{0}\right\Vert _{W^{s,p}}\right) \\
&  <\left\Vert E_{n}^{\prime}\right\Vert _{L^{2}\left(  0,T\right)  }^{2},
\end{align*}
when $n$ is large enough. Thus for large $n,\ \left\Vert E_{n}^{\prime
}\right\Vert _{L^{2}\left(  0,T\right)  }=0$ and thus $E_{n}\left(  x\right)
\equiv0,\ $which is a contradiction.

Case 2: $\left\vert \mathbf{P}_{n}\right\vert \rightarrow d>0$ when
$n\rightarrow\infty$. When $n$ is large enough, we have $\left\vert
\mathbf{P}_{n}\right\vert \geq\frac{d}{2}$. By the trace Theorem,
\begin{align*}
\left\Vert \partial_{v}f_{n}\left(  x,c_{n}\right)  -\partial_{v}f_{0}\left(
c_{n}\right)  \right\Vert _{L^{p}\left(  \mathbf{P}_{n}\right)  }  &
\leq\left\Vert \partial_{v}f_{n}\left(  x,c_{n}\right)  -\partial_{v}%
f_{0}\left(  c_{n}\right)  \right\Vert _{L^{p}\left(  0,T\right)  }\\
&  \leq C\left\Vert f_{n}-f_{0}\right\Vert _{W^{s,p}}\leq C\varepsilon_{n}.
\end{align*}
Since $\partial_{v}f_{n}\left(  x,c_{n}\right)  =0\,$ for a.e. $x\in
\mathbf{P}_{n}$, so $\left\Vert \partial_{v}f_{0}\left(  c_{n}\right)
\right\Vert _{L^{p}\left(  \mathbf{P}_{n}\right)  }\leq C\varepsilon_{n}$
which implies that
\[
\left\vert \partial_{v}f_{0}\left(  c_{n}\right)  \right\vert \leq
\frac{C\varepsilon_{n}}{\left(  \frac{d}{2}\right)  ^{\frac{1}{p}}}\text{.}%
\]
Thus $\partial_{v}f_{0}\left(  c_{n}\right)  \rightarrow0$ when $n\rightarrow
+\infty$. Therefore there exist a subsequence of $\left\{  c_{n}\right\}  $,
such that either it converges to one of the critical points of $f_{0}$, say
$v_{i}\in S$ or it diverges. We discuss these two cases separately below. To
simplify notations, we still denote the subsequence by $\left\{
c_{n}\right\}  $.

\ \ \ \ \ \ Case 2.1: $c_{n}\rightarrow v_{i}\in S$. Rewrite
(\ref{eqn-e_n'-square-2}) as
\begin{equation}
\int_{0}^{T}\left\vert E_{n}^{\prime}\left(  x\right)  \right\vert ^{2}%
dx=\int_{\mathbf{R}}\frac{\partial_{v}f_{0}}{v-v_{i}}dv\int_{\mathbf{P}_{n}%
}E_{n}\left(  x\right)  ^{2}dx+\int_{\mathbf{P}_{n}}V_{n}\left(  x\right)
E_{n}\left(  x\right)  ^{2}dx, \label{eqn-e_n'-square-3}%
\end{equation}
where
\[
V_{n}\left(  x\right)  =\int_{\mathbf{R}}\frac{\partial_{v}f_{n}}{v-c_{n}%
}dv-\int_{\mathbf{R}}\frac{\partial_{v}f_{0}}{v-v_{i}}dv=\int_{\mathbf{R}%
}\frac{\partial_{v}\left(  f_{n}\left(  x,v+c_{n}\right)  -f_{0}\left(
v+v_{i}\right)  \right)  }{v}dv.
\]
Note that $\partial_{v}\left(  f_{n}\left(  x,v+c_{n}\right)  -f_{0}\left(
v+v_{i}\right)  \right)  |_{v=0}=0\,$\ for $x\in\mathbf{P}_{n}$, so by Lemma
\ref{lemma-hardy}, we have
\begin{align*}
\int_{\mathbf{P}_{n}}\left\vert V_{n}\left(  x\right)  \right\vert dx  &  \leq
C\int_{\mathbf{P}_{n}}\left\Vert f_{n}\left(  x,v+c_{n}\right)  -f_{0}\left(
v+v_{i}\right)  \right\Vert _{W_{v}^{s,p}}dx\\
&  \leq C\int_{0}^{T}\left(  \left\Vert f_{n}-f_{0}\right\Vert _{W_{v}^{s,p}%
}+\left\Vert f_{0}\left(  v+c_{n}\right)  -f_{0}\left(  v+v_{i}\right)
\right\Vert _{W^{s,p}}\right)  dx\\
&  \leq C\left(  \left\Vert f_{n}-f_{0}\right\Vert _{W_{x,v}^{s,p}}+\left\Vert
f_{0}\left(  v+c_{n}\right)  -f_{0}\left(  v+v_{i}\right)  \right\Vert
_{W^{s,p}}\right)
\end{align*}
So $\int_{\mathbf{P}_{n}}\left\vert V_{n}\left(  x\right)  \right\vert
dx\rightarrow0$ when $n\rightarrow\infty$. Since $\int_{0}^{T}E_{n}\left(
x\right)  dx=0$ and $E_{n}\in H^{1}\left(  0,T\right)  $ is $T-$periodic$,$ we
have
\[
\left\Vert E_{n}^{\prime}\right\Vert _{L^{2}\left(  0,T\right)  }\geq
\frac{2\pi}{T}\left\Vert E_{n}\right\Vert _{L^{2}\left(  0,T\right)  }.\text{
}%
\]
Also by the assumption of Theorem \ref{thm-non-existence},
\[
a_{i}=\int_{\mathbf{R}}\frac{\partial_{v}f_{0}}{v-v_{i}}<\left(  \frac{2\pi
}{T}\right)  ^{2}.
\]
Combining above, from (\ref{eqn-e_n'-square-3}), we get
\begin{align*}
\left\Vert E_{n}^{\prime}\right\Vert _{L^{2}\left(  0,T\right)  }^{2}  &
\leq\frac{\max\left\{  a_{i},0\right\}  }{\left(  \frac{2\pi}{T}\right)  ^{2}%
}\left\Vert E_{n}\right\Vert _{L^{2}\left(  0,T\right)  }^{2}+\int
_{\mathbf{P}_{n}}\left\vert V_{n}\left(  x\right)  \right\vert dx\left\Vert
E_{n}\right\Vert _{L^{\infty}}^{2}\\
&  \leq\left\Vert E_{n}^{\prime}\right\Vert _{L^{2}\left(  0,T\right)  }%
^{2}\left(  \frac{\max\left\{  a_{i},0\right\}  }{\left(  \frac{2\pi}%
{T}\right)  ^{2}}+T\int_{\mathbf{P}_{n}}\left\vert V_{n}\left(  x\right)
\right\vert dx\right) \\
&  <\left\Vert E_{n}^{\prime}\right\Vert _{L^{2}\left(  0,T\right)  }^{2},
\end{align*}
when $n$ is large enough. A contradiction again.

\ \ \ \ \ \ \ \ Case 2.2: $\left\{  c_{n}\right\}  $ diverges. We assume
$c_{n}\rightarrow+\infty,$ and the case when $c_{n}\rightarrow-\infty$ is
similar. Again, for a.e. $x\in\mathbf{P}_{n}$, $\partial_{v}f_{n}\left(
x,c_{n}\right)  =0$. Let $\chi_{n}\left(  v\right)  $ be a cut-off function
such that: $0\leq\chi_{n}\leq1,\ \chi_{n}\left(  v\right)  =1$ when
$v\in\left[  \frac{c_{n}}{2},\frac{3c_{n}}{2}\right]  $;\ $\chi_{n}\left(
v\right)  =0$ when $v\notin\left[  \frac{c_{n}}{2}-1,\frac{3c_{n}}%
{2}+1\right]  $ and $\left\vert \chi_{n}\right\vert _{C^{1}}\leq M\ \left(
\text{independent of }n\right)  $. Since $W^{s_{1},p}\hookrightarrow
W^{s_{2},p}$ when $s_{1}>s_{2},\ $we can assume $\frac{1}{p}<s-1\leq1$.\ Then
\begin{align*}
&  \ \ \ \ \ \int_{\mathbf{P}_{n}}\left\vert \int_{\mathbf{R}}\frac
{\partial_{v}f_{n}}{v-c_{n}}dv\right\vert dx\leq\int_{\mathbf{P}_{n}}\left(
\int_{\mathbf{R}}\left\vert \frac{\chi_{n}\partial_{v}f_{n}}{v-c_{n}%
}\right\vert dv+\int_{\mathbf{R}}\left\vert \frac{\left(  1-\chi_{n}\right)
\partial_{v}f_{n}}{v-c_{n}}\right\vert dv\right)  dx\\
&  \leq C\int_{\mathbf{P}_{n}}\left(  \left\Vert \chi_{n}\partial_{v}%
f_{n}\right\Vert _{W_{v}^{s-1,p}}+\int_{\left\vert v-c_{n}\right\vert
\geq\frac{c_{n}}{2}}\left\vert \frac{\partial_{v}f_{n}}{v-c_{n}}\right\vert
dv\right)  dx\\
&  \leq C\int_{0}^{T}\left(  \left\Vert \chi_{n}\partial_{v}\left(
f_{n}-f_{0}\right)  \right\Vert _{W_{v}^{s-1,p}}+\left\Vert \chi_{n}%
\partial_{v}f_{0}\right\Vert _{_{W^{s-1,p}}}+c_{n}^{-1+\frac{1}{p^{\prime}}%
}\left\Vert f_{n}\right\Vert _{W_{v}^{1,p}}\right)  dx\\
&  \leq C\left(  M\right)  \left\Vert f_{n}-f_{0}\right\Vert _{W_{x,v}^{s,p}%
}+CT\left\Vert \chi_{n}\partial_{v}f_{0}\right\Vert _{_{W^{s-1,p}}}%
+CTc_{n}^{-1+\frac{1}{p^{\prime}}}\left\Vert f_{n}\right\Vert _{W_{x,v}^{1,p}%
}\\
&  \rightarrow0,\ \text{when }n\rightarrow\infty\text{,}%
\end{align*}
and this again leads to a contradiction as in Case 1. In the above, we use two estimates:

i)
\[
\left\Vert \chi_{n}\partial_{v}\left(  f_{n}-f_{0}\right)  \right\Vert
_{W_{v}^{s-1,p}}\leq C\left(  M\right)  \left\Vert \partial_{v}\left(
f_{n}-f_{0}\right)  \right\Vert _{W_{v}^{s-1,p}}.
\]

ii)
\begin{equation}
\left\Vert \chi_{n}\partial_{v}f_{0}\right\Vert _{_{W^{s-1,p}}}\rightarrow
0,\ \text{when }n\rightarrow\infty\text{. } \label{estimate-xi-n}%
\end{equation}
We prove them below. Estimate i) follows from the following general estimate:

Given $u\left(  v\right)  \in C_{0}^{1}\left(  \mathbf{R}\right)  ,\ $then for
any $g\in W^{\alpha,p}\left(  \mathbf{R}\right)  $ $\left(  p>1,0\leq
\alpha\leq1\right)  $, we have
\begin{equation}
\left\Vert ug\right\Vert _{W^{\alpha,p}\left(  \mathbf{R}\right)  }\leq
C\left(  \left\Vert u\right\Vert _{C^{1}}\right)  \left\Vert g\right\Vert
_{W^{\alpha,p}\left(  \mathbf{R}\right)  }. \label{estimate-s,p}%
\end{equation}
This estimate is obvious for $\alpha=0$ and $\alpha=1$, and the case
$\alpha\in\left(  0,1\right)  $ then follows from the interpolation theorem.
To show estimate ii), we first note that for any $h\in C_{0}^{\infty}\left(
\mathbf{R}\right)  ,\ $obviously
\[
\left\Vert \chi_{n}h\right\Vert _{_{W^{s-1,p}}}\leq C\left\Vert \chi
_{n}h\right\Vert _{_{W^{1,p}}}\rightarrow0,\ \text{when\ }n\rightarrow\infty.
\]
Then the estimate (\ref{estimate-xi-n}) follows by using the fact that
$C_{0}^{\infty}\left(  \mathbf{R}\right)  $ is dense in $W^{s-1,p}$ and the
estimate (\ref{estimate-s,p}). This finishes the proof of Theorem
\ref{thm-non-existence}.
\end{proof}

In the above proof of Theorem \ref{thm-non-existence}, we do not assume that
the possible BGK waves to have the form (\ref{form-f}) or the electric field
to vanish only at finitely many points. So we can exclude any traveling
structures which might have the form of a nontrivial wave profile plus a
homogeneous part.

The following Lemma shows that the condition $0<T<T_{0}$ in Theorem
\ref{thm-non-existence} is necessary.

\begin{lemma}
\label{lemma-existence-critical}Assume $f_{0}\left(  v\right)  \in
C^{4}\left(  \mathbf{R}\right)  \cap W^{2,p}\left(  \mathbf{R}\right)  $
$\left(  p>1\right)  .\ $Let $S=\left\{  v_{i}\right\}  _{i=1}^{l}$ be the set
of all extrema points of $f_{0}$ and $0<T_{0}<+\infty$ be defined by
\begin{equation}
\left(  \frac{2\pi}{T_{0}}\right)  ^{2}=\max_{v_{i}\in S}\int\frac
{f_{0}^{\prime}\left(  v\right)  }{v-v_{i}}dv=\int\frac{f_{0}^{\prime}\left(
v\right)  }{v-v_{m}}dv \label{bifurcation-period}%
\end{equation}
Then $\exists\ \varepsilon_{0}>0$, such that for any $0<\varepsilon
<\varepsilon_{0}\ $there exist nontrivial travelling wave solutions $\left(
f_{\varepsilon}\left(  x-v_{m}t,v\right)  ,E_{\varepsilon}\left(
x-v_{m}t\right)  \right)  \ $to (\ref{vpe}), such that$\ \left(
f_{\varepsilon}\left(  x,v\right)  ,E_{\varepsilon}\left(  x\right)  \right)
$ has period $T_{0}$ in $x$,$\ E_{\varepsilon}\left(  x\right)  $ not
identically zero, and$\ $%
\begin{equation}
\left\Vert f_{\varepsilon}-f_{0}\right\Vert _{L_{x,v}^{1}}+\ \int_{0}^{T}%
\int_{\mathbf{R}}v^{2}\left\vert f_{0}-f_{\varepsilon}\right\vert
\ dxdv+\left\Vert f_{\varepsilon}-f_{0}\right\Vert _{W_{x,v}^{2,p}%
}<\varepsilon\label{norm-lemma-critical}%
\end{equation}

\end{lemma}

\begin{proof}
To simplify notations, we assume $v_{m}=0$. As in the proof of Lemma
\ref{lemma-approxi}, for $\delta_{1}>0\ $we define
\begin{align*}
f_{\delta_{1}}\left(  v\right)   &  =f_{0}\left(  v\right)  \left(
1-\sigma\left(  \frac{v}{\delta_{1}}\right)  \right)  +\left(  \frac
{f_{0}\left(  v\right)  +f_{0}\left(  -v\right)  }{2}\right)  \sigma\left(
\frac{v}{\delta_{1}}\right) \\
&  =f_{0}\left(  v\right)  -\left(  \frac{f_{0}\left(  v\right)  -f_{0}\left(
-v\right)  }{2}\right)  \sigma\left(  \frac{v}{\delta_{1}}\right)  ,
\end{align*}
where $\sigma\left(  v\right)  $ is the cut-off function defined by
(\ref{cut-off}). Then we have:
\[
\text{i)}\ \ \ \ f_{\delta_{1}}^{\prime}\left(  0\right)  =0,\ \int\frac
{f_{0}^{\prime}\left(  v\right)  }{v}dv=\int\frac{f_{\delta_{1}}^{\prime
}\left(  v\right)  }{v}dv=\left(  \frac{2\pi}{T_{0}}\right)  ^{2},
\]
and
\[
\text{ii)\ \ \ }f_{\delta_{1}}\left(  v\right)  \in C^{4}\left(
\mathbf{R}\right)  \cap W^{2,p}\left(  \mathbf{R}\right)  ;\ \left\Vert
f_{\delta_{1}}-f_{0}\right\Vert _{W^{2,p}\left(  \mathbf{R}\right)
}\rightarrow0,\text{when\ \ \ }\delta_{1}\rightarrow0.
\]
Property i) follows since $\sigma\left(  v\right)  $ is even. To prove
property ii), we only need to show that $\left\Vert \partial_{vv}\left(
f_{\delta_{1}}-f_{0}\right)  \right\Vert _{L^{p}\left(  \mathbf{R}\right)
}\rightarrow0$ when $\delta_{1}\rightarrow0$. Since in the proof of Lemma
\ref{lemma-approxi}, it is already shown that $\left\Vert f_{\delta_{1}}%
-f_{0}\right\Vert _{W^{1,p}\left(  \mathbf{R}\right)  }\rightarrow0$ when
$\delta_{1}\rightarrow0$. Note that
\begin{align*}
\partial_{vv}\left(  f_{\delta_{1}}-f_{0}\right)   &  =\frac{1}{2\delta
_{1}^{2}}\sigma^{\prime\prime}\left(  \frac{v}{\delta_{1}}\right)  \left(
f_{0}\left(  v\right)  -f_{0}\left(  -v\right)  \right)  +\frac{1}{\delta_{1}%
}\sigma^{\prime}\left(  \frac{v}{\delta_{1}}\right)  \left(  f_{0}^{\prime
}\left(  v\right)  +f_{0}^{\prime}\left(  -v\right)  \right) \\
&  +\frac{1}{2}\sigma\left(  \frac{v}{\delta_{1}}\right)  \left(
f_{0}^{\prime\prime}\left(  v\right)  -f_{0}^{\prime\prime}\left(  -v\right)
\right) \\
&  =I+II+III.
\end{align*}
Since%
\begin{align*}
f_{0}\left(  v\right)  -f_{0}\left(  -v\right)   &  =\int_{-v}^{v}%
f_{0}^{\prime}\left(  s\right)  ds=\int_{-v}^{v}\int_{0}^{s}f_{0}%
^{\prime\prime}\left(  \tau\right)  d\tau ds\\
&  =\int_{0}^{v}\left(  v-\tau\right)  f_{0}^{\prime\prime}\left(
\tau\right)  d\tau+\int_{-v}^{0}\left(  -v-\tau\right)  f_{0}^{\prime\prime
}\left(  \tau\right)  d\tau,
\end{align*}
and
\begin{align*}
&  \ \ \ \ \ \left\vert f_{0}\left(  v\right)  -f_{0}\left(  -v\right)
\right\vert ^{p}\\
&  \leq C\left(  \left\vert \int_{0}^{v}\left(  v-\tau\right)  f_{0}%
^{\prime\prime}\left(  \tau\right)  d\tau\right\vert ^{p}+\left\vert \int
_{-v}^{0}\left(  -v-\tau\right)  f_{0}^{\prime\prime}\left(  \tau\right)
d\tau\right\vert ^{p}\right) \\
&  \leq C\left(  \int_{0}^{v}\left\vert f_{0}^{\prime\prime}\left(
\tau\right)  \right\vert ^{p}d\tau\left(  \int_{0}^{v}\left(  v-\tau\right)
^{\frac{p}{p-1}}d\tau\right)  ^{p-1}+\int_{-v}^{0}\left\vert f_{0}%
^{\prime\prime}\left(  \tau\right)  \right\vert ^{p}d\tau\left(  \int_{-v}%
^{0}\left(  v+\tau\right)  ^{\frac{p}{p-1}}d\tau\right)  ^{p-1}\right) \\
&  \leq Cv^{2p-1}\left\Vert f_{0}^{\prime\prime}\right\Vert _{L^{p}\left(
-v,v\right)  }^{p},
\end{align*}
so
\begin{align*}
\int_{\mathbf{R}}\left\vert I\right\vert ^{p}\ dv  &  \leq\frac{C}{\delta
_{1}^{2p}}\int_{\delta_{1}}^{2\delta_{1}}\left\vert f_{0}\left(  v\right)
-f_{0}\left(  -v\right)  \right\vert ^{p}dv\leq\frac{C}{\delta_{1}^{2p}}%
\int_{\delta_{1}}^{2\delta_{1}}v^{2p-1}\left\Vert f_{0}^{\prime\prime
}\right\Vert _{L^{p}\left(  -v,v\right)  }^{p}dv\\
&  \leq C\left\Vert f_{0}^{\prime\prime}\right\Vert _{L^{p}\left(
-2\delta_{1},2\delta_{1}\right)  }^{p}.
\end{align*}
Similarly,
\begin{align*}
\int_{\mathbf{R}}\left\vert II\right\vert ^{p}\ dv  &  \leq\frac{C}{\delta
_{1}^{p}}\int_{\delta_{1}}^{2\delta_{1}}\left(  \left\vert \int_{0}^{v}%
f_{0}^{\prime\prime}\left(  \tau\right)  d\tau\right\vert ^{p}+\left\vert
\int_{-v}^{0}f_{0}^{\prime\prime}\left(  \tau\right)  d\tau\right\vert
^{p}\right)  dv\\
&  \leq\frac{C}{\delta_{1}^{p}}\int_{\delta_{1}}^{2\delta_{1}}v^{p-1}%
\left\Vert f_{0}^{\prime\prime}\right\Vert _{L^{p}\left(  -v,v\right)  }%
^{p}dv\leq C\left\Vert f_{0}^{\prime\prime}\right\Vert _{L^{p}\left(
-2\delta_{1},2\delta_{1}\right)  }^{p}%
\end{align*}
and
\[
\int_{\mathbf{R}}\left\vert III\right\vert ^{p}\ dv\leq C\left\Vert
f_{0}^{\prime\prime}\right\Vert _{L^{p}\left(  -2\delta_{1},2\delta
_{1}\right)  }^{p},
\]
thus when $\delta_{1}\rightarrow0$, $\left\Vert \partial_{vv}\left(
f_{\delta_{1}}-f_{0}\right)  \right\Vert _{L^{p}\left(  \mathbf{R}\right)
}\rightarrow0.\ $Choose $\delta_{1}>0$ such that
\[
\left\Vert f_{\delta_{1}}-f_{0}\right\Vert _{W^{2,p}\left(  \mathbf{R}\right)
}<\varepsilon/2.
\]
Since $f_{\delta_{1}}\in C^{4}\left(  \mathbf{R}\right)  \cap W^{2,p}\left(
\mathbf{R}\right)  $ and
\[
\int\frac{f_{\delta_{1}}^{\prime}\left(  v\right)  }{v}dv=\left(  \frac{2\pi
}{T_{0}}\right)  ^{2},
\]
this is exactly the Case 3 treated in the proof of Proposition
\ref{propo-existence}, so we can construct a nontrivial BGK solution $\left(
f_{\varepsilon},E_{\varepsilon}\right)  $ near $\left(  f_{\delta_{1}}\left(
v\right)  ,0\right)  $ satisfying
\[
\left\Vert f_{\varepsilon}-f_{\delta_{1}}\right\Vert _{L_{x,v}^{1}}+\ \int
_{0}^{T}\int_{\mathbf{R}}v^{2}\left\vert f_{\varepsilon}-f_{\delta_{1}%
}\right\vert \ dxdv+\left\Vert f_{\varepsilon}-f_{\delta_{1}}\right\Vert
_{W_{x,v}^{2,p}}<\frac{\varepsilon}{2}.
\]
Thus $\left(  f_{\varepsilon},E_{\varepsilon}\right)  $ is a BGK solution
satisfying (\ref{norm-lemma-critical}).
\end{proof}

From the proof of Theorem \ref{thm-non-existence}, it is easy to get Corollary
\ref{cor-stable}.

\begin{proof}
[Proof of Corollary \ref{cor-stable}]Suppose otherwise, then there exists a
sequence $\varepsilon_{n}\rightarrow0$, and homogeneous states $\left\{
f_{n}\left(  v\right)  \right\}  $ which are linear unstable with $x-$period
$T$ and $\left\Vert f_{n}-f_{0}\right\Vert _{W^{s,p}\left(  \mathbf{R}\right)
}<\varepsilon_{n}$. By Lemma \ref{lemma-penrose}, for each $n$, there exists a
critical point $v_{n}$ of $f_{n}\left(  v\right)  $ such that
\[
\int\frac{f_{n}^{\prime}\left(  v\right)  }{v-v_{n}}dv>\left(  \frac{2\pi}%
{T}\right)  ^{2}.
\]
Since
\[
\left\vert f_{0}^{\prime}\left(  v_{n}\right)  \right\vert \leq\left\Vert
\partial_{v}\left(  f_{n}-f_{0}\right)  \right\Vert _{C\left(  \mathbf{R}%
\right)  }\leq C\left\Vert f_{n}-f_{0}\right\Vert _{W^{s,p}\left(
\mathbf{R}\right)  }\leq C\varepsilon_{n},
\]
either $\left\{  v_{n}\right\}  $ converges to one of the critical point of
$f_{0}\left(  v\right)  ,$ say $v_{0}$, or $\left\{  v_{n}\right\}  $
diverges. As in the proof of Theorem \ref{thm-non-existence}, in the first
case, we have
\[
\int\frac{f_{n}^{\prime}\left(  v\right)  }{v-v_{n}}dv\rightarrow\int
\frac{f_{0}^{\prime}\left(  v\right)  }{v-v_{0}}dv,\ \text{when }%
n\rightarrow\infty.
\]
This implies that
\[
\int\frac{f_{0}^{\prime}\left(  v\right)  }{v-v_{0}}dv\geq\left(  \frac{2\pi
}{T}\right)  ^{2}>\left(  \frac{2\pi}{T_{0}}\right)  ^{2},
\]
a contradiction. For the second case, we have
\[
\int\frac{f_{n}^{\prime}\left(  v\right)  }{v-v_{n}}dv\rightarrow0,\text{when
}n\rightarrow\infty,\
\]
a contradiction again.
\end{proof}

\section{Linear damping}

In this section, we study in details the linear damping problem in Sobolev
spaces. First, the linear decay estimates derived here are used in Section 5
to show that all invariant structures in $H^{s}$ $\left(  s>\frac{3}%
{2}\right)  $ neighborhood of stable homogeneous states are trivial. Second,
the linear decay holds true for initial data as rough as $f\left(  t=0\right)
\in L^{2}$, and this suggests that Theorems \ref{thm-existence},
\ref{thm-non-existence} and \ref{thm-invariant} about nonlinear dynamics have
no analogues at the linear level. We refer to Remark \ref{rmk-linear} for more discussions.

The linearized Vlasov-Poisson around a homogeneous state $\left(  f_{0}\left(
v\right)  ,0\right)  \ $is the following%
\begin{equation}
\label{lvpe}\left\{
\begin{array}
[c]{c}%
\frac{\partial f}{\partial t}+v\frac{\partial f}{\partial x}-E\frac{\partial
f_{0}}{\partial v}=0,\\
\frac{\partial E}{\partial x}=-\int_{-\infty}^{+\infty}f\ dv,
\end{array}
\right.
\end{equation}
where $f$ and $E$ are $T-$periodic in$\ x$ and the neutralizing condition
becomes $\int_{0}^{T}\int_{\mathbf{R}}f\ dvdx=0$. Notice that any $\left(
f,E\right)  =\left(  g\left(  v\right)  ,0\right)  $ with $\int g\left(
v\right)  dv=0\ $is a steady solution of the linear system (\ref{lvpe}). For a
general solution $\left(  f,E\right)  $ of (\ref{lvpe})$,$ the homogeneous
component of $f\ $remains steady and does not affect the evolution of $E$. So
we can consider a function $h\left(  x,v\right)  $ which is $T-$periodic in
$x$ and $\int_{0}^{T}h\left(  x,v\right)  dx=0$. Denote its Fourier series
representation by%

\[
h\left(  x,v\right)  =\sum_{0\neq k\in\mathbf{Z}}e^{i\frac{2\pi}{T}kx}%
h_{k}\left(  v\right)  .
\]
We define the space $H_{x}^{s_{x}}H_{v}^{s_{v}}$ by
\[
h\in H_{x}^{s_{x}}H_{v}^{s_{v}}\text{ if }\left\Vert h\right\Vert
_{H_{x}^{s_{x}}H_{v}^{s_{v}}}=\left(  \sum_{k\neq0}\left\vert k\right\vert
^{2s_{x}}\left\Vert h_{k}\right\Vert _{H_{v}^{s_{v}}}^{2}\right)  ^{\frac
{1}{2}}<\infty.
\]

\begin{proposition}
\label{prop-linear-integral-estimate}Assume $f_{0}\left(  v\right)  \in
H^{s_{0}}\left(  \mathbf{R}\right)  $ $\left(  s_{0}>\frac{3}{2}\right)
\ $and let $0<T_{0}\leq+\infty$ be defined by (\ref{defn-T-0}). Let $\left(
f\left(  x,v,t\right)  ,E\left(  x,t\right)  \right)  $ be a solution of
(\ref{lvpe}) with $x-$period $T<T_{0}\ $and $g\left(  x,v\right)  =f\left(
x,v,0\right)  -\frac{1}{T}\int_{0}^{T}f\left(  x,v,0\right)  dx$. If$\ g\in
H_{x}^{s_{x}}H_{v}^{s_{v}}$ with $\left\vert s_{v}\right\vert \leq s_{0}-1,$
then%
\begin{equation}
\left\Vert t^{s_{v}}E\left(  x,t\right)  \right\Vert _{L_{t}^{2}H_{x}%
^{\frac{3}{2}+s_{x}+s_{v}}}\leq C_{0}\left\Vert g\right\Vert _{H_{x}^{s_{x}%
}H_{v}^{s_{v}}}\leq C_{0}\left\Vert f\left(  x,v,0\right)  \right\Vert
_{H_{x}^{s_{x}}H_{v}^{s_{v}}}, \label{estimate-linear-integral}%
\end{equation}
for some constant $C_{0}.$
\end{proposition}

One may compare this proposition with other smoothing estimates in PDEs. Here
based on the most naive estimate, the initial value $g \in H_{x}^{s_{x}}%
H_{v}^{s_{v}} \subset H_{x, v}^{s_{x} + s_{v}}$ only implies $E(0) \in
H_{x}^{s_{x} +1}$ which is much weaker than $H_{x}^{\frac32 + s_{v} + s_{x}}$
in the above proposition. However, this improved regularity of $E$ may blow up
as $t \to0$.

\begin{proof}
To simplify notations, we assume $T=2\pi$. Let
\[
g\left(  x,v\right)  =\sum_{0\neq k\in\mathbf{Z}}e^{ikx}g_{k}\left(  v\right)
,
\]
then by assumption
\[
\left\Vert g\right\Vert _{H_{x}^{s_{x}}H_{v}^{s_{v}}}=\sum_{0\neq
k\in\mathbf{Z}}\left\vert k\right\vert ^{2s_{x}}\left\Vert g_{k}\right\Vert
_{H^{s_{v}}}^{2}<\infty\text{. }%
\]
Let%
\[
f\left(  x,v,t\right)  =\sum_{k\neq0}e^{ikx}h_{k}\left(  v,t\right)
,\ E\left(  x,t\right)  =\sum_{k\neq0\in\mathbf{Z}}e^{ikx}E_{k}\left(
t\right)  ,
\]
then
\[
E_{k}\left(  t\right)  =-\frac{1}{ik}\int_{\mathbf{R}}h_{k}\left(  v,t\right)
\ dv,\ E_{k}\left(  0\right)  =-\frac{1}{ik}\int_{\mathbf{R}}g_{k}\left(
v\right)  dv.
\]
Below we denote $C$ to be a generic constant depending only on $f_{0}.$When
$k>0,\ $we use the the well-known formula for $E_{k}\left(  t\right)  $
\begin{equation}
E_{k}\left(  t\right)  =\frac{1}{2\pi i}\int_{\sigma-i\infty}^{\sigma+i\infty
}\frac{G_{k}\left(  -p/ik\right)  }{k^{2}-F\left(  -p/ik\right)  }e^{pt}dp,
\label{formula-E_k}%
\end{equation}
where
\[
G_{k}\left(  z\right)  =\int_{-\infty}^{+\infty}\frac{g_{k}\left(  v\right)
}{v-z}dv,\ F\left(  z\right)  =\int_{-\infty}^{+\infty}\frac{f_{0}^{\prime
}\left(  v\right)  }{v-z}dv\text{, }\operatorname{Im}z>0
\]
and $\sigma$ is chosen so that the integrand in (\ref{formula-E_k}) has no
poles for $\operatorname{Re}p>\sigma$. The formula (\ref{formula-E_k}) was
derived in Landau's original 1946 paper (\cite{landau}) by using Laplace
transforms. Here we follow the notations in (\cite{weizner63}). By using the
new variable $z=-p/ik$, we get
\begin{equation}
E_{k}\left(  t\right)  =\frac{k}{2\pi}\int_{\frac{i\sigma}{k}-\infty}%
^{\frac{i\sigma}{k}+\infty}\frac{G_{k}\left(  z\right)  }{k^{2}-F\left(
z\right)  }e^{-ikzt}dz. \label{formula-E_k-2}%
\end{equation}
By assumption $k\geq1=\frac{2\pi}{T}>\frac{2\pi}{T_{0}}$, so by Penrose's
criterion (Lemma \ref{lemma-penrose}), there exist no unstable modes to the
linearized equation with $x-$period $2\pi/k$. Therefore, $k^{2}-F\left(
z\right)  \neq0$ when $\operatorname{Im}z>0$. Moreover, by the proof of Lemma
\ref{lemma-penrose}, under the condition $k>\frac{2\pi}{T_{0}},\ k^{2}%
-F\left(  x+i0\right)  \neq0$ for any $x\in\mathbf{R}$. It is also easy to see
that $F\left(  x+i0\right)  \rightarrow0$ when $x\rightarrow\infty$. So there
exists $c_{0}>0$, such that
\begin{equation}
\left\vert k^{2}-F\left(  x+i0\right)  \right\vert \geq c_{0}k^{2},\ \text{for
any }x\in\mathbf{R}\text{ and }k\mathbf{.} \label{estimate-bottom}%
\end{equation}
Note that for $z=i\sigma+x,\ $when$\ \sigma\rightarrow0+,\ $by (\ref{plemelj}%
),
\[
G_{k}\left(  z\right)  \rightarrow G_{k}\left(  x+i0\right)  =P\int
_{\mathbf{R}}\frac{g_{k}\left(  v\right)  }{v-x}dv+i\pi g_{k}\left(  x\right)
=\mathcal{H}g_{k}+i\pi g_{k},
\]
and
\[
F\left(  z\right)  \rightarrow F\left(  x+i0\right)  =P\int_{\mathbf{R}}%
\frac{f_{0}^{\prime}\left(  v\right)  }{v-x}dv+i\pi f_{0}^{\prime}\left(
x\right)  =\mathcal{H}f_{0}^{\prime}+i\pi f_{0}^{\prime},
\]
where $\mathcal{H}$ is the Hilbert transform. So letting $\sigma\rightarrow
0+$, from (\ref{formula-E_k-2}), we have
\begin{equation}
E_{k}\left(  t\right)  =\frac{k}{2\pi}\int_{\mathbf{R}}\frac{G_{k}\left(
x+i0\right)  }{k^{2}-F\left(  x+i0\right)  }e^{-ikxt}dx.
\label{formula-E_k-limit}%
\end{equation}
Let
\[
A_{k}\left(  t\right)  =\frac{1}{2\pi}\int_{\mathbf{R}}\frac{G_{k}\left(
x+i0\right)  }{k^{2}-F\left(  x+i0\right)  }e^{-ixt}dx
\]
be the Fourier transform of
\[
H_{k}\left(  x\right)  =\frac{G_{k}\left(  x+i0\right)  }{k^{2}-F\left(
x+i0\right)  },
\]
then $E_{k}\left(  t\right)  =kA_{k}\left(  kt\right)  $. Since $\mathcal{H}:$
$H^{s}\rightarrow H^{s}\ $is bounded for any $s\in\mathbf{R}$,
\[
\left\Vert G_{k}\left(  x+i0\right)  \right\Vert _{H^{s_{v}}}\leq C\left\Vert
g_{k}\right\Vert _{H_{v}^{s_{v}}}, \qquad\left\vert F\left(  x+i0\right)
\right\Vert _{H^{s_{0}-1}}\leq C\left\Vert f_{0}\right\Vert _{H_{v}^{s_{0}}}.
\]
By (\ref{estimate-bottom}) and the inequality
\[
\left\Vert f_{1}f_{2}\right\Vert _{H^{s}}\leq C_{s,s_{1}}\left\Vert
f_{1}\right\Vert _{H^{s_{1}}}\left\Vert f_{2}\right\Vert _{H^{s}},\ \text{if
}s_{1}>\frac{1}{2},\ \left\vert s\right\vert \leq s_{1},
\]
we have
\begin{align*}
\left\Vert H_{k}\right\Vert _{H^{s_{v}}}  &  \leq\frac{1}{k^{2}}\left\Vert
G_{k}\left(  x+i0\right)  \right\Vert _{H^{s_{v}}}+\frac{1}{k^{4}}\left\Vert
G_{k}\left(  x+i0\right)  \frac{F\left(  x+i0\right)  }{1-F\left(
x+i0\right)  /k^{2}}\right\Vert _{H^{s_{v}}}\\
&  \leq\frac{C}{k^{2}}\left\Vert g_{k}\right\Vert _{H_{v}^{s_{v}}}\left(
1+\frac{1}{k^{2}}\left\Vert \frac{F\left(  x+i0\right)  }{1-F\left(
x+i0\right)  /k^{2}}\right\Vert _{H^{s_{0}-1}}\right)  \leq\frac{C^{\prime}%
}{k^{2}}\left\Vert g_{k}\right\Vert _{H_{v}^{s_{v}}}.
\end{align*}
where $C^{\prime}$ depends on $f_{0}$ but not $k$. In the above,
the second inequality holds since the estimates
\[
\left\vert 1-F\left(  x+i0\right)  /k^{2}\right\vert \geq c_{0}\text{ and
}\left\Vert F\left(  x+i0\right)  \right\Vert _{H^{s_{0}-1}}\leq C\left\Vert
f_{0}\right\Vert _{H^{s_{0}}},
\]
imply
\begin{equation}
\left\Vert \frac{F\left(  x+i0\right)  }{1-F\left(  x+i0\right)  /k^{2}%
}\right\Vert _{H^{s_{0}-1}}< C\left\Vert f_{0}\right\Vert _{H^{s_{0}}}\text{ }
\label{inequality-s_0-1}%
\end{equation}
through direct verification where, for $0<s_{0} -1<1$, one needs to use
the equivalent characterization of $W^{s,p}\left(  \mathbf{R}^{n}\right)  $
when $0<s<1,\ p>1$ (See \cite[Lemma 35.2]{tartar-sobolev}):%
\[
W^{s,p}\left(  \mathbf{R}^{n}\right)  =\left\{  u\in L^{p}\left(
\mathbf{R}^{n}\right)  \ |\ \int\int_{\mathbf{R}^{n}\times\mathbf{R}^{n}}%
\frac{\left\vert u\left(  x\right)  -u\left(  y\right)  \right\vert ^{p}%
}{\left\vert x-y\right\vert ^{n+sp}}dxdy<\infty\right\}  .
\]
So $\ $%
\[
\int_{\mathbf{R}}\left\vert t\right\vert ^{2s_{v}}\left\vert A_{k}\left(
t\right)  \right\vert ^{2}dt\leq\left\Vert H_{k}\right\Vert _{H^{s_{v}}}%
^{2}\leq\frac{C}{k^{4}}\left\Vert g_{k}\right\Vert _{H_{v}^{s_{v}}}^{2}%
\]
and
\begin{align*}
\left\Vert t^{s_{v}}E_{k}\left(  t\right)  \right\Vert _{L^{2}}^{2}  &
=\int_{\mathbf{R}}\left\vert t\right\vert ^{2s_{v}}\left\vert E_{k}\left(
t\right)  \right\vert ^{2}dt=\int\left\vert t\right\vert ^{2s_{v}}%
k^{2}\left\vert A_{k}\left(  kt\right)  \right\vert ^{2}dt\\
&  =k^{1-2s_{v}}\int_{\mathbf{R}}\left\vert t\right\vert ^{2s_{v}}\left\vert
A_{k}\left(  t\right)  \right\vert ^{2}dt\leq Ck^{-3-2s_{v}}\left\Vert
g_{k}\right\Vert _{H_{v}^{s_{v}}}^{2}.
\end{align*}
For $k<0$, the same estimate
\[
\left\Vert t^{s_{v}}E_{k}\left(  t\right)  \right\Vert _{L^{2}}^{2}\leq
C\left\vert k\right\vert ^{-3-2s_{v}}\left\Vert g_{k}\right\Vert
_{H_{v}^{s_{v}}}^{2},
\]
follows by taking the complex conjugate of the $k>0$ case. Thus%
\begin{align*}
\left\Vert t^{s_{v}}E\left(  x,t\right)  \right\Vert _{L_{t}^{2}H_{x}%
^{\frac{3}{2}+s_{x}+s_{v}}}^{2}  &  =\sum_{k\neq0}\left\vert k\right\vert
^{3+2s_{v}+2s_{x}}\left\Vert t^{s_{v}}E_{k}\left(  t\right)  \right\Vert
_{L^{2}}^{2}\\
&  \leq C\sum_{k\neq0}\left\vert k\right\vert ^{2s_{x}}\left\Vert
g_{k}\right\Vert _{H_{v}^{s_{v}}}^{2}=C\left\Vert g\right\Vert _{H_{x}^{s_{x}%
}H_{v}^{s_{v}}}^{2}\text{.}%
\end{align*}
This finishes the proof.
\end{proof}

The decay estimate in Proposition \ref{prop-linear-integral-estimate} is in
the integral form. With some additional assumption on the initial data, we can
obtain the pointwise decay estimate.

\begin{proposition}
\label{prop-linear} Assume $f_{0}\left(  v\right)  \in H^{s_{0}}\left(
\mathbf{R}\right)  $ $\left(  s_{0}>\frac{3}{2}\right)  \ $and let
$0<T_{0}\leq+\infty$ be defined by (\ref{defn-T-0}). Let $\left(  f\left(
x,v,t\right)  ,E\left(  x,t\right)  \right)  $ be a solution of (\ref{lvpe})
with $x-$period $T<T_{0}\ $and
\[
g\left(  x,v\right)  =f\left(  x,v,0\right)  -\frac{1}{T}\int_{0}^{T}f\left(
x,v,0\right)  dx.
\]
If $g\in H_{x}^{s_{x}}H_{v}^{s_{v}}$, $vg\in H_{x}^{s_{x}^{\prime}}%
H_{v}^{s_{v}^{\prime}}$ with $s_{v}>-\frac{1}{2},\ s_{v}+s_{v}^{\prime}\geq0$
and $\max\left\{  \left\vert s_{v}\right\vert ,\left\vert s_{v}^{\prime
}\right\vert \right\}  \leq s_{0}-1$, then
\[
\left\Vert E\right\Vert _{H^{s}}\left(  t\right)  =o\left(  t^{-\frac
{s_{v}+s_{v}^{\prime}}{2}}\right)  \text{, when }t\rightarrow\infty,
\]
where
\begin{equation}
s=\min\left\{  \frac{3}{2}+s_{x}+s_{v},\frac{1}{2}+s_{x}^{\prime}%
+s_{v}^{\prime}\right\}  . \label{defn-s}%
\end{equation}

\end{proposition}

\begin{corollary}
\label{cor-decay}Assume $f_{0}\left(  v\right)  \in H^{s_{0}}\left(
\mathbf{R}\right)  $ $\left(  s_{0}>\frac{3}{2}\right)  $ and $T<T_{0}$.

(i) If $g\in H_{x}^{-\frac{3}{2}}L_{v}^{2}$ and $vg\in H_{x}^{-\frac{1}{2}%
}L_{v}^{2}$, then $\left\Vert E\right\Vert _{L_{x}^{2}}\left(  t\right)
\rightarrow0$ when $t\rightarrow\infty$.

(ii) If $g,vg\in H_{x,v}^{k}$ with $k\leq s_{0}-1,\ $then $\left\Vert
E\right\Vert _{H^{k+\frac{1}{2}}}\left(  t\right)  =o\left(  t^{-k}\right)  $
when $t\rightarrow\infty$.
\end{corollary}

Proposition \ref{prop-linear} and its Corollary shows that linear damping is
true for initial data of very low regularity, even in certain negative Sobolev
spaces. It also shows that the decay rate is mainly determined by the
regularity in $v$, although the regularity in $x$ affects the norm of
electrical field that decays.

\begin{proof}
[Proof of Proposition \ref{prop-linear}]First we derive a formula for
$E_{t}\left(  t\right)  $. We notice that $\left(  f_{t},E_{t}\right)  $
satisfies the linear system (\ref{lvpe}) and
\begin{align*}
f_{t}\left(  x,v,0\right)   &  =-v\partial_{x}f\left(  x,v,0\right)  +E\left(
x,0\right)  f_{0}^{\prime}\left(  v\right) \\
&  =\sum_{k\neq0}e^{ikx}\left(  -ikvg_{k}\left(  v\right)  -\frac{1}{ik}%
\int_{\mathbf{R}}g_{k}\left(  v\right)  dv\ f_{0}^{\prime}\left(  v\right)
\right)  =\sum_{j=1}^{3}\tilde{g}^{j}\left(  x,v\right)  ,
\end{align*}
\qquad where
\begin{align*}
\tilde{g}^{1}\left(  x,v\right)   &  =-\partial_{x}\left(  vg\left(
x,v\right)  \right)  ,\ \\
\tilde{g}^{2}\left(  x,v\right)   &  =-\frac{1}{ik}\sum_{k\neq0}e^{ikx}%
\int_{\mathbf{R}}g_{k}\left(  v\right)  \sigma\left(  v\right)  dv\ f_{0}%
^{\prime}\left(  v\right)  =\sum_{k\neq0}e^{ikx}\tilde{g}_{k}^{2}\left(
v\right)  ,\\
\tilde{g}^{3}\left(  x,v\right)   &  =-\frac{1}{ik}\sum_{k\neq0}e^{ikx}%
\int_{\mathbf{R}}g_{k}\left(  v\right)  \left(  1-\sigma\left(  v\right)
\right)  dv\ f_{0}^{\prime}\left(  v\right)  =\sum_{k\neq0}e^{ikx}\tilde
{g}_{k}^{3}\left(  v\right)  ,
\end{align*}
and$\ \sigma\left(  v\right)  $ is the cut-off function defined by
(\ref{cut-off}). Then $\tilde{g}_{1}\in H_{x}^{s_{x}^{\prime}-1}H_{v}%
^{s_{v}^{\prime}},\ \tilde{g}_{2}\in H_{x}^{s_{x}+1}H_{v}^{s_{0}-1},$
$\tilde{g}_{3}\in H_{x}^{s_{x}^{\prime}+1}H_{v}^{s_{0}-1}$, and
\[
\left\Vert \tilde{g}_{1}\right\Vert _{H_{x}^{s_{x}^{\prime}-1}H_{v}%
^{s_{v}^{\prime}}}=\left\Vert vg\right\Vert _{H_{x}^{s_{x}^{\prime}}%
H_{v}^{s_{v}^{\prime}}}%
\]%
\begin{align*}
\left\Vert \tilde{g}_{2}\right\Vert _{H_{x}^{s_{x}+1}H_{v}^{s_{0}-1}}^{2}  &
=\sum_{k}\left\vert k\right\vert ^{2s_{x}}\left\Vert \int_{\mathbf{R}}%
g_{k}\left(  v\right)  \sigma\left(  v\right)  dv\ f_{0}^{\prime}\right\Vert
_{H_{v}^{s_{0}-1}}^{2}\\
&  \leq\sum_{k}\left\vert k\right\vert ^{2s_{x}}\left\Vert g_{k}\right\Vert
_{H_{v}^{s_{v}}}^{2}\left\Vert \sigma\left(  v\right)  \right\Vert
_{H_{v}^{-s_{v}}}^{2}\left\Vert f_{0}\right\Vert _{H_{v}^{s_{0}}}^{2}\leq
C\left\Vert g\right\Vert _{H_{x}^{s_{x}}H_{v}^{s_{v}}}^{2},
\end{align*}%
\[
\left\Vert \tilde{g}_{3}\right\Vert _{H_{x}^{s_{x}^{\prime}+1}H_{v}^{s_{0}-1}%
}^{2}\leq\sum_{k}\left\vert k\right\vert ^{2s_{x}^{\prime}}\left\Vert
vg_{k}\right\Vert _{H_{v}^{s_{v}^{\prime}}}^{2}\left\Vert \frac{1-\sigma
\left(  v\right)  }{v}\right\Vert _{H_{v}^{-s_{v}^{\prime}}}^{2}\left\Vert
f_{0}\right\Vert _{H_{v}^{s_{0}}}^{2}\leq C\left\Vert vg\right\Vert
_{H_{x}^{s_{x}^{\prime}}H_{v}^{s_{v}^{\prime}}}^{2}.
\]
Correspondingly, we decompose
\[
\left(  f_{t},E_{t}\right)  =\sum_{i=1}^{3}\left(  f_{t}^{i},E_{t}^{i}\right)
\]
with $\left(  f_{t}^{i},E_{t}^{i}\right)  $ being the solution of (\ref{lvpe})
with initial data $f_{t}^{i}\left(  t=0\right)  =\tilde{g}_{i}\left(
x,v\right)  $. Then by
Proposition \label{prop-linear-integral-estimate}
we have
\[
\left\Vert t^{s_{v}^{\prime}}E_{t}^{1}\right\Vert _{L_{t}^{2}H_{x}^{\frac
{1}{2}+s_{x}^{\prime}+s_{v}^{\prime}}}\leq C\left\Vert vg\right\Vert
_{H_{x}^{s_{x}^{\prime}}H_{v}^{s_{v}^{\prime}}},\ \ \ \left\Vert t^{s_{0}%
-1}E_{t}^{2}\right\Vert _{L_{t}^{2}H_{x}^{\frac{5}{2}+s_{x}+s_{0}-1}}\leq
C\left\Vert g\right\Vert _{H_{x}^{s_{x}}H_{v}^{s_{v}}}%
\]
and%
\[
\left\Vert t^{s_{0}-1}E_{t}^{3}\right\Vert _{L_{t}^{2}H_{x}^{\frac{5}{2}%
+s_{x}^{\prime}+s_{0}-1}}\leq C\left\Vert vg\right\Vert _{H_{x}^{s_{x}%
^{\prime}}H_{v}^{s_{v}^{\prime}}}.
\]
For any $t_{2}>t_{1}$ sufficiently large and $s$ defined by (\ref{defn-s}), we
have
\begin{align*}
&  \left\vert \left\Vert E\right\Vert _{H_{x}^{s}}^{2}\left(  t_{2}\right)
-\left\Vert E\right\Vert _{H_{x}^{s}}^{2}\left(  t_{1}\right)  \right\vert \\
&  =\left\vert \int_{t_{1}}^{t_{2}}\left\langle E\left(  t\right)
,E_{t}\left(  t\right)  \right\rangle _{H_{x}^{s}}dt\right\vert \leq
\int_{t_{1}}^{t_{2}}\left\Vert E\left(  t\right)  \right\Vert _{H_{x}^{s}%
}\left(  \sum_{i=1}^{3}\left\Vert E_{t}^{i}\left(  t\right)  \right\Vert
_{H_{x}^{s}}\right)  dt\\
&  \leq t_{1}^{-s_{v}-s_{v}^{\prime}}\int_{t_{1}}^{t_{2}}\left\Vert t^{s_{v}%
}E\left(  t\right)  \right\Vert _{H_{x}^{s}}\left\Vert t^{s_{v}^{\prime}}%
E_{t}^{1}\right\Vert _{H_{x}^{\frac{1}{2}+s_{x}^{\prime}+s_{v}^{\prime}}}dt\\
&  \ \ \ \ \ \ \ \ +t_{1}^{-s_{v}-\left(  s_{0}-1\right)  }\int_{t_{1}}%
^{t_{2}}\left\Vert t^{s_{v}}E\left(  t\right)  \right\Vert _{H_{x}^{s}%
}\left\Vert t^{s_{0}-1}E_{t}^{2}\right\Vert _{H_{x}^{\frac{5}{2}+s_{x}%
+s_{0}-1}}dt\\
\ \ \  &  \ \ \ \ \ \ \ \ +t_{1}^{-s_{v}-\left(  s_{0}-1\right)  }\int_{t_{1}%
}^{t_{2}}\left\Vert t^{s_{v}}E\left(  t\right)  \right\Vert _{H_{x}^{s}%
}\left\Vert t^{s_{0}-1}E_{t}^{3}\right\Vert dt\\
&  \leq t_{1}^{-s_{v}-s_{v}^{\prime}}\left\Vert t^{s_{v}}E\left(  x,t\right)
\right\Vert _{L_{t}^{2}\left(  t_{1},t_{2}\right)  H_{x}^{\frac{3}{2}%
+s_{x}+s_{v}}}\\
&  \cdot\left(  \left\Vert t^{s_{v}^{\prime}}E_{t}^{1}\right\Vert _{L_{t}%
^{2}\left(  t_{1},t_{2}\right)  H_{x}^{\frac{1}{2}+s_{x}^{\prime}%
+s_{v}^{\prime}}}+\left\Vert t^{s_{0}-1}E_{t}^{2}\right\Vert _{L_{t}%
^{2}\left(  t_{1},t_{2}\right)  H_{x}^{\frac{5}{2}+s_{x}+s_{0}-1}}+\left\Vert
t^{s_{0}-1}E_{t}^{3}\right\Vert _{L_{t}^{2}\left(  t_{1},t_{2}\right)
H_{x}^{\frac{5}{2}+s_{x}^{\prime}+s_{0}-1}}\right)  .
\end{align*}
So $\left\{  \left\Vert E\right\Vert _{H_{x}^{s}}^{2}\left(  t\right)
\right\}  _{t\geq0}$ is a Cauchy sequence, thus $\lim_{t\rightarrow\infty
}\left\Vert E\right\Vert _{H_{x}^{s}}^{2}\left(  t\right)  $ exists and must
be zero since $\left\Vert t^{s_{v}}E\right\Vert _{L_{t}^{2}H_{x}^{s}}%
^{2}<\infty$ with $s_{v}>-\frac{1}{2}$. By fixing $t_{1}$ and letting
$t_{2}\rightarrow\infty$ in the above computation, it follows that
\[
\left\Vert E\right\Vert _{H_{x}^{s}}^{2}\left(  t_{1}\right)  =o\left(
t_{1}^{-s_{v}-s_{v}^{\prime}}\right)  .
\]
This finishes the proof.
\end{proof}

\begin{remark}
The integral decay estimate in Proposition \ref{prop-linear-integral-estimate}
is optimal and the pointwise decay estimate in Proposition \ref{prop-linear}
is close to be optimal. Intuitively, the integral estimate
(\ref{estimate-linear-integral}) suggests that
\begin{equation}
\left\Vert E\left(  x,t\right)  \right\Vert _{H_{x}^{\frac{3}{2}+s_{x}+s_{v}}%
}=o\left(  t^{-\left(  s_{v}+\frac{1}{2}\right)  }\right)  .
\label{estimate-decay-intuitive}%
\end{equation}
In \cite{weizner63}, the single-mode solution $e^{ikx}\left(  f\left(
v,t\right)  ,E\left(  t\right)  \right)  $ with initial profile
\[
f\left(  v,0\right)  =g\left(  v\right)  =\left\{
\begin{array}
[c]{cc}%
\left(  v-\alpha\right)  ^{2}e^{-\left(  v-\alpha\right)  ^{2}} & v\geq
\alpha\\
0 & v\leq\alpha
\end{array}
\right.  ,\ \alpha\text{ is arbitrary constant,}%
\]
was calculated explicitly for the linearized problem at Maxwellian, and the
decay rate for $\left\vert E\left(  t\right)  \right\vert $ was found to be
$O\left(  t^{-3}\right)  $. Note that $g\left(  v\right)  ,vg\in H^{2}$ and
$g^{\prime\prime\prime},\left(  vg\right)  ^{\prime\prime\prime}$ are delta
functions which belong to $H^{-\left(  \frac{1}{2}+\varepsilon\right)  }$ for
any $\varepsilon>0,~$\ and thus $g\left(  v\right)  ,vg\in H^{\frac{5}%
{2}-\varepsilon}$. So Proposition \ref{prop-linear-integral-estimate} suggests
a decay rate $o\left(  t^{-\left(  3-\varepsilon\right)  }\right)  \ $in the
integral form and Corollary \ref{cor-decay} (ii) yields a pointwise decay rate
$o\left(  t^{-\frac{5}{2}+\varepsilon}\right)  $. In \cite[pp. 188-189]%
{akheizer-et}, the authors made a more general claim about the decay rate of
single mode solutions: for initial profile $g\left(  v\right)  $ with $\left(
n+1\right)  -$th derivative being $\delta-$function like, the decay rate of
$\left\vert E\left(  t\right)  \right\vert $ is $O\left(  t^{-\left(
n+1\right)  }\right)  $. In such cases, our results give the decay rates
$o\left(  t^{-\left(  n-\varepsilon\right)  }\right)  $ in the integral form
and $o\left(  t^{-\left(  n+\frac{1}{2}-\varepsilon\right)  }\right)  $ pointwise.

In Theorem \ref{thm-invariant}, we use the integral estimate
(\ref{estimate-linear-integral}) to prove that $H^{\frac{3}{2}}$ is the
critical regularity for existence or nonexistence of nontrivial invariant
structures near stable homogeneous states. This again suggests that the decay
estimate in Proposition \ref{prop-linear-integral-estimate} is optimal.
\end{remark}

\begin{remark}
The linear decay result is also true for initial data in $L^{p}$ space. For
simplicity, we consider a single mode solution%
\begin{equation}
\left(  f\left(  x,v,t\right)  ,E\left(  x,t\right)  \right)  =e^{ikx}\left(
h\left(  v,t\right)  ,E\left(  t\right)  \right)  \label{single-mode}%
\end{equation}
to (\ref{lvpe}) with $h\left(  v,0\right)  =g\left(  v\right)  $. Assume
$f_{0}\left(  v\right)  \in L^{1}\left(  \mathbf{R}\right)  \cap W^{2,p_{0}%
}\left(  \mathbf{R}\right)  $ $\left(  p_{0}>1\right)  $ and $0<T_{0}%
\leq+\infty$ be defined by (\ref{defn-T-0}). We have the following result: If
$T=\frac{2\pi}{k}<T_{0}$ and $g\left(  v\right)  \in L^{p}\ \left(
p>1\right)  ,\ v^{2}g\in L^{1}$, then $\left\vert E\left(  t\right)
\right\vert \rightarrow0$ when $t\rightarrow+\infty$. We prove it briefly
below. Since $g\in L^{p},v^{2}g\in L^{1}$, so
\[
\left\Vert g\right\Vert _{L^{1}\left(  \mathbf{R}\right)  }\leq\int
_{\left\vert v\right\vert \leq1}\left\vert g\right\vert \ dv+\int_{\left\vert
v\right\vert \geq1}\left\vert g\right\vert \ dv\leq2^{1/p^{\prime}}\left\Vert
g\right\Vert _{L^{p}}+\left\Vert v^{2}g\right\Vert _{L^{1}}<\infty,
\]
and%
\[
\left\Vert vg\right\Vert _{L^{q}}\leq\left\Vert v\left\vert g\right\vert
^{\frac{1}{2}}\right\Vert _{L^{2}}\left\Vert \left\vert g\right\vert
^{\frac{1}{2}}\right\Vert _{L^{2p}}\leq\left\Vert v^{2}g\right\Vert _{L^{1}%
}^{\frac{1}{2}}\left\Vert g\right\Vert _{L^{p}}^{\frac{1}{2}},
\]
for $1<q<2$ satisfying $\frac{1}{q}=\frac{1}{2}+\frac{1}{2p}$. Since $q<p\,$,
for any $1<q_{1}<q,$ letting $\frac{1}{q_{2}}=\frac{1}{q_{1}}-\frac{1}{q},$ we
have
\begin{align*}
\left\Vert g\right\Vert _{L^{q_{1}}\left(  \mathbf{R}\right)  }  &
\leq\left(  \left\Vert g\right\Vert _{L^{q_{1}}\left(  \left\vert v\right\vert
\leq1\right)  }+\left\Vert g\right\Vert _{L^{q_{1}}\left(  \left\vert
v\right\vert \geq1\right)  }\right) \\
&  \leq C\left(  \left\Vert g\right\Vert _{L^{p}}+\left\Vert \frac{1}%
{v}\right\Vert _{L^{q_{2}}\left(  \left\vert v\right\vert \geq1\right)
}\left\Vert vg\right\Vert _{L^{q}}\right)  <\infty.
\end{align*}
Since $\mathcal{H}$ is bounded $L^{p}\rightarrow L^{p}$ for any $p>1\ $and the
Fourier transform is bounded $L^{p}\rightarrow L^{p^{\prime}}$ for any
$1<p\leq2$, so from (\ref{formula-E_k-limit}),%
\begin{equation}
\left\Vert E\left(  t\right)  \right\Vert _{L^{q_{1}\prime}}\leq C\left\Vert
g\right\Vert _{L^{q_{1}}\left(  \mathbf{R}\right)  }<\infty.
\label{estimate-E}%
\end{equation}
As in the proof of Proposition \ref{prop-linear}, $\left(  f_{t},E_{t}\right)
$ satisfies (\ref{lvpe})with
\[
f_{t}\left(  t=0\right)  =e^{ikx}\left(  -ikvg\left(  v\right)  -\frac{1}%
{ik}\int_{\mathbf{R}}g\left(  v\right)  dv\ f_{0}^{\prime}\left(  v\right)
\right)  =e^{ikx}\tilde{g}\left(  v\right)  .
\]
Since
\[
\left\Vert \tilde{g}\left(  v\right)  \right\Vert _{L^{q}}\leq C\left(
\left\Vert vg\right\Vert _{L^{q}}+\left\Vert g\right\Vert _{L^{1}\left(
\mathbf{R}\right)  }\left\Vert f_{0}^{\prime}\right\Vert _{L^{q}}\right)
<\infty,
\]
by using the estimate for $E\left(  t\right)  ,\ $we get%
\begin{equation}
\left\Vert E^{\prime}\left(  t\right)  \right\Vert _{L^{q^{\prime}}}\leq
C\left\Vert \tilde{g}\left(  v\right)  \right\Vert _{L^{q}}<\infty.
\label{estimate-E'}%
\end{equation}
The decay of $\left\vert E\left(  t\right)  \right\vert \ $follows from the
estimates (\ref{estimate-E}) and (\ref{estimate-E'}).
\end{remark}

\begin{remark}
\label{rmk-linear}In Proposition \ref{prop-linear}, we prove that the linear
decay of electrical field $E\ $in $L^{2}$ norm holds true for initial data as
rough as
\[
f\left(  t=0\right)  \in H_{x}^{-\frac{3}{2}}L_{v}^{2},\ vf\left(  t=0\right)
\in H_{x}^{-\frac{1}{2}}L_{v}^{2}.
\]
In particular, it is not necessary to have any assumption on derivatives of
$f\left(  t=0\right)  $ to get linear decay of $E$. The linear decay result
implies that there exist no nontrivial invariant structures even in
$H_{x}^{-\frac{3}{2}}L_{v}^{2}$ space for the linearized problem. So our
result on existence of BGK waves in $W^{s,p}$ $\left(  s<1+\frac{1}{p}\right)
\ $neighborhood (Theorem \ref{thm-existence}) can not be traced back to the
linearized level. Also, the contrasting nonlinear dynamics in $W^{s,p}\left(
s>1+\frac{1}{p}\right)  \ $and particularly in $H^{s}\ \left(  s>\frac{3}%
{2}\right)  \ $spaces (Theorems \ref{thm-non-existence} and
\ref{thm-invariant}) have no analogue on the linearized level. These again are
due to the fact that particle trapping effects are completely ignored on the
linear level, but instead they play an important role on nonlinear dynamics.
\end{remark}

\section{Invariant structures in $H^{s}$ $\left(  s>\frac{3}{2}\right)  $}

We define invariant structures near a homogeneous state $\left(  f_{0}\left(
v\right)  ,0\right)  \ $in $H_{x,v}^{s}$ $\left(  s\geq0\right)  \ $space to
be the solutions $\left(  f\left(  t\right)  ,E\left(  t\right)  \right)  $ of
nonlinear VP equation (\ref{vlasov})-(\ref{poisson}), satisfying that for all
$t\in\mathbf{R,}$
\[
\left\Vert f\left(  t\right)  -f_{0}\right\Vert _{H^{s}\left(  \left(
0,T\right)  \times\mathbf{R}\right)  }<\varepsilon_{0},
\]
for some constant $\varepsilon_{0}>0$. The above defined invariant structures
include the well known structures such as travelling waves, time-periodic,
quasi-periodic or almost periodic solutions. In Sections 2 and 3, we prove
that $W^{1+\frac{1}{p},p}$ is the critical regularity for existence of
nontrivial travelling waves near a stable homogeneous state. For $p=2,$ this
critical regularity is $H^{\frac{3}{2}}$. In this section, we prove a much
stronger result that $H^{\frac{3}{2}}$ is also the critical regularity for
existence of any nontrivial invariant structure near a stable homogeneous
state. In the proof, we use the linear decay estimate in Proposition
\ref{prop-linear-integral-estimate}.

\begin{lemma}
\label{lemma-estimate-integral-small}Assume $f_{0}\left(  v\right)  \in
H^{s_{0}}\left(  \mathbf{R}\right)  $ $\left(  s_{0}>\frac{3}{2}\right)
\ $and let $0<T_{0}\leq+\infty$ be defined by (\ref{defn-T-0}). Let $\left(
f\left(  x,v,t\right)  ,E\left(  x,t\right)  \right)  $ be a solution of
(\ref{vlasov})-(\ref{poisson}) with $x-$period $T<T_{0}$, satisfying that: For
some $\frac{3}{2}<s\leq s_{0}$ and sufficiently small $\varepsilon_{0},$
\[
\left\Vert f\left(  t\right)  -f_{0}\right\Vert _{L_{x}^{2}H_{v}^{s}\left(
\left(  0,T\right)  \times\mathbf{R}\right)  }<\varepsilon_{0},\ \text{for all
}t\geq0.
\]
Then%
\begin{equation}
\left\Vert \left(  1+t\right)  ^{s-1}E\left(  x,t\right)  \right\Vert
_{L_{\left\{  t\geq0\right\}  }^{2}H_{x}^{\frac{3}{2}}}\leq C\varepsilon_{0},
\label{estimate-stability-integral}%
\end{equation}
for some constant $C.$
\end{lemma}

\begin{proof}
Denote $L_{0}$ to be the linearized operator corresponding to the linearized
Vlasov-Poisson equation at $\left(  f_{0}\left(  v\right)  ,0\right)  $, and
$\mathcal{E}$ is the mapping from $f\left(  x,v\right)  $ to $E\left(
x\right)  $ by the Poisson equation
\[
E_{x}=-\int f\ dv,
\]
where $f$ satisfies the neutral condition $\int_{0}^{T}\int_{\mathbf{R}%
}f\left(  x,v\right)  dvdx=0.$ It follows from Proposition
\ref{prop-linear-integral-estimate} that: For any $0\leq s_{v}\leq s_{0}%
-1,\ $if $h\left(  x,v\right)  \in L_{x}^{2}H_{v}^{s_{v}},$then
\begin{equation}
\left\Vert \left(  1+t\right)  ^{s_{v}}\mathcal{E}\left(  e^{tL_{0}}h\right)
\right\Vert _{L_{t}^{2}H_{x}^{\frac{3}{2}}}\leq C\left\Vert h\left(
x,v\right)  \right\Vert _{L_{x}^{2}H_{v}^{s_{v}}}.
\label{estimate-linear-notation}%
\end{equation}
Denote $f_{1}\left(  t\right)  =f\left(  t\right)  -f_{0}$, then
\[
\partial_{t}f_{1}=L_{0}f_{1}+E\partial_{v}f_{1}.
\]
Thus
\[
f_{1}\left(  t\right)  =e^{tL_{0}}f_{1}\left(  0\right)  +\int_{0}%
^{t}e^{\left(  t-u\right)  L_{0}}\left(  E\partial_{v}f_{1}\right)  \left(
u\right)  du=f_{\text{lin}}\left(  t\right)  +f_{\text{non}}\left(  t\right)
,
\]
and correspondingly
\[
E\left(  t\right)  =\mathcal{E}\left(  f_{\text{lin}}\left(  t\right)
\right)  +\mathcal{E}\left(  f_{\text{non}}\left(  t\right)  \right)
=E_{\text{lin}}\left(  t\right)  +E_{\text{non}}\left(  t\right)  .
\]
By the linear estimate (\ref{estimate-linear-notation}),
\[
\left\Vert \left(  1+t\right)  ^{s-1}E_{\text{lin}}\left(  x,t\right)
\right\Vert _{L_{\left\{  t\geq0\right\}  }^{2}H_{x}^{\frac{3}{2}}}\leq
C\left\Vert f_{1}\left(  0\right)  \right\Vert _{L_{x}^{2}H_{v}^{s-1}},
\]
and
\begin{align*}
&  \left\Vert \left(  1+t\right)  ^{s-1}E_{\text{non}}\left(  x,t\right)
\right\Vert _{L_{\left\{  t\geq0\right\}  }^{2}H_{x}^{\frac{3}{2}}}^{2}\\
&  =\int_{0}^{\infty}\left(  1+t\right)  ^{2\left(  s-1\right)  }\left\Vert
E_{\text{non}}\left(  x,t\right)  \right\Vert _{H_{x}^{\frac{3}{2}}}^{2}dt\\
&  \leq\int_{0}^{\infty}\left(  1+t\right)  ^{2\left(  s-1\right)  }\left(
\int_{0}^{t}\left\Vert \mathcal{E}\left[  e^{\left(  t-u\right)  L_{0}}\left(
E\partial_{v}f_{1}\right)  \left(  u\right)  \right]  \right\Vert
_{H_{x}^{\frac{3}{2}}}du\right)  ^{2}dt\\
&  \leq\int_{0}^{\infty}\left(  1+t\right)  ^{2\left(  s-1\right)  }\int
_{0}^{t}\left(  1+\left(  t-u\right)  \right)  ^{-2\left(  s-1\right)
}\left(  1+u\right)  ^{-2\left(  s-1\right)  }du\\
&  \ \ \ \ \ \ \cdot\int_{0}^{t}\left(  1+u\right)  ^{2\left(  s-1\right)
}\left(  1+\left(  t-u\right)  \right)  ^{2\left(  s-1\right)  }\left\Vert
\mathcal{E}\left[  e^{\left(  t-u\right)  L_{0}}\left(  E\partial_{v}%
f_{1}\right)  \left(  u\right)  \right]  \right\Vert _{H_{x}^{\frac{3}{2}}%
}^{2}\ dudt\\
&  \leq C\int_{0}^{\infty}\int_{0}^{t}\left(  1+u\right)  ^{2\left(
s-1\right)  }\left(  1+\left(  t-u\right)  \right)  ^{2\left(  s-1\right)
}\left\Vert \mathcal{E}\left[  e^{\left(  t-u\right)  L_{0}}\left(
E\partial_{v}f_{1}\right)  \left(  u\right)  \right]  \right\Vert
_{H_{x}^{\frac{3}{2}}}^{2}\ dudt\\
&  =C\int_{0}^{\infty}\left(  1+u\right)  ^{2\left(  s-1\right)  }\int
_{u}^{\infty}\left(  1+\left(  t-u\right)  \right)  ^{2\left(  s-1\right)
}\left\Vert \mathcal{E}\left[  e^{\left(  t-u\right)  L_{0}}\left(
E\partial_{v}f_{1}\right)  \left(  u\right)  \right]  \right\Vert
_{H_{x}^{\frac{3}{2}}}^{2}\ dtdu\\
&  \leq C\int_{0}^{\infty}\left(  1+u\right)  ^{2\left(  s-1\right)
}\left\Vert \left(  E\partial_{v}f_{1}\right)  \left(  u\right)  \right\Vert
_{L_{x}^{2}H_{v}^{s-1}}^{2}du\\
&  \leq C\int_{0}^{\infty}\left(  1+u\right)  ^{2\left(  s-1\right)
}\left\Vert E\left(  u\right)  \right\Vert _{H_{x}^{\frac{3}{2}}}%
^{2}\left\Vert f_{1}\left(  u\right)  \right\Vert _{L_{x}^{2}H_{v}^{s}}%
^{2}du\\
&  \leq C\varepsilon_{0}^{2}\left\Vert \left(  1+t\right)  ^{s-1}E\left(
x,t\right)  \right\Vert _{L_{\left\{  t\geq0\right\}  }^{2}H_{x}^{\frac{3}{2}%
}}^{2}.
\end{align*}
In the above estimate, we use the fact that
\[
\int_{0}^{t}\left(  1+\left(  t-u\right)  \right)  ^{-2\left(  s-1\right)
}\left(  1+u\right)  ^{-2\left(  s-1\right)  }du\leq C\left(  1+t\right)
^{-2\left(  s-1\right)  }%
\]
because $2\left(  s-1\right)  >1\ $by our assumption that $s>\frac{3}{2}$, and
the inequality
\[
\left\Vert E\partial_{v}f_{1}\right\Vert _{L_{x}^{2}H_{v}^{s-1}}\leq
C\left\Vert E\right\Vert _{H_{x}^{\frac{3}{2}}}\left\Vert f_{1}\right\Vert
_{L_{x}^{2}H_{v}^{s}}.
\]
Thus
\begin{align*}
&  \left\Vert \left(  1+t\right)  ^{s-1}E\left(  x,t\right)  \right\Vert
_{L_{\left\{  t\geq0\right\}  }^{2}H_{x}^{\frac{3}{2}}}\\
&  \leq\left\Vert \left(  1+t\right)  ^{s-1}E_{\text{lin}}\left(  x,t\right)
\right\Vert _{L_{\left\{  t\geq0\right\}  }^{2}H_{x}^{\frac{3}{2}}}+\left\Vert
\left(  1+t\right)  ^{s-1}E_{\text{non}}\left(  x,t\right)  \right\Vert
_{L_{\left\{  t\geq0\right\}  }^{2}H_{x}^{\frac{3}{2}}}\\
&  \leq C\left\Vert f_{1}\left(  0\right)  \right\Vert _{L_{x}^{2}H_{v}^{s}%
}+C\varepsilon_{0}\left\Vert \left(  1+t\right)  ^{s-1}E\left(  x,t\right)
\right\Vert _{L_{\left\{  t\geq0\right\}  }^{2}H_{x}^{\frac{3}{2}}}.
\end{align*}
By taking $\varepsilon_{0}=\frac{1}{2C},$ we get the estimate
(\ref{estimate-stability-integral}).
\end{proof}

\begin{proof}
[Proof of Theorem \ref{thm-invariant}]For any $t_{0}>0$, let $\left(
\tilde{f}\left(  t\right)  ,\tilde{E}\left(  t\right)  \right)  $ be the
solution of nonlinear VP equation (\ref{vlasov})-(\ref{poisson}) with the
initial data
\[
\left(  \tilde{f}\left(  0\right)  ,\tilde{E}\left(  0\right)  \right)
=\left(  f\left(  -t_{0}\right)  ,E\left(  -t_{0}\right)  \right)  .
\]
Then
\[
\left(  f\left(  t\right)  ,E\left(  t\right)  \right)  =\left(  \tilde
{f}\left(  t+t_{0}\right)  ,\tilde{E}\left(  t+t_{0}\right)  \right)  .
\]
The assumption (\ref{assumption-thm-invariant}) implies that
\[
\left\Vert \tilde{f}\left(  t\right)  -f_{0}\right\Vert _{L_{x}^{2}H_{v}^{s}%
}<\varepsilon_{0},\ \text{for\ all\ }t\in\mathbf{R}.
\]
Thus by Lemma \ref{lemma-estimate-integral-small},
\[
\left\Vert \left(  1+t\right)  ^{s-1}\tilde{E}\left(  x,t\right)  \right\Vert
_{L_{\left\{  t\geq0\right\}  }^{2}H_{x}^{\frac{3}{2}}}\leq C\varepsilon_{0}.
\]
So%
\begin{align*}
\int_{0}^{1}\left\Vert E\left(  x,t\right)  \right\Vert _{H_{x}^{\frac{3}{2}}%
}^{2}dt  &  =\int_{t_{0}}^{t_{0}+1}\left\Vert \tilde{E}\left(  x,t\right)
\right\Vert _{H_{x}^{\frac{3}{2}}}^{2}dt\\
&  \leq\frac{1}{\left(  1+t_{0}\right)  ^{2\left(  s-1\right)  }}\int_{t_{0}%
}^{t_{0}+1}\left(  1+t\right)  ^{2\left(  s-1\right)  }\left\Vert \tilde
{E}\left(  x,t\right)  \right\Vert _{H_{x}^{\frac{3}{2}}}^{2}dt\\
&  \leq\frac{\left(  C\varepsilon_{0}\right)  ^{2}}{\left(  1+t_{0}\right)
^{2\left(  s-1\right)  }}.
\end{align*}
Since $t_{0}$ can be arbitrarily large, we have
\[
\int_{0}^{1}\left\Vert E\left(  x,t\right)  \right\Vert _{H_{x}^{\frac{3}{2}}%
}^{2}dt=0
\]
and thus $E\left(  x,t\right)  \equiv0$ when $t\in\left[  0,1\right]  $.
Repeating the above argument for any finite time interval $I\subset\mathbf{R}%
$, we get $E\left(  x,t\right)  \equiv0$ when $t\in I$. Thus $E\left(
x,t\right)  \equiv0$ for any $t\in\mathbf{R}$.
\end{proof}

The following nonlinear instability result follows immediately from Theorem
\ref{thm-invariant}.

\begin{corollary}
Assume the homogeneous profile $f_{0}\left(  v\right)  \in H^{s}\left(
\mathbf{R}\right)  $ $\left(  s>\frac{3}{2}\right)  .\ $For any $T<T_{0}$
(defined by (\ref{defn-T-0})), there exists $\varepsilon_{0}>0$, such that for
any solution $\left(  f\left(  t\right)  ,E\left(  t\right)  \right)  $ to the
nonlinear VP equation (\ref{vlasov})-(\ref{poisson}) with nonzero $E\left(
0\right)  $, there exists $T\in\mathbf{R}$ such that $\left\Vert f\left(
T\right)  -f_{0}\right\Vert _{L_{x}^{2}H_{v}^{s}}\geq\varepsilon_{0}.$
\end{corollary}

The invariant structures studied in Theorem \ref{thm-invariant} stay in the
$L_{x}^{2}H_{v}^{s}$ $\left(  s>\frac{3}{2}\right)  \ $neighborhood of a
stable homogeneous state $\left(  f_{0}\left(  v\right)  ,0\right)  \ $for all
time $t\in\mathbf{R}$. We can also study the positive (or negative) invariant
structures near $\left(  f_{0}\left(  v\right)  ,0\right)  ,$ which are
solutions $\left(  f\left(  t\right)  ,E\left(  t\right)  \right)  $ to
nonlinear VP equation satisfying that $\left\Vert f\left(  t\right)
-f_{0}\right\Vert _{L_{x}^{2}H_{v}^{s}}<\varepsilon_{0},\ $for\ all\ $t\geq0$
(or $t\leq0$)$.$ The next theorem shows that the electric field of these
semi-invarint structures must decay when $t\rightarrow+\infty$ (or
$t\rightarrow-\infty).$

\begin{theorem}
\label{thm-semi-invariant}Assume the homogeneous profile $f_{0}\left(
v\right)  \in H^{s}\left(  \mathbf{R}\right)  $ $\left(  s>\frac{3}{2}\right)
.\ $For any $T<T_{0}$ (defined by (\ref{defn-T-0})), there exists
$\varepsilon_{0}>0$ sufficiently small, such that if
\[
\left\Vert f\left(  t\right)  -f_{0}\right\Vert _{L_{x}^{2}H_{v}^{s}%
}<\varepsilon_{0},\ \text{for\ all\ }t\geq0\text{ }\left(  \text{or }%
t\leq0\right)  ,
\]
and
\[
\left\Vert f\left(  0\right)  \right\Vert _{L_{x,v}^{\infty}}<\infty
,\ \int_{0}^{T}\int_{\mathbf{R}}v^{2}f\left(  0,x,v\right)  dvdx<\infty,
\]
then $\left\Vert E\left(  t,x\right)  \right\Vert _{L_{x}^{2}}\rightarrow0$
when $t\rightarrow+\infty$ $\left(  \text{or }t\rightarrow-\infty\right)  .$
\end{theorem}

\begin{proof}
We only consider the positive invariant case, since the proof is the same for
the negative invariant case. First, there exists a constant $C$ depending on
$M_{1}=\left\Vert f\left(  0\right)  \right\Vert _{L^{\infty}}$ and
$M_{2}=\int_{0}^{T}\int_{\mathbf{R}}\frac{1}{2}v^{2}f\left(  0\right)  dvdx$,
such that
\[
\left\Vert E\left(  x,t\right)  \right\Vert _{H_{x}^{1}}\leq C,\ \text{for all
}t.
\]
Indeed, by the same estimate as in (\ref{estimate-rho-n}),%
\begin{align*}
\left\Vert \rho\left(  x,0\right)  \right\Vert _{L^{3}}  &  =\left\Vert \int
f\left(  x,v,0\right)  dv\right\Vert _{L^{3}}\leq\left\Vert f\left(  0\right)
\right\Vert _{L^{\infty}}^{\frac{2}{3}}\left(  \int_{0}^{T}\int_{\mathbf{R}%
}v^{2}f\left(  x,v,0\right)  dvdx\right)  ^{1/3}\\
&  =M_{1}^{\frac{2}{3}}M_{2}^{\frac{1}{3}}%
\end{align*}
and%
\begin{equation}
\left\Vert E\left(  x,0\right)  \right\Vert _{H^{1}}\leq C\left\Vert
1-\rho\left(  x,0\right)  \right\Vert _{L^{2}}\leq C\left(  T^{/2}%
+T^{1/6}\left\Vert \rho\left(  x,0\right)  \right\Vert _{L^{3}}\right)  \leq
C. \label{estimate-E-0}%
\end{equation}
By the energy conservation,
\[
\int_{0}^{T}\int_{\mathbf{R}}v^{2}f\left(  x,v,t\right)  dvdx+\left\Vert
E\left(  x,t\right)  \right\Vert _{L_{x}^{2}}^{2}=\int_{0}^{T}\int
_{\mathbf{R}}v^{2}f\left(  x,v,0\right)  dvdx+\left\Vert E\left(  x,0\right)
\right\Vert _{L^{2}}^{2}<C.
\]
Let $j=\int vf\ dv$, then
\[
\left\vert j\left(  t\right)  \right\vert =\left\vert \int vf\ \left(
t\right)  dv\right\vert \leq\left\Vert f\left(  t\right)  \right\Vert
_{L^{\infty}}^{1/3}\left(  \int_{\mathbf{R}}v^{2}f\left(  x,v,t\right)
dvdx\right)  ^{2/3},
\]
and thus
\[
\left\Vert j\left(  x,t\right)  \right\Vert _{L_{x}^{\frac{3}{2}}}\leq
M_{1}^{\frac{1}{3}}M_{2}^{\frac{3}{2}}\leq C.
\]
Since%
\begin{align*}
\frac{d}{dt}\left\Vert E\left(  x,t\right)  \right\Vert _{L_{x}^{2}}^{2}  &
=\int_{0}^{T}j\left(  x,t\right)  E\left(  x,t\right)  dx\\
&  \leq\left\Vert j\left(  x,t\right)  \right\Vert _{L_{x}^{\frac{3}{2}}%
}\left\Vert E\left(  x,t\right)  \right\Vert _{L_{x}^{3}}\leq C\left\Vert
E\left(  x,t\right)  \right\Vert _{H_{x}^{1}},
\end{align*}
and
\begin{align*}
\int_{0}^{\infty}\left\Vert E\left(  x,t\right)  \right\Vert _{H_{x}^{1}}dt
&  \leq\left(  \int_{0}^{\infty}\left(  1+t\right)  ^{-2\left(  s-1\right)
}dt\right)  ^{\frac{1}{2}}\left(  \int_{0}^{\infty}\left(  1+t\right)
^{2\left(  s-1\right)  }\left\Vert E\left(  x,t\right)  \right\Vert
_{H_{x}^{\frac{3}{2}}}^{2}dt\right)  ^{\frac{3}{2}}\\
&  \leq C\varepsilon_{0},
\end{align*}
thus $\lim_{t\rightarrow\infty}\left\Vert E\left(  x,t\right)  \right\Vert
_{L_{x}^{2}}$ exists and this limit must be zero. This finishes the proof.
\end{proof}

\section{Appendix}

In this appendix, we reformulate Penrose's linear stability criterion. The
main purpose is to clarify the intervals of wave numbers (periods) for which
linear instability can be found. In the original paper of Penrose
\cite{penrose}, a necessary and sufficient condition was given for linear
instability of a homogeneous state at certain wave number. However, the
precise range of unstable wave numbers was not given in \cite{penrose}.

\begin{lemma}
\label{lemma-penrose}Assume $f_{0}\left(  v\right)  \in W^{2,p}\left(
\mathbf{R}\right)  $ $\left(  p>1\right)  .\ $Let $S=\left\{  v_{i}\right\}
_{i=1}^{l}$ be the set of all extrema points of $f_{0}.\ $If for some $1\leq
i\leq l$,
\begin{equation}
\int\frac{f_{0}^{\prime}\left(  v\right)  }{v-v_{i}}dv=\left(  \frac{2\pi
}{T_{i}}\right)  ^{2}>0, \label{instability-T-i}%
\end{equation}
then there exists linearly growing mode with $x-$period $T\ $near $T_{i}$.
More precisely, when $v_{i}$ is a minimum (maximum) point of $f_{0}$, unstable
modes exist for $T$ slightly greater (smaller) than $T$. Let $0<T_{0}%
\leq+\infty$ be defined by
\[
\left(  \frac{2\pi}{T_{0}}\right)  ^{2}=\max\left\{  0,\max_{v_{i}\in S}%
\int\frac{f_{0}^{\prime}\left(  v\right)  }{v-v_{i}}dv\right\}  .
\]
Then for $T<T_{0}$, there exist no unstable modes with $x-$Period $T$.
\end{lemma}

\begin{proof}
Plugging the normal mode solution
\[
\left(  f\left(  x,v,t\right)  ,E\left(  x,t\right)  \right)  =e^{ik\left(
x-ct\right)  }\left(  f_{k}\left(  v\right)  ,E_{k}\right)
\]
into the linearized Vlasov-Poisson equation, we obtain the standard dispersion
relation
\begin{equation}
k^{2}-\int\frac{f_{0}^{\prime}\left(  v\right)  }{v-c}dv=0. \label{dispersion}%
\end{equation}
Linear instability with $x-$period $T$ corresponds to a solution of
(\ref{dispersion}) with $k=\frac{2\pi}{T}$ and $\operatorname{Im}c>0$. When
the condition (\ref{instability-T-i}) is satisfied, we have a neutral mode of
stability with $k_{0}=\left(  \frac{2\pi}{T_{i}}\right)  ^{2}$ and
$c_{0}=v_{i}$. Then local bifurcation of unstable modes near $\left(
k_{0},c_{0}\right)  $ can be shown, for example, by the arguments used in
\cite{lin-siam} for the shear flow instability. The bifurcation direction can
be seem from the following computation. Let $\left(  k,c\right)  $ be an
unstable mode near $\left(  k_{0},c_{0}\right)  .$Then
\[
k^{2}-k_{0}^{2}=\int\frac{f_{0}^{\prime}\left(  v\right)  }{v-c}dv-\int
\frac{f_{0}^{\prime}\left(  v\right)  }{v-v_{i}}dv=\left(  c-v_{i}\right)
\int\frac{f_{0}^{\prime}\left(  v\right)  }{\left(  v-v_{i}\right)  \left(
v-c\right)  }dv
\]
and by Plemelj formula when $\operatorname{Im}c\rightarrow0+,$%
\[
\frac{k^{2}-k_{0}^{2}}{c-v_{i}}=\int\frac{f_{0}^{\prime}\left(  v\right)
}{\left(  v-v_{i}\right)  \left(  v-c\right)  }dv\rightarrow P\int\frac
{f_{0}^{\prime}\left(  v\right)  }{\left(  v-v_{i}\right)  ^{2}}dv+i\pi
f_{0}^{\prime\prime}\left(  v_{i}\right)  ,
\]
where $P\int$ is the Cauchy principal value. So when $f_{0}^{\prime\prime
}\left(  v_{i}\right)  >0$ $\left(  <0\right)  $, we have to let $k^{2}%
<k_{0}^{2}\ \left(  k^{2}>k_{0}^{2}\right)  \ $to ensure $\operatorname{Im}%
c>0$. The linear stability when $T<T_{0}$ can be seem most easily from the
following Nyquist graph (see \cite{penrose}) in the complex plane
\begin{equation}
Z\left(  \xi+i0\right)  =\lim_{\eta\rightarrow0+}\int\frac{f_{0}^{\prime
}\left(  v\right)  }{v-\left(  \xi+i\eta\right)  }dv=P\int\frac{f_{0}^{\prime
}\left(  v\right)  }{v-\xi}dv+i\pi f_{0}^{\prime}\left(  \xi\right)  ,\ \xi
\in\mathbf{R}. \label{plemelj}%
\end{equation}
The unstable wave numbers consist of the part on the positive real axis
enclosed by the graph of $Z\left(  \xi+i0\right)  $. So the maximal unstable
wave number correspond to the right-most intersection point of the graph of
$Z\left(  \xi+i0\right)  $ with the positive real axis. Therefore if one of
the integral $\int\frac{f_{0}^{\prime}\left(  v\right)  }{v-v_{i}}dv$ is
positive, the maximal unstable wave number $k_{\max}$ is
\[
k_{\max}^{2}=\max_{v_{i}\in S}\int\frac{f_{0}^{\prime}\left(  v\right)
}{v-v_{i}}dv=\left(  \frac{2\pi}{T_{0}}\right)  ^{2},
\]
and all perturbations with $k>k_{\max}$ or equivalently $T<T_{0}$ are linearly
stable. For homogeneous states with all $\int\frac{f_{0}^{\prime}\left(
v\right)  }{v-v_{i}}dv$ to be non-positive, such as Maxwellian $e^{-\frac
{1}{2}v^{2}}$, perturbations of any period (wave number) are linearly stable
and thus $T_{0}=+\infty.$
\end{proof}

\begin{remark}
\label{rmk-penrose}1) The assumption $f_{0}\left(  v\right)  \in
W^{2,p}\left(  \mathbf{R}\right)  $ $\left(  p>1\right)  $ is used to ensure
that $f_{0}^{\prime}\left(  v\right)  $ is locally H\"{o}lder continuous and
thus the function $Z\left(  \xi+i0\right)  $ is well defined, continuous and
bounded. Lemma \ref{lemma-penrose} is still true for $f_{0}\in W^{1,p}$ and
$f_{0}^{\prime}$ locally H\"{o}lder continuous, particularly for $f_{0}\left(
v\right)  \in W^{s,p}\left(  \mathbf{R}\right)  $ $\left(  p>1,s>1+\frac{1}%
{p}\right)  $.

2) The local bifurcation of unstable modes near a neutral mode $\left(
k_{0},v_{i}\right)  $ can be extended globally in the following way. Let
$v_{i}$ be an extrema point of $f_{0}\left(  v\right)  ,\ $
\[
k_{0}^{2}=\left(  \frac{2\pi}{T_{i}}\right)  ^{2}=\int\frac{f_{0}^{\prime
}\left(  v\right)  }{v-v_{i}}dv>0.
\]
Suppose $f_{0}^{\prime\prime}\left(  v_{i}\right)  >0$, then the unstable
modes with $\operatorname{Im}c>0\ $exist when $k$ is slightly less than
$k_{0}$. This unstable mode can be continuated by decreasing $k$ as long as
the growth rate is not zero. This continuation process can only stop at
another neutral mode $\left(  k_{1},c_{1}\right)  $ with $k_{1}<k_{0}%
,\ c_{1}\in\mathbf{R,}.$ By (\ref{plemelj}), we must have
\[
f_{0}^{\prime}\left(  c_{1}\right)  =0,\ k_{1}^{2}=\left(  \frac{2\pi}{T_{1}%
}\right)  ^{2}=\int\frac{f_{0}^{\prime}\left(  v\right)  }{v-c_{1}}dv>0.
\]
For any wave number $k\in\left(  k_{1},k_{0}\right)  $, there exists an
unstable mode. Moreover, since the local bifurcation of unstable modes near
$k_{1}$ is only for slightly larger wave number, we must have $f_{0}%
^{\prime\prime}\left(  c_{1}\right)  <0$. Similarly, when $f_{0}^{\prime
\prime}\left(  v_{i}\right)  <0$, the unstable modes exist for wave numbers
$k\in\left(  k_{0},k_{2}\right)  $, where
\[
k_{2}^{2}=\left(  \frac{2\pi}{T_{2}}\right)  ^{2}=\int\frac{f_{0}^{\prime
}\left(  v\right)  }{v-c_{2}}dv,\ \text{with }f_{0}^{\prime}\left(
c_{2}\right)  =0,\ f_{0}^{\prime\prime}\left(  c_{2}\right)  >0\text{. }%
\]

From the above continuation argument, it is also easy to see the linear
stability for $k>k_{\max}$ without using the Nyquist graph. Suppose at some
$k^{\prime}>k_{\max}$ there exists an unstable mode. Then we can extend this
unstable mode for $k>k^{\prime}$ until it stops at a neutral mode$\ \left(
k^{\prime\prime},c^{\prime\prime}\right)  $ with
\[
\left(  k^{\prime\prime}\right)  ^{2}=\int\frac{f_{0}^{\prime}\left(
v\right)  }{v-c^{\prime\prime}}dv>0,\ f_{0}^{\prime}\left(  c^{\prime\prime
}\right)  =0.
\]
But $k^{\prime\prime}>k^{\prime}>k_{\max}$, this is a contradiction with the
definition of $k_{\max.}$ We also note that $k_{\max}$ must occur at a minimal
point of $f_{0}$, since the unstable modes only bifurcate for wave numbers
less than $k_{\max.}$

3) Finally, we point out that there could exist \textquotedblleft stability
gaps\textquotedblright\ of wave numbers in $\left(  0,k_{\max}\right)  $. By
our discussions above in 2), such stability gap must be of the form $\left(
\bar{k},\tilde{k}\right)  $ where
\[
\bar{k}^{2}=\int\frac{f_{0}^{\prime}\left(  v\right)  }{v-\bar{c}%
}dv>0,\ \tilde{k}^{2}=\int\frac{f_{0}^{\prime}\left(  v\right)  }{v-\tilde{c}%
}dv>0,
\]
and $\bar{c},\ \tilde{c}$ are minimum and maximum points of $f_{0}%
\ $respectively. From the Nyquist graph of $Z\left(  \xi+i0\right)  $, it is
easy to see that these stability gaps correspond to positive intervals in the
real axis not enclosed by the Nyquist curve.
\end{remark}

\begin{center}
{\Large Acknowledgement}
\end{center}

This work is supported partly by the NSF grants DMS-0908175 (Lin) and
DMS-0801319 (Zeng). We thank C\'{e}dric Villani for useful comments.


\begin{thebibliography}{99}                                                                                               %


\bibitem {adams}Adams, Robert A.; Fournier, John J. F., \textit{Sobolev
spaces.} Second edition. Pure and Applied Mathematics (Amsterdam), 140.
Elsevier/Academic Press, Amsterdam, 2003.

\bibitem {akheizer-et}Akhiezer, A., Akhiezer, I., Polovin, R., Sitenko, A.,
and Stepanov, K, \textit{Plasma electrodynamics, }Vol. I: Linear theory,
Pergamon Press, 1975 (Enlglish Edition). Translated by D. ter Haar.

\bibitem {amstrong-monto-67}Armstrong, T., Montgomery, D., \textit{Asymptotic
state of the two-stream instability, }J.Plasma. Physics, \textbf{1}, part 4,
425-433 (1967).

\bibitem {backus60}Backus, G. \textit{Linearized plasma oscillations in
arbitrary electron distributions.} J. Math. Phys. \textbf{1}, 178--191, (1960).

\bibitem {bertrand-et-1988}Gizzo, A., Izrar, B., Bertrand, P., Fijalkow, E.,
Feix, M. R., Shoucri, M., \textit{Stability of Bernstein-Greene-Kruskal plasma
equilibria. Numerical experiments over a long time, }Phys, Fluids, \textbf{31,
}no. 1, 72-82 (1988).

\bibitem {bgk}Bernstein, I., Greene, J., Kruskal, M., \textit{Exact nonlinear
plasma oscillations. }Phys. Rev. \textbf{108, }3, 546-550 (1957).

\bibitem {bernstein58-newcomb}Bernstein, Ira B. \textit{Waves in a Plasma in a
Magnetic Field}, Phys. Rev. \textbf{109}, 10 - 21 (1958).

\bibitem {bohn-gross}Bohm, D. and Gross, E. P. \textit{Theory of Plasma
Oscillations. A. Origin of Medium-Like Behavior,} Phys. Rev. \textbf{75}, 1851
- 1864 (1949).

\bibitem {brunetti-et-00}Brunetti, M., Califano, F. and Pegoraro, F.
\textit{Asymptotic evolution of nonlinear Landau damping,} Physical Review E
\textbf{62} 4109-4114 (2000).

\bibitem {buchanan-dorning95}Buchanan, M. L. and Dorning, J. J.,
\textit{Nonlinear electrostatic waves in collisionless plasmas,} Phys. Rev. E
\textbf{52}, 3015 - 3033 (1995).

\bibitem {buchanan-dorning-mBGK}Buchanan, M. L. and Dorning, J. J.,
\textit{Superposition of nonlinear plasma waves}, Phys. Rev. Lett.
\textbf{70}, 3732 - 3735 (1993).

\bibitem {case60}Case, K. \textit{Plasma oscillations.} Ann. Phys. \textbf{7},
349--364 (1959).

\bibitem {caglioti}Caglioti, E., and Maffei, C. \textit{Time asymptotics for
solutions of Vlasov--Poisson equation in a circle.} J. Statist. Phys.
\textbf{92}, 301--323 (1998).

\bibitem {driscoll-et-04}Danielson, J. R., Anderegg, F. and Driscoll, C. F.
\textit{Measurement of Landau Damping and the Evolution to a BGK Equilibrium,}
Physical Review Letters \textbf{92}, 245003-1-4 (2004).

\bibitem {degond}Degond, P. \textit{Spectral theory of the linearized
Vlasov--Poisson equation}, Trans. Amer. Math. Soc. \textbf{294}, 2, 435--453 (1986).

\bibitem {deimo-zweifel}Demeio, L. and Zweifel, P. F. \textit{Numerical
simulations of perturbed Vlasov equilibria}, Phys. Fluids B \textbf{2},
1252-1255 (1990).

\bibitem {demeio-holloway}Demeio, L. and Holloway, J. P., \textit{Numerical
simulations of BGK modes,} Journal of Plasma Physics, \textbf{46}, 63-84 (1991).

\bibitem {glassey-schaeffer}Glassey, R., and Schaeffer, J., \textit{On time
decay rates in Landau damping. }Comm. Partial Differential Equations
\textbf{20}, 647--676 (1995).

\bibitem {glassey-schaeffer-2}Glassey, R., and Schaeffer, J. \textit{Time
decay for solutions to the linearized Vlasov equation}, Transport Theory
Statist. Phys. \textbf{23}, 411--453 (1994).

\bibitem {gs2}Guo, Y. and Strauss, W., \textit{Instability of periodic BGK
equilibria, }Comm. Pure Appl. Math. Vol XLVIII, 861-894 (1995).

\bibitem {klimas-cooper83}Klimas, A. J. and Cooper, J. \textit{Vlasov--Maxwell
and Vlasov--Poisson equations as models of a one-dimensional electron plasma},
Phys. Fluids \textbf{26}, 478-480 (1983).

\bibitem {dorning-holloway}Holloway, J. P. and Dorning, J. J. \textit{Undamped
plasma waves}, Phys. Rev. A \textbf{44}, 3856-3868 (1991).

\bibitem {dorning-holloway2}Holloway, J. P. and Dorning, J. J.,
\textit{Nonlinear but small amplitude longitudinal plasma waves.} Modern
mathematical methods in transport theory (Blacksburg, VA, 1989), 155--179,
Oper. Theory Adv. Appl., 51, Birkh\"{a}user, Basel, 1991.

\bibitem {hormander-1}H\"{o}rmander, Lars \textit{The analysis of linear
partial differential operators. I. Distribution theory and Fourier analysis.}
Second edition. Grundlehren der Mathematischen Wissenschaften, 256.
Springer-Verlag, Berlin, 1990.

\bibitem {hwang}Hwang, J.-H., and V\'{e}lazquez, J. O\textit{n the existence
of exponentially decreasing solutions of the nonlinear landau damping
problem}, Preprint, 2008.

\bibitem {ishenko}Isichenko, M. B., \textit{Nonlinear Landau Damping in
Collisionless Plasma and Inviscid Fluid}, Phys. Rev. Lett. \textbf{78},
2369-2372 (1997).

\bibitem {krasovsky-et-04}Krasovsky, V. L., H. Matsumoto, and Y. Omura,
\textit{Electrostatic solitary waves as collective charges in a magnetospheric
plasma: Physical structure and properties of Bernstein--Greene--Kruskal (BGK)
solitons}, J. Geophys. Res., \textbf{108}(A3), 1117 (2004).

\bibitem {lancellotti-dorning}Lancellotti, C. and Dorning, J. J.
\textit{Time-asymptotic wave propagation in collisionless plasmas}, Physical
Review E \textbf{68} 026406 (2003).

\bibitem {landau}Landau, L. \textit{On the vibration of the electronic
plasma.} J. Phys. USSR \textbf{10}, 25 (1946).

\bibitem {lin-siam}Lin, Zhiwu, \textit{Instability of some ideal plane flows},
SIAM J. Math. Anal. \textbf{35}, 318-356 (2003).

\bibitem {lin01}Lin, Zhiwu, \textit{Instability of periodic BGK waves, }Math.
Res. Letts., \textbf{8}, 521-534 (2001).

\bibitem {lin-cpam}Lin, Zhiwu, \textit{Nonlinear instability of periodic waves
for Vlasov-Poisson system}, Comm. Pure. Appl. Math. \textbf{58}, 505-528 (2005).

\bibitem {lin-zeng-euler-invar}Lin, Zhiwu and Zeng, Chongchun,
\textit{Invariant manifolds of Euler equations}, preprint in preparation.

\bibitem {lin-zeng-euler-decay}Lin, Zhiwu and Zeng, Chongchun,
\textit{Dynamical structures near the Couette flow}, preprint, 2010.

\bibitem {lin-zeng-vp-invar}Lin, Zhiwu and Zeng, Chongchun, \textit{Invariant
manifolds of Vlasov-Poisson equations}, work in progress.

\bibitem {rosenbluth-et-98}Medvedev, M. V. Diamond, P. H., Rosenbluth, M. N.
and Shevchenko, V. I., \textit{Asymptotic Theory of Nonlinear Landau Damping
and Particle Trapping in Waves of Finite Amplitude}, Physical Review Letters
\textbf{81}, 5824 (1998).

\bibitem {manfredi97}Manfredi, Giovanni, \textit{Long-Time Behavior of
Nonlinear Landau Damping}, Physical Review Letters \textbf{79} 2815 (1997).

\bibitem {villani09}Mouhot, C., and Villani, C., \textit{On Landau damping},
Preprint, 2009.

\bibitem {muschietti-et-99}Muschietti, L. Ergun,R. E., Roth, I. and Carlson,
C. W. \textit{Phase-space electron holes along magnetic field lines}, Geophys.
Res. Lett. \textbf{26}, 1093-1096, (1999).

\bibitem {orr-1907}Orr, W. McF. \textit{Stability and instability of steady
motions of a perfect liquid}, Proc. Ir. Acad. Sect. A, Math Astron. Phys. Sci.
\textbf{27}, 9-66, (1907).

\bibitem {penrose}Penrose, O., \textit{Electrostatic instability of a
non-Maxwellian plasma}. Phys. Fluids \textbf{3}, 258--265 (1960).

\bibitem {oneil65}O'Neil, T. \textit{Collisionless damping of nonlinear plasma
oscillations}. Phys. Fluids \textbf{8}, 2255--2262 (1965).

\bibitem {stein-1970}Stein, Elias M., \textit{Singular integrals and
differentiability properties of functions}. Princeton Mathematical Series, No.
30 Princeton University Press, 1970.

\bibitem {strichartz}Strichartz, Robert S, \textit{Multipliers on fractional
Sobolev spaces}, J. Math. Mech. \textbf{16}, 1031--1060 (1967).

\bibitem {tartar-sobolev}Tartar, Luc, \textit{An introduction to Sobolev
spaces and interpolation spaces}, Lecture Notes of the Unione Matematica
Italiana, 3. Springer, Berlin; UMI, Bologna, 2007.

\bibitem {triebel}Triebel, Hans, \textit{Theory of function spaces,}
Monographs in Mathematics, 78, Birkh\"{a}user Verlag, Basel, 1983.

\bibitem {valentini-et-05}Valentini, F., Carbone, V., Veltri, P. and Mangeney,
A., \textit{Wave-Particle Interaction and Nonlinear Landau Damping in
Collisionless Electron Plasmas}, Transport Theory and Statistical Physics,
\textbf{34}, 89 - 101 (2005).

\bibitem {weizner63}Weitzner, Harold, \textit{Plasma oscillations and Landau
damping}. Phys. Fluids \textbf{6, }1123--1127 (1963).

\bibitem {van-kampen}van Kampen, N. \textit{On the theory of stationary waves
in plasma}. Physica \textbf{21}, 949--963 (1955).

\bibitem {zhou-et-98}Zhou, T., Guo, Y., and Shu, C.-W. \textit{Numerical study
on Landau damping}. Physica D \textbf{157}, 322--333 (2001).
\end{thebibliography}
\end{document}